\documentclass[]{amsart}

\usepackage{mathrsfs}
\usepackage{amsmath}
\usepackage{mathtools}
\usepackage{amssymb}
\usepackage{amscd}
\usepackage{appendix}
\usepackage{amsthm}
\usepackage{amsfonts}
\usepackage[all]{xy}
\usepackage{tikz-cd}
\usepackage{array}
\usepackage{graphicx}
\usepackage{hyperref}
\usepackage{longtable}
\usepackage[OT2,T1]{fontenc}
\DeclareSymbolFont{cyrletters}{OT2}{wncyr}{m}{n}
\DeclareMathSymbol{\Sha}{\mathalpha}{cyrletters}{"58}

\theoremstyle{plain}
\newtheorem{thm}{Theorem}[section]
\newtheorem{lem}[thm]{Lemma}
\newtheorem{prop}[thm]{Proposition}
\newtheorem{cor}[thm]{Corollary}

\theoremstyle{definition}
\newtheorem{defn}[thm]{Definition}
\newtheorem{ex}[thm]{Example}

\theoremstyle{remark}
\newtheorem{rem}[thm]{Remark}

\renewcommand{\AA}{\mathbb{A}}
\newcommand{\FF}{{\mathbb F}}

\newcommand{\EE}{{\mathbb E}}
\newcommand{\ZZ}{{\mathbb Z}}
\newcommand{\QQ}{{\mathbb Q}}

\newcommand{\RR}{{\mathbb R}}
\newcommand{\CC}{{\mathbb C}}

\newcommand{\GG}{{\mathbb G}}
\newcommand{\PP}{{\mathbb P}}

\newcommand{\YY}{{\mathbb Y}}

\newcommand{\afk}{\mathfrak{a}}

\newcommand{\bfk}{\mathfrak{b}}

\newcommand{\dfk}{\mathfrak{d}}
\newcommand{\mfk}{\mathfrak{m}}

\newcommand{\ufk}{\mathfrak{u}}

\newcommand{\Afk}{\mathfrak{A}}

\newcommand{\Cfk}{\mathfrak{C}}

\newcommand{\Hfk}{\mathfrak{H}}
\newcommand{\Ifk}{\mathfrak{I}}

\newcommand{\Mfk}{\mathfrak{M}}

\newcommand{\Zfk}{\mathfrak{Z}}

\newcommand{\Ecal}{\mathcal{E}}

\newcommand{\Lcal}{\mathcal{L}}

\newcommand{\Ocal}{\mathcal{O}}
\newcommand{\Pcal}{\mathcal{P}}
\newcommand{\Rcal}{\mathcal{R}}

\newcommand{\Ical}{\mathcal{I}}

\newcommand{\Acal}{\mathcal{A}}
\newcommand{\Bcal}{\mathcal{B}}

\newcommand{\Pscr}{\mathscr{P}}

\newcommand{\tr}{\operatorname{Tr}}

\newcommand{\Mat}{\operatorname{Mat}}

\newcommand{\ord}{\operatorname{ord}}

\newcommand{\GL}{\operatorname{GL}}

\newcommand{\PGL}{\operatorname{PGL}}

\newcommand{\gal}{\operatorname{Gal}}

\newcommand{\Pic}{\operatorname{Pic}}

\newcommand{\Frob}{\operatorname{Frob}}

\newcommand{\im}{\operatorname{Im}}

\newcommand{\End}{\operatorname{End}}

\newcommand{\re}{\operatorname{Re}}

\numberwithin{equation}{section}

\newcommand\subfrac[2]{\genfrac{}{}{0pt}{}{#1}{#2}}

\begin{document}

\title{On Kronecker terms over global function fields}
\author[Fu-Tsun Wei]{Fu-Tsun Wei}
\address{Department of Mathematics, National Tsing Hua Univeristy, Taiwan}
\email{ftwei@math.nthu.edu.tw}

\subjclass[2010]{11M36, 11G09, 11R58}
\keywords{Function field, Drinfeld period domain, Bruhat-Tits building, Kronecker limit formula, Drinfeld-Siegel unit, Mirabolic Eisenstein series, CM Drinfeld module, Taguchi height, Colmez-type fomula,  Special $L$-value}

\begin{abstract}
We establish a general Kronecker limit formula of arbitrary rank over global function fields with Drinfeld period domains playing the role of upper-half plane.  The Drinfeld-Siegel units come up as equal characteristic modular forms replacing the classical $\Delta$. This leads to analytic means of deriving a Colmez-type formula for \lq\lq stable Taguchi height\rq\rq\ of CM Drinfeld modules having arbitrary rank. A Lerch-Type formula for \lq\lq totally real\rq\rq\ function fields is also obtained, with the Heegner cycle on the Bruhat-Tits buildings intervene.  Also our limit formula is naturally applied to
the special values of both the Rankin-Selberg $L$-functions and the Godement-Jacquet $L$-functions associated to automorphic cuspidal representations over global function fields.
\end{abstract}

\maketitle

\section{Introduction}\label{Intro}

The celebrated first (resp.\ second) limit formula of Kronecker expresses the \lq\lq second term\rq\rq\ of the non-holomorphic Eisenstein series by the \lq\lq order\rq\rq\ of the modular discriminants (resp. modular units) at the archimedean place. This formula reveals very interesting stories in arithmetic geometry concealed inside the Euler-Kronecker constants of quadratic number fields (cf.\ Colmez \cite{Col}, Hecke \cite{Hec}) and \lq\lq non-central\rq\rq\ special $L$-values coming from classical modular forms (cf.\ Beilinson \cite{Bei3}).
The aim of the present paper is to take up the study of this phenomenon in the function field setting.
We first establish an \lq\lq adelic\rq\rq\ Kronecker limit formula of arbitrary rank in the mixed characteristic context, and then explore the arithmetic of various \lq\lq Kronecker terms\rq\rq\ over global function fields.


\subsection{Kronecker limit formula for arbitrary rank}\label{sec 0.1}

Let $k$ be a global function field with a finite constant field $\FF_q$. 
Fix a place $\infty$ of $k$, referred to the infinite place of $k$.
Let $k_\infty$ be the completion of $k$ at $\infty$ and $\CC_\infty$ the completion of a chosen algebraic closure of $k_\infty$.
Let $\Hfk^r$ be the Drinfeld period domain of rank $r$, 
which admits a \lq\lq M\"obius\rq\rq\ left action of $\GL_r(k_\infty)$.
Let $\AA$ (resp.\ $\AA^\infty$) be the (resp.\ finite) adele ring of $k$.
Put $\Hfk^r_\AA := \Hfk^r \times \GL_r(\AA^\infty)$.
Then $\GL_r(k)$ acts on $\Hfk^r_\AA$ diagonally from the left, 
and $\GL_r(\AA^\infty)$ acts on the second component of $\Hfk^r_\AA$ by right multiplication.
Let $S((\AA^\infty)^r)$ be the space of Schwartz functions (i.e.\ locally constant and compactly supported) on $(\AA^\infty)^r$.
For each $\varphi^\infty \in S((\AA^\infty)^r)$, we introduce the following \lq\lq non-holomorphic\rq\rq\ Eisenstein series on $\Hfk^r_\AA$:
for $z_\AA = \big((z_1:\cdots: z_{r-1}:1), g^\infty\big) \in \Hfk^r_\AA$ we set
$$
\EE(z_\AA,s;\varphi^\infty) := \sum_{0 \neq x = (x_1,...,x_r) \in k^r} \varphi^\infty(x g^\infty) \cdot 
\frac{\im(z_\AA)^s}{|x_1z_1+\cdots + x_{r-1} z_{r-1} + x_r|_\infty^{rs}}.
$$
Here $\im(z_\AA)$ is the \lq\lq total imaginary part\rq\rq\ of $z_\AA$ (cf.\ the equation (\ref{eqn 1.5}) and (\ref{eqn 2.1})),
and $|\cdot|_\infty$ is the normalized absolute value on $\CC_\infty$ (cf.\ Section~\ref{sec 1.1}).
One can (formally) check that
$$\EE(\gamma \cdot z_\AA,s; \varphi^\infty) = \EE(z_\AA,s; \varphi^\infty), \quad \forall
\gamma \in \GL_r(k).$$
For instance, let $A$ be the ring of integers of $k$ (with respect to $\infty$). 
Taking a rank $r$ projective $A$-module $Y \subset k^r$,
let $\mathbf{1}_{\widehat{Y}}$ be the characteristic function of $\widehat{Y}$, the closure of $Y$ in $(\AA^\infty)^r$.
Then we may write
$$
\EE\big((z,1),s; \mathbf{1}_{\widehat{Y}}\big)
= \sum_{0 \neq (c_1,...,c_{r-1},d) \in Y} \frac{\im(z)^s}{|c_1z_1+\cdots + c_{r-1}z_{r-1}+d|_\infty^{rs}} 
=: \EE^Y(z,s),
\quad \forall z \in \Hfk^r.
$$

The main theorem of this paper is presented in the following:

\begin{thm}\label{thm 0.1}
${}$
\begin{itemize}
\item[(1)] The Eisenstein series $\EE(z_\AA,s;\varphi^\infty)$ is an \lq\lq extension\rq\rq\ of the \lq\lq mirabolic\rq\rq\ Eisenstein series on $\GL_r(\AA)$ associated to $\varphi^\infty$ to $\Hfk^r_\AA$ through the building map.
\item[(2)]
Let $S((\AA^\infty)^r)_\ZZ$ be the subspace of $\ZZ$-valued Schwartz functions in $S((\AA^\infty)^r)$.
For $z_\AA \in \Hfk^r_\AA$ and $\varphi^\infty \in S((\AA^\infty)^r)_\ZZ$, 
we have $\EE(z_\AA,0;\varphi^\infty) = - \varphi^\infty(0)$ and
$$
\frac{\partial}{\partial s} \EE(z_\AA,s;\varphi^\infty)\bigg|_{s=0}=
-\varphi^\infty(0) \cdot \ln \im(z_\AA)- \frac{r}{q^{r\deg \infty} -1} \cdot \ln |\ufk(z_\AA;\varphi^\infty)|_\infty,
$$
where $\ufk(\cdot;\varphi^\infty)$ is the Drinfeld-Siegel units on $\Hfk^r_\AA$ associated to $\varphi^\infty$$($which is an \lq\lq equal characteristic\rq\rq\ $\CC_\infty$-valued modular form on $\Hfk^r_\AA$, cf.\ {\rm Definition~\ref{defn 3.2}}$)$.
\end{itemize}
\end{thm}

In order to state Theorem~\ref{thm 0.1} (1) more concretely, we first recall basic properties of mirabolic Eisenstein series on $\GL_r(\AA)$.
Given a unitary Hecke character $\chi$ on $k^\times \backslash \AA^\times$ and $\varphi \in S(\AA^r)$,
the mirabolic Eisenstein series on $\GL_r(\AA)$ associated to $\chi$ and $\varphi$ can be expressed as follows (cf.\ {\it Remark}~\ref{rem 2.3}):
$$
\Ecal(g,s;\chi,\varphi) =|\det g|_\AA^s \int_{k^\times \backslash \AA^\times}
\left(\sum_{0 \neq x \in k^r}\varphi(a^{-1} x g)\right) \chi(a) |a|_\AA^{-rs} d^\times a, \ \  \forall g \in \GL_r(\AA), \ \ \re(s)>1.
$$
Here the Haar measure $d^\times a$ is normalized so that the maximal compact subgroup $O_\AA^\times$ of $\AA^\times$ has volume one. 
For $\varphi^\infty \in S((\AA^\infty)^r)$, define the \lq\lq finite\rq\rq\ mirabolic Eisenstein series associated to $\varphi^\infty$ by:
$$
\Ecal^\infty(g,s;\varphi^\infty):= \frac{q-1}{\#\Pic(A)} \cdot
\sum_{\chi \in \widehat{\Ical}^\infty}
\Ecal(g,s;\chi,\varphi^\infty \otimes \mathbf{1}_{O_\infty^r}), \quad \forall g \in \GL_r(\AA),
$$
where $\widehat{\Ical}^\infty$ is the Pontryagin dual group of the finite idele class group  $\Ical^\infty:= k^\times \backslash \AA^{\infty, \times}$, and $\mathbf{1}_{O_\infty^r}$ is the characteristic function of $O_\infty^r \subset k_\infty^r$ where $O_\infty$ is the maximal compact subring of $k_\infty$.
Note that the normalized factor $(q-1)\cdot \#\Pic(A)^{-1}$ comes from the volume of $\Ical^\infty$ with respect to the chosen Haar measure (cf.\ the equation~(\ref{eqn Ivol})).
On the other hand, let $\Bcal^r(\RR)$ be the realization of the Bruhat-Tits building $\Bcal^r$ of $\PGL_r(k_\infty)$, and denote by $\lambda: \Hfk^r \rightarrow \Bcal^r(\RR)$ the building map (cf.\ Definition~\ref{defn 1.6}).
Put $\Bcal^r_\AA(\RR):= \Bcal^r(\RR) \times \GL_r(\AA^\infty)$,
and extend $\lambda$ naturally to a 
$\GL_r(k) \times \GL_r(\AA^\infty)$-biequivariant map $\lambda_\AA$ from $\Hfk_\AA^r$ to $\Bcal^r_\AA(\RR)$.
We observe that $\EE(\cdot,s;\varphi^\infty)$ actually factors through $\lambda_\AA$
(cf.\ Proposition~\ref{prop 2.7}), and satisfies
(cf.\ the equation~(\ref{eqn 2.2})):
\begin{eqnarray}\label{eqn 0.1}
\EE(z_\AA, s;\varphi^\infty) &=& (1-q^{-(r\deg \infty) s}) \cdot \Ecal^\infty(g_{z_\AA},s;\varphi^\infty)
\end{eqnarray}
for every $z_\AA \in \Hfk^r_\AA$ with $\lambda(z_\AA) = [g_{z_\AA}] \in \Bcal^r_\AA(\ZZ) \cong \GL_r(\AA)/k_\infty^\times \GL_r(O_\infty)$.
In other words,
the equality \eqref{eqn 0.1} links our non-holomorphic Eisenstein series with automorphic Eisenstein series on $\GL_r(\AA)$ in a conceptual way.

\begin{rem}
${}$
\begin{itemize}
\item[(1)]
From the analytic behavior of mirabolic Eisenstein series (recalled in Theorem \ref{thm 2.2}), the equality \eqref{eqn 0.1} says that
$\EE(z_\AA,s;\varphi^\infty)$ converges absolutely for $\re(s)>1$, 
and has a meromorphic continuation to the whole complex $s$-plane satisfying a \lq\lq weak\rq\rq\ functional equation with the symmetry between values at $s$ and $1-s$ (cf.\ Proposition \ref{prop 2.7} and {\it Remark} \ref{rem 2.8}).
\item[(2)]
Let $\Delta^{Y}(z)$ be the Drinfeld-Gekeler discriminant function associated to $Y$ on $\Hfk^r$,
which is a Drinfeld modular form of weight $q^{r\deg \infty}-1$ on $\Hfk^r$ (cf.\ Section~\ref{sec 1.4}).
Then for every $z \in \Hfk^r$ we have $\ufk((z,1), \mathbf{1}_{\widehat{Y}}) = \Delta^{Y}(z)$.
Accordingly, the above theorem leads to a precise function field analogue of the Kronecker (first) limit formula:
$$
\EE^{Y}(z,0) = -1 \quad \text{ and } \quad
\frac{\partial}{\partial s}\EE^{Y}(z,s)\bigg|_{s=0} = - \ln\left(\im(z) \cdot |\Delta^{Y}(z)|_\infty^{\frac{r}{q^{r\deg\infty}-1}}\right),
\quad \forall z \in \Hfk^r.
$$
\item[(3)] In the number field case, there seems no ideal candidates of the symmetric space for $\GL_r$ when $r>2$. 
Contrarily, the period domain of Drinfeld, together with the building map to the Bruhat-Tits building, provide a perfect choice of the symmetric space for $\GL_r$ over global function fields.
Our treatment of higher rank Eisenstein series is thereby well-developed.
This framework over global function fields is so natural that we can get hold of the Kronecker terms, just as the classical $\GL_2$ case.
In other words, Kronecker terms can be understood completely even for $r>2$ in the function field setting.
\end{itemize}
\end{rem}

\subsubsection{Outline the proof of {\rm Theorem~\ref{thm 0.1} (2)}}
Our approach is completely different from the classical one. 
Note that
$\EE(z_\AA,s;\varphi^\infty)$ is $\CC$-valued, and $\ufk(z_\AA;\varphi^\infty)$ lies in the positive characteristic world.
The building map $\lambda_\AA : \Hfk^r_\AA \rightarrow \Bcal^r_\AA(\RR)$ is the main bridge in the mix characteristic scene.
In particular, we pin point that the building map strips out all the transcendentals in the Drinfeld period domain leaving only an elegant discrete structure. 

Another key ingredient of our proof is an explicit description of the meromorphic continuation of $\EE^{Y}(z,s)$ (cf.\ Lemma~\ref{lem 2.10}). 
This enables us to derive a Stieltjes-type formula of all the Taylor coefficients of $\EE^{Y}(z,s)$ at $s = 0$ (cf.\ Corollary~\ref{cor 2.11}). 

\begin{rem} 
In \cite{Ge2}, Gekeler first connects an \lq\lq improper\rq\rq\  Eisenstein series on $\GL_2(k_\infty)$ over rational function fields with Drinfeld discriminant functions (cf.\ \cite{Ge2}).
His result is then generalized by P\'al \cite[Section 4]{Pal} to a special family of modular units (in the rank $2$ case), 
and also by Kondo \cite{Kon}, Kondo-Yasuda \cite[Section 3.5]{K-Y} who target at Jacobi-type Eisenstein series with arbitrary rank.
One purpose of this paper is to give a complete account of this phenomenon in adelic settings from the point of view of automorphic representation theory.

Comparing with \cite{Wei2} on the rank $2$ case, there are many new terminologies and approaches when dealing with higher ranks. For instance:
\begin{itemize}
\item We introduce a concept of the \lq\lq total imaginary part\rq\rq\ of $z \in \Hfk^r$ for arbitrary rank (cf.\ Section \ref{sec 1.5.1}). This new notion is in fact very essential in the whole paper.
\item The connection of our non-holomorphic Eisenstein series and automorphic (mirabolic) Eisenstein series (in Theorem~\ref{thm 0.1} (1)) is much more conceptual.
\item The adelic formulation in Theorem~\ref{thm 0.1}, together with the new input of Schwartz functions $\varphi^\infty$, enables us to utilize the well-developed tools of the automorphic representation theory in our study on special $L$-values.
\end{itemize}
Different from Kondo-Yasuda \cite{K-Y}, our Eisenstein series totally reflect the whole combinatorial structures of the Bruhat-Tits building, and provide information not just for vertices.
\end{rem}


\subsubsection{Lerch-type formula}

Applying Theorem~\ref{thm 0.1}, we obtain a Lerch-type formula of mirabolic Eisenstein series on $\GL_r(\AA)$ over function fields. 
Indeed, 
let $\chi$ be a unitary Hecke character on $k^\times \backslash \AA^\times$ with $\chi(k_\infty^\times) = 1$.
Given $\varphi^\infty \in S((\AA^\infty)^r)_\ZZ$, 
suppose (for simplicity) that either $\chi$ is non-trivial or $\varphi^\infty$ vanishes at $0$. 
For $g \in \GL_r(\AA)$, 
we set
$$
\eta_\chi(g;\varphi^\infty) :=  \int_{k^\times \backslash \AA^{\infty,\times}}
\overline{\chi(a^\infty)} \cdot \log_{q^{\deg \infty}} |\ufk(z_\AA a^\infty;\varphi^\infty)|_\infty d^\times a^\infty,
$$
where $z_\AA \in \Hfk^r_\AA$ is any point satisfying $\lambda_\AA(z_\AA) = [g] \in \Bcal^r_\AA(\ZZ)$, 
and the Haar measure $d^\times a^\infty$ is normalized so that the maximal compact subgroup $O_{\AA^\infty}^\times$ of $\AA^{\infty,\times}$ has volume one. 
Then (cf.\ Corollary~\ref{cor 3.7}):

\begin{cor}\label{cor 0.2}
Suppose either $\chi$ is non-trivial or $\varphi^\infty$ vanishes at $0$.
For every $g \in \GL_r(\AA)$ we have 
$$
\Ecal(g,0;\chi, \varphi^\infty\otimes \mathbf{1}_{O_\infty})
= \frac{1}{1-q^{r\deg \infty}} \cdot \eta_{\chi}(g;\varphi^\infty).
$$
\end{cor}

\begin{rem}\label{rem 0.3}
${}$
\begin{itemize}
\item[(1)]
In Corollary~\ref{cor 3.7}, we also include the case when $\chi$ is the trivial character and $\varphi^\infty(0)$ is non-vanishing.
Then the corresponding mirabolic Eisenstein series may have a simple pole at $s=0$. 
Our formula in Corollary~\ref{cor 3.7} actually describes the first two terms of its Laurant expansion arround $s=0$.
\item[(2)]
Theorem \ref{thm 0.1} enables us to express the Kronecker term of zeta functions over \lq\lq totally real\rq\rq\ function fields as integrations of an \lq\lq eta function\rq\rq\ along the corresponding \lq\lq Heegner cycles\rq\rq\ in $\Bcal^r_\AA(\ZZ)$, and leads us to a Lerch-type formula for the Dirichlet $L$-functions associated to \lq\lq ring class characters\rq\rq\ over totally real function fields (cf.\ Theorem~\ref{thm E-K})\end{itemize}
\end{rem}

\subsection{Colmez-type formula}\label{sec 0.2}

Colmez \cite{Col} proposes a conjectural formula expressing explicitly the stable Faltings height of CM abelian varieties over number fields in terms of a precise linear combination of logarithmic derivatives of Artin $L$-functions. This formula provides a very interesting arithmetic interpretation of the geometric invariant in question, took a stand near the center of arithmetic geometry ever since its discovery (cf.\ \cite{Col}, \cite{Obu}, \cite{Yan}, \cite{Y-Z}, \cite{AGHP}, and \cite{BCCGM}). 
Here we apply Theorem~\ref{thm 0.1} to derive an analogue of the Colmez formula for the stable \lq\lq Taguchi height\rq\rq\ of Drinfeld modules.

Let $\rho$ be a Drinfeld $A$-module of rank $r$ over a finite extension $F$ of $k$ in $\CC_\infty$. 
The endomorphism ring $\End_A(\rho/\bar{k})$ can be identified with an $A$-order $\Ocal$ of an \lq\lq imaginary\rq\rq\ field $K/k$ with the degree $[K:k]$ dividing $r$.
We call $\rho$ {\it CM} if $[K:k] = r$.

In \cite{Tag}, Taguchi introduced a \lq\lq metrized line bundle\rq\rq\ $\Lcal_\rho$ associated to $\rho$, and define the height $h_{\text{Tag}}(\rho/F)$ of $\rho$ be the \lq\lq degree\rq\rq\ of $\Lcal_\rho$ (cf.\ Section \ref{sec 4.1} for an alternative definition). The stable Taguchi height of $\rho$ is defined by:
$$h_{\text{Tag}}^{\text{st}}(\rho) := \lim_{F:  \atop [F:k]<\infty} h_{\text{Tag}}(\rho/F),$$
where the limit always exists from the fact that every Drinfeld $A$-module has potential stable reduction (cf.\ \cite[Proposition 7.2]{Ha1}).

Suppose now that $\rho$ is a CM Drinfeld $A$-module.
Let $\Lambda_\rho \subset \CC_\infty$ be the $A$-lattice associated to $\rho$.
Viewing $\Lambda_\rho$ as an $\Ocal$-module, 
we take an ideal $\Ifk$ of $\Ocal$ so that $\Lambda_\rho$ and $\Ifk$ have the same genus.
Let $\zeta_\Ifk(s)$ be the zeta function associated to $\Ifk$:
$$
\zeta_\Ifk(s) := \sum_{\text{\rm invertible fractional ideal $I$ of $\Ocal$} \atop I \subset \Ifk}
\frac{1}{N(I)^s},
$$
where $N(I) := \#(\Ifk/I)$.
Note that $\zeta_\Ifk(s)$ only depends on the genus of $\Ifk$ (as an $\Ocal$-module).
Our Colmez-type formula is stated as follows:

\begin{thm}\label{thm 0.4}
Let $h^{\text{\rm st}}_{\text{\rm Tag}}(\rho)$ be the stable Taguchi height of the CM Drinfeld $A$-module $\rho$.
We have:
$$
h^{\text{\rm st}}_{\text{\rm Tag}}(\rho) =
-\ln D_A(\Ocal) - \frac{1}{r}\cdot \frac{\zeta_\Ifk'(0)}{\zeta_\Ifk(0)}.
$$
Here $D_A(\Ocal)$ is the \lq\lq lattice discriminant\rq\rq\ of $\Ocal$ (as an $A$-lattice in $\CC_\infty$, cf.\ Remark {\rm \ref{rem 1.10}}).
\end{thm}

\begin{rem}\label{rem 0.5}
${}$
\begin{itemize}
\item[(1)] 
This formula also provides a geometric interpretation for $\zeta_\Ifk'(0)/ \zeta_\Ifk(0)$.
\item[(2)]
In the proof of Theorem \ref{thm 0.4}, we need to extend Hayes' CM theory of Drinfeld modules to the case when the corresponding lattice has arbitrary genus. The details are attached in Appendix~\ref{sec A} for the sake of completeness.
\item[(3)]
when the \lq\lq CM\rq\rq\ function field $K$ is separable over $k$ and tamely ramified at $\infty$, we get (cf.\ {\it Remark}~\ref{rem 4.4} (2))
$$
D_A(\Ocal) = \Vert \dfk(\Ocal/A)\Vert^{\frac{1}{2r}},
$$
where $\dfk = \dfk(\Ocal/A) \subset A$ is the discriminant ideal of $\Ocal$ over $A$ and $\Vert \dfk \Vert := \#(A/\dfk)$.
\item[(4)] 
Hartl-Singh \cite{H-S} investigate the Colmez conjecture for $A$-motives, and prove a product formula for the Carlitz module. Our formula in Theorem~\ref{thm 0.4} in the case of the Carlitz module coincides with their result (cf.\ \cite[Example 1.5]{H-S}), but the analytic approach is completely different from theirs.
\item[(5)]
Let $O_K$ be the integral closure of $A$ in $K$.
From the Ihara estimate of the Euler-Kronecker constant of the zeta function $\zeta_{O_K}(s)$ in \cite[(0.6) and (1.2)]{Iha}, an asymptotic formula of the Taguchi height of Drinfeld modules with CM by $O_K$ is worked out in Section \ref{sec 4.2.1}.
\end{itemize}
\end{rem}

\subsection{Special values of automorphic $L$-functions}\label{sec 0.3}

In the theory of automorphic representation, mirabolic Eisenstein series naturally occur as the kernel functions in the integral representations of automorphic $L$-functions (cf.\ \cite{GPSR} and \cite{S-S}).
From Theorem~\ref{thm 0.1} (1), we may naturally apply our Kronecker limit formula to
special values of Rankin-Selberg $L$-functions and Godement-Jacquet $L$-functions over global function fields.

\subsubsection{Rankin-Selberg $L$-functions}\label{sec 0.3.1}
Let $\Pi$ and $\Pi'$ be two automorphic cuspidal representations of $\GL_r(\AA)$ with unitary central characters $\omega$ and $\omega'$, respectively. 
Let $\chi := (\omega \cdot \omega')^{-1}$. 
Suppose $\chi\big|_{k_\infty^\times} = 1$.
We introduce the follwing multi-linear functional $\Pcal^{\text{RS}} : \Pi \times \Pi' \times S((\AA^\infty)^r)_{\ZZ} \rightarrow \CC$ (where $S((\AA^\infty)^r)_{\ZZ}$ consists of $\ZZ$-valued Schwartz functions in $S((\AA^\infty)^r)$):
$$
\Pcal^{\text{RS}}(f,f',\varphi^\infty) := \frac{1}{1-q^{r\deg \infty}} \cdot
\int_{\AA^\times \GL_r(k)\backslash \GL_r(\AA)} f(g) f'(g) \eta_\chi(g;\varphi^\infty) dg.
$$
Here $\eta_\chi(\cdot;\varphi^\infty)$ is defined in the above of Corollary~\ref{cor 0.2}, and $dg$ is chosen to be the Tamagawa measure 
(i.e.\ $\text{vol}(\AA^\times \GL_r(k)\backslash \GL_r(\AA), dg) = 2$, cf.\ \cite[Theorem 3.3.1]{We2}).
On the other hand, 
we have another multi-linear functional $\Pscr^{\text{RS}}$ on $\Pi \times \Pi' \times S((\AA^\infty)^r)_\ZZ$ coming from the product of \lq\lq local integrals\rq\rq\ (cf.\ the equality~(\ref{eqn 5.1})).
Let $L(s,\Pi\times \Pi')$ be the Rankin-Selberg $L$-function associated to $\Pi$ and $\Pi'$
(following the definition in \cite[Lecture 4, p.\ 137]{S-S}).
The Lerch-type formula in Corollary~\ref{cor 0.2} results in:

\begin{thm}\label{thm 0.7}
Let $\Pi$ and $\Pi'$ be two automorphic cuspidal representations of $\GL_r(\AA)$ with unitary central characters $\omega$ and $\omega'$, respectively.
Suppose $\Pi'$ is not isomorphic to the contragredient representation of $\Pi$ and $(\omega \cdot \omega')\big|_{k_\infty^\times} = 1$.
Then
$$\Pcal^{\text{\rm RS}} = L(0,\Pi\times \Pi') \cdot \Pscr^{\text{\rm RS}}.$$
\end{thm}

\subsubsection{Godement-Jacquet $L$-functions}
Let $\Pi$ be an automorphic cuspidal representation of $\GL_r(\AA)$ with unitary central character denoted by $\omega$. 
For $f_1,f_2 \in \Pi$ and $\varPhi \in S(\Mat_r(\AA))$, the {\it Godement-Jacquet $L$-function associated to $f_1,f_2$ and $\varPhi$} is defined by (cf.\ \cite[p.\ 12]{GPSR}):
$$
L^{\text{GJ}}(s;f_1,f_2,\Phi) := \int_{\GL_r(\AA)} \varPhi(g) \cdot \langle \Pi(g) f_1,f_2 \rangle_{\text{Pet}} \cdot |\det g|_\AA^s dg, \quad \re(s)>r.
$$
Here $\langle \cdot,\cdot\rangle_{\text{Pet}}$ is the Petersson inner product on $\Pi$.
We choose the Haar measure $dg$ on $\GL_r(\AA)$ to be induced from the Tamagawa measure on $\AA^\times \backslash \GL_r(\AA)$ and the measure $d^\times a$ on $\AA^\times$ with $\text{vol}(O_\AA^\times, d^\times a) = 1$.
It is known that this $L$-function has analytic continuation to the whole complex $s$-plane and a functional equation with the symmetry between values at $s$ and $r-s$.
On the other hand, identifying $\Mat_r(k)$ with $k^{r^2}$ suitably (as in the identity~(\ref{eqn 5.3})), 
we have a group homomorphism $\iota: \GL_r^2 = \GL_r \times \GL_r \rightarrow \GL_{r^2}$ via the left and right multiplications:
$$
(g_1,g_2) \cdot X :={}^t g_2 X g_1, \quad
\forall g_1,g_2 \in \GL_r \text{ and } X \in \Mat_r.
$$

Suppose $\omega\big|_{k_\infty^\times} = 1$.
Define the multi-linear functional
$\Pcal^{\text{GJ}} : \Pi \times \Pi \times S(\Mat_r(\AA^\infty))_{\ZZ} \rightarrow \CC$ by:
$$
\Pcal^{\text{GJ}}(f_1, f_2, \varPhi^\infty) := \frac{1}{1-q^{r^2\deg \infty}} \cdot
\iint_{(\AA^\times \GL_r(k)\backslash \GL_r(\AA))^2} f_1(g_1) \cdot \eta_{\omega^{-1}}\big(\iota(g_1,g_2);\varPhi^\infty\big) \cdot \overline{f_2(g_2)} dg_1 dg_2.
$$
Here $\eta_{\omega^{-1}}(\cdot; \varPhi^\infty)$ is the function on $\GL_{r^2}(\AA)$ coming from Drinfeld-Siegel units on $\Hfk_\AA^{r^2}$ (cf.\ Corollary~\ref{cor 0.2}).
Thus for $f_1,f_2 \in \Pi$ and $\varPhi^\infty \in S(\Mat_r(\AA^\infty))_{\ZZ}$, 
from the \lq\lq doubling method\rq\rq\ of Piatetski-Shapiro and Rallis (cf.\ Proposition~\ref{prop 5.2}) and Corollary~\ref{cor 0.2}, we immediately get
\begin{eqnarray}
L^{\text{GJ}}(0; f_1,f_2, \varPhi^\infty \otimes \mathbf{1}_{\Mat_r(O_\infty)})
&=& \Pcal^{\text{GJ}}(f_1, f_2, \varPhi^\infty). 
\end{eqnarray}
Let $L(s,\Pi)$ be the automorphic $L$-function associated to $\Pi$. It is a fact that (cf.\ \cite[Theorem 3.3 (2)]{G-J})
$$
\frac{L^{\text{GJ}}(s;f_1,f_2,\varPhi)}{L(s-(r-1)/2,\Pi)} \in \CC[q^s,q^{-s}], \quad \forall f_1,f_2 \in \Pi \text{ and } \varPhi \in S(\Mat_r(\AA)).
$$
Using \lq\lq local zeta integrals\rq\rq\ at each place of $k$, we obtain anther multi-linear functional 
$\Pscr^{\text{GJ}}$ on $\Pi \times \Pi \times S(\Mat_r(\AA^\infty))_{\ZZ}$
(cf.\ the equality~(\ref{eqn 5.4})).
Then we arrive at:

\begin{thm}\label{thm 0.8}
Let $\Pi$ be an automorphic cuspidal representation of $\GL_r(\AA)$ with a unitary central character $\omega$.
Suppose $\omega\big|_{k_\infty^\times} = 1$.
The following equality holds:
$$\Pcal^{\text{\rm GJ}} = L(\frac{1-r}{2},\Pi) \cdot \Pscr^{\text{\rm GJ}}.$$
\end{thm}

To gain an in-depth understanding of the special $L$-values $L(0,\Pi\times \Pi')$ and $L((1-r)/2,\Pi)$, Theorem \ref{thm 0.7} and \ref{thm 0.8} reduces the technicalities to local calculations.
More precisely, taking suitable test functions at each place, it is possible to determine the corresponding values of $\Pscr^{\text{\rm RS}}$ and $\Pscr^{\text{\rm GJ}}$ in concrete terms, which gives rise to explicit formulas for the special values $L(0,\Pi\times \Pi')$ and $L((1-r)/2,\Pi)$. This is actually a key ingredient in the study of Beilinson's regulators for Drinfeld modular varieties, which will be explored in a subsequent paper.

\begin{rem}
Kondo-Yasuda \cite{K-Y} consider \lq\lq partial $L$-functions $L^{I,J}(s,\Pi)$\rq\rq, and connect their special \lq\lq derivatives\rq\rq\ with an \lq\lq Euler system\rq\rq\ coming from rank $r$ Drinfeld-Siegel units.
Contrarily, Theorem \ref{thm 0.8} illustrates a complete different phenomenon. Our formula states for the complete $L$-function $L(s,\Pi)$, and expresses the special $L$-value in question by an \lq\lq inner product\rq\rq\ with rank $r^2$ Drinfeld-Siegel units. 
We may expect, after further study, there is a natural link between Drinfeld-Siegel units having rank $r$ and $r^2$ hidden behind the special $L$-value in question.
\end{rem}

\subsection{The content of the paper}\label{sec 0.4}
We fix basic notations used throughout this paper in Section~\ref{sec 1.1}. The analytic theory of Drinfeld modules and Drinfeld period domain are reviewed in Section~\ref{sec 1.2} and \ref{sec 1.3}, respectively. Drinfeld-Gekeler discriminant functions are introduced in Section~\ref{sec 1.4}. In Section~\ref{sec 1.5}, we discuss the needed properties of the building map from $\Hfk^r$ to the Bruhat-Tits building associated to $\PGL_r(k_\infty)$, and introduce \lq\lq imaginary parts\rq\rq\ of $z \in \Hfk^r$. 

Section~\ref{sec 2} is to understand the analytic behavior of our non-holomorphic Eisenstein series. We first recall the well-known analytic properties of mirabolic Eisenstein series on $\GL_r(\AA)$ in Section~\ref{sec 2.1}, and establish a natural identity (via the building map) between these automorphic Eisenstein series with our non-holomorphic Eisenstein series in Section~\ref{sec 2.2} and \ref{sec 2.3}. In Section~\ref{sec 2.4}, we present a Stieltjes-type formula of $\EE^Y(z,s)$ from an explicit description of its meromorphic continuation.  

In Section~\ref{sec 3}. We first introduce the Drinfeld-Siegel units on $\Hfk^r_\AA$ in Section~\ref{sec 3.1}, and prove our Kronecker limit formula in Section~\ref{sec 3.2}. In Section~\ref{sec 3.3}, we derive a Lerch-type formula for our non-holomorphic Eisenstein series on $\Hfk^r_\AA$ and mirabolic Eisenstein series on $\GL_r(\AA)$. 

In Section~\ref{sec 4}, we apply the Kronecker limit formula to prove a Colmez-type formula for the Taguchi height of CM Drinfeld modules. The definition of the Taguchi height of a Drinfeld module is recalled in Section~\ref{sec 4.1}, and our Colmez-type formula is derived in Section~\ref{sec 4.2}, together with a short discussion on the asymptotic behavior of the CM Taguchi height. 

In Section~\ref{sec Z-E}, we expresses the Euler-Kronecker constants of zeta functions over \rq\rq totally real\rq\rq\ (with respect to $\infty$) function fields as integrations of finite mirabolic Eisenstein series along the corresponding \lq\lq Heegner cycles\rq\rq\ in $\Bcal^r_\AA(\ZZ)$. Consequently, we obtain a Lerch-type formula of the Dirichlet $L$-functions associated to ring class characters.

In Section~\ref{sec 5}, we study applications of our Kronecker limit formula to special values of automorphic $L$-functions. Theorem~\ref{thm 0.7} and \ref{thm 0.8} are demonstrated in Section~\ref{sec 5.1} and Section~\ref{sec 5.2}, respectively. 

Finally, we extend Hayes' CM theory of Drinfeld modules to the case of arbitrary genus in Appendix~\ref{sec A}.

\subsection*{Acknowledgements}
The author is very grateful to Jing Yu and Chieh-Yu Chang for their steady interest, encouragements, and very useful suggestions. He would also like to thank Mihran Papikian for helpful discussions. The author is deeply appreciate the anonymous referee for very careful reading and many useful comments to improve the manuscript. This work is supported by the Ministry of Science and Technology (grant no.\ 105-2115-M-007-018-MY2 and 107-2628-M-007-004-MY4) and the National Center for Theoretical Sciences.

\section{Preliminaries}\label{sec 1}

\subsection{Basic settings}\label{sec 1.1}

Let $\FF_q$ be the finite field with $q$ elements.
Let $k$ be a global function field with constant field $\mathbb{F}_q$, 
i.e.\ $k$ is a finitely generated field extension over $\mathbb{F}_q$ with transcendence degree one and $\mathbb{F}_q$ is algebraically closed in $k$. 
For each place $v$ of $k$, the completion of $k$ at $v$ is denoted by $k_v$, and $O_v$ denotes the valuation ring in $k_v$. 
Choosing a uniformizer $\pi_v$ in $O_v$ once and for all, we set $\FF_v := O_v/\pi_vO_v$ and $q_v := \#(\FF_v)$. Let $\deg v := [\FF_v: \FF_q]$, called the degree of $v$.
The absolute value on $k_v$ is normalized to:
$$|\alpha_v|_v:= q_v^{-\ord_v(\alpha_v)} = q^{-\deg v\ord_v(\alpha_v)}, \quad \forall \alpha_v \in k_v.$$
Let $\AA:= \prod_v'k_v$, the adele ring of $k$ and $O_{\AA}:= \prod_v O_v$, the maximal compact subring of $\AA$.
We embed $k$ (resp.\ $k^{\times}$) into $\AA$ (resp.\ $\AA^{\times}$) diagonally.
For each element $\alpha = (\alpha_v)_v$ in the idele group $\AA^{\times}$, the norm $|\alpha|_{\AA}$ is defined to be
$$|\alpha|_{\AA}:= \prod_v |\alpha_v|_v.$$

Throughout this paper, we fix a non-trivial additive character $\psi = \otimes_v \psi_v :\AA \rightarrow \CC^{\times}$ which is trivial on $k$
(here $\psi_v:= \psi|_{k_v}$).
For each place $v$ of $k$, let $\delta_v$ be the \lq\lq conductor of $\psi$ at $v$,\rq\rq\ i.e.\ the maximal integer $r$ so that $\pi_v^{-r}O_v$ is contained in the kernel of $\psi_v$ (cf.\ \cite[Def.\ 4 in Chap.\ II \S 5]{We1}).
It is known that (cf.\ \cite[Cor.\ 1 of Theorem 2 in Chap.\ VI]{We1}) $\sum_v \delta_v \deg v = 2g_k-2$,
where $g_k$ is the genus of $k$. We call $\delta = (\pi_v^{\delta_v})_v \in \AA^{\times}$ a \textit{differential idele of $k$ associated to $\psi$}.\\

Fix a place $\infty$ of $k$, regarded as the place at infinity. 
We set $\AA^\infty:= \prod_{v \neq \infty}'k_v$, called the finite adele ring of $k$, and $O_{\AA^\infty}:= \prod_{v \neq \infty} O_v$.
Let $A$ be the ring of functions in $k$ regular away from $\infty$. Then the finite places of $k$ (i.e.\ the place not equal to $\infty$) are canonically identified with non-zero prime ideals of $A$.  For a fractional ideal $\afk$ of $A$, we denote by $\afk \lhd A$ if $\afk$ is an integral ideal. In this paper, every ideal is assumed to be non-zero.
For each fractional ideal $I$ of $A$, writing $I = \afk^{-1} \bfk$ where $\afk, \bfk \lhd A$ we set
$$\Vert I \Vert :=\frac{\#(A/ \bfk)}{\#(A/ \afk)}.$$
Note that $\Vert \alpha A \Vert = |\alpha|_\infty$ for $\alpha \in k^\times$.
Finally, we put $\deg a := -\deg \infty \ord_\infty (a)$ for $a \in A$.

\subsection{Drinfeld modules}\label{sec 1.2}

Let $(F,\iota)$ be an \textit{$A$-field}, i.e.\ $F$ is a field together with a ring homomorphism $\iota:A \rightarrow F$. The $\FF_q$-linear endomorphism ring $\End_{\FF_q}(\GG_{a/F})$ is isomorphic to the \textit{twisted polynomial ring} $F\{\tau\}$, where $\tau : \GG_{a/F} \rightarrow \GG_{a/F}$ is the Frobenius map $(x \mapsto x^q)$ and $\tau a = a^q \tau$ for every $a \in F$. 

\begin{defn}\label{defn 1.1}
Suppose an $A$-field $(F, \iota)$ and a positive integer $r$ is given.\\
$(1)$ A \text{\it Drinfeld $A$-module over $F$ of rank $r$} is a ring homomorphism $\rho: A \rightarrow F\{\tau\}$ satisfying that
$$\rho_a = \iota(a) + \sum_{i = 1}^{r \deg a} l_i(\rho_a) \tau^i \in F\{\tau \}, \quad \text{with } l_{r\deg a}(\rho_a) \neq 0 \ \ \forall a \in A.$$
$(2)$ Given two Drinfeld $A$-modules $\rho$ and $\rho'$ over $F$, a {\it homomorphism $f: \rho \rightarrow \rho'$ over $F$} is an element in $F\{\tau\}$ satisfying $ f \cdot \rho_a = \rho'_a \cdot f$ for every $a \in A$. We call $f$ an {\it isogeny} if $f \neq 0$. We denote the endomorphism ring of $\rho$ over $F$ by $\End_A(\rho/F)$.
\end{defn}

\subsection{Drinfeld period domain}\label{sec 1.3}

Let $\CC_\infty$ be the completion of a chosen algebraic closure of $k_\infty$. We may view $\CC_\infty$ as an $A$-field via the natural embedding $A\hookrightarrow \CC_\infty$. Given a Drinfeld $A$-module $\rho$ of rank $r$ over $\CC_\infty$.
There exists a unique $\FF_q$-linear entire function $\exp_\rho$ on $\CC_\infty$ satisfying
$$\exp_\rho(w) = w + \sum_{i=1}^\infty c_i w^{q^i} \quad  \text{ and }\quad  \exp_\rho(aw) = \rho_a\big( \exp_\rho(w)\big), \quad \forall a \in A.$$
It is known that (cf.\ \cite[Theorem 4.6.9]{Gos})
$\Lambda_\rho := \{\lambda \in \CC_\infty: \exp_\rho(\lambda) = 0\}$
 is a discrete projective $A$-submodule of rank $r$ in $\CC_\infty$ (i.e.\ an $A$-lattice of rank $r$ in $\CC_\infty$). We call $\Lambda_\rho$ the \textit{$A$-lattice associated to $\rho$}.
On the other hand, given an $A$-lattice $\Lambda$ of rank $r$ in $\CC_\infty$, set
$$\exp_\Lambda(w):= w \prod_{0 \neq \lambda \in \Lambda}\left(1-\frac{w}{\lambda}\right).$$
This uniquely determines a rank $r$ Drinfeld $A$-module $\rho^{\Lambda}$ over $\CC_\infty$ satisfying that 
\begin{equation}\label{eqn 1.1}
\exp_{\Lambda}(aw) = \rho^{\Lambda}_a\left(\exp_{\Lambda}(w)\right),\quad \forall a \in A.
\end{equation}
In other words, the correspondence $\rho \leftrightarrow \Lambda_\rho$ gives us a bijection (cf.\ \cite[Proposition 3.1]{Dri})
$$\{\text{Drinfeld $A$-modules of rank $r$ over $\CC_\infty$}\} \cong \{\text{$A$-lattices of rank $r$ in $\CC_\infty$}\}.$$

We now recall the analytic description of the moduli space for rank $r$ Drinfeld $A$-modules over $\CC_\infty$. Given $a = (a_1: \cdots :a_r) \in \PP^{r-1}(k_\infty)$, let $\widetilde{H}_a \subset \CC_\infty^r$ (resp.\ $H_a \subset \PP^{r-1}(\CC_\infty)$) be the $k_\infty$-rational hyperplane corresponding to $a$, i.e.\
$$\widetilde{H}_a :=\Big\{(z_1, ... ,z_r) \in \CC_\infty^r : \sum_{i=1}^r a_iz_i = 0\Big\}  \text{ and } 
H_a :=\Big\{(z_1:\cdots:z_r) \in \PP^{r-1}(\CC_\infty):  \sum_{i=1}^r a_iz_i = 0\Big\}.$$
Let $$\widetilde{\Hfk}^{r} := \CC_\infty^r - \bigcup_{a \in \PP^{r-1}(k_\infty)} \widetilde{H}_a \quad \text{ and } \quad \Hfk^r:= \PP^{r-1}(\CC_\infty) - \bigcup_{a \in \PP^{r-1}(k_\infty)} H_a \ (= \widetilde{\Hfk}^{r}/\CC_\infty^\times).$$ 
We call $\Hfk^r$ the \textit{Drinfeld period domain of rank $r$}.
Note that $\widetilde{\Hfk}^{r}$ and $\Hfk^r$ are equipped with a (compatible) left action of $\GL_r(k_\infty)$:
given $\tilde{z} = (z_1,...,z_r) \in \widetilde{\Hfk}^{r}$ and the corresponding point $z = (z_1:\cdots: z_r) \in \Hfk^r$, for $g = (a_{ij})_{1\leq i,j\leq r} \in \GL_r(k_\infty)$ we
put $g \cdot \widetilde{z}:= (z'_1,..., z'_r) \in \widetilde{\Hfk}^{r}$ and $g \cdot z := (z'_1:\cdots: z'_r) \in \Hfk^r$ where
$$\begin{pmatrix}z_1'\\ \vdots \\ z_r'\end{pmatrix} = \begin{pmatrix} a_{11} & \cdots &a_{1r} \\ \vdots & & \vdots \\ a_{r1} & \cdots & a_{rr} \end{pmatrix}\begin{pmatrix}z_1\\ \vdots \\ z_r\end{pmatrix}.$$

Note that every $z \in \Hfk^r$ has a unique representative $z=(z_1:\cdots: z_{r-1}: 1)$.
For each
$g = \begin{pmatrix} * & * \\ c_1 \cdots c_{r-1} & d \end{pmatrix} \in \GL_r(k_\infty)$, we set
$$j(g,z):= c_1z_1 + \cdots c_{r-1} z_{r-1} + d.$$
Let $Y \subset k^r$ be a projective $A$-module of rank $r$.
For $z = (z_1: \cdots : z_{r-1}:1) \in \Hfk^r$, put
$$\Lambda^Y_z:= \{a_1z_1+ \cdots + a_{r-1} z_{r-1} +a_r\subset \CC_\infty : (a_1,...,a_r) \in Y\}.$$
Observe that
\begin{eqnarray}\label{eqn 1.2}
\Lambda^{Y\gamma^{-1}}_{\gamma z}
= j(\gamma,z)^{-1} \cdot \Lambda^Y_z, \quad \forall \gamma  \in \GL_r(k).
\end{eqnarray}
Let $\rho^{Y,z}$ denote the rank $r$ Drinfeld $A$-module over $\CC_\infty$ corresponding to the $A$-lattice $\Lambda^Y_z$.

\begin{thm}\label{thm 1.2}
{\rm (cf.\ \cite{Dri})} The map $(Y,z)\mapsto \rho^{Y,z}$ induces the following bijection
$$
\Mfk^{(r)}_A := \left(\coprod\limits_{[Y] \in \Pscr_A^r} \GL(Y)\backslash \Hfk^r \right) \longleftrightarrow \left\{\text{\rm rank-$r$ Drinfeld $A$-modules over $\CC_\infty$}\right\} / \cong.$$
Here $\Pscr_A^r$ is the set of isomorphism classes of projective $A$-modules of rank $r$.
\end{thm}

\subsection{Drinfeld-Gekeler discriminant function}\label{sec 1.4}

Given  $z = (z_1:\cdots: z_{r-1}: 1) \in \Hfk^r$ and a projective $A$-module $Y$ of rank $r$ in $k^r$, let
$$\Delta^Y_a(z) := l_{r\deg a}(\rho^{Y,z}_a), \quad \forall a \in A.$$
Then the equation~(\ref{eqn 1.2}) implies (cf.\ \cite[Chapter V, 3.4 Example]{Ge1})
$$\Delta^Y_a(\gamma z) = j(\gamma,z)^{q^{r\deg a}-1} \cdot \Delta^Y_a(z), \quad \forall \gamma  \in \GL(Y).$$
Moreover, the functional equation (\ref{eqn 1.1}) implies
$$\rho^{Y,z}_a(x) = \Delta^Y_a(z) \cdot x \cdot \prod_{0 \neq w \in \frac{1}{a}\Lambda^Y_z/\Lambda^Y_z} \big(x - \exp_{\Lambda^Y_z}(w)\big), \quad \forall a \in A -\{0\}.$$
Therefore we have the following product formula of $\Delta^Y_a(z)$:

\begin{lem}\label{lem 1.3}
For every $a \in A$,
$$\Delta^Y_a(z) = a \cdot \prod_{0\neq w \in \frac{1}{a} \Lambda^Y_z/\Lambda^Y_z} \exp_{\Lambda^Y_z}(w)^{-1}.$$
\end{lem}

Since 
$$\rho^{Y,z}_a \cdot \rho^{Y,z}_b = \rho^{Y,z}_{ab} = \rho^{Y,z}_b \cdot \rho^{Y,z}_a, \quad \forall a,b \in A,$$
one gets
\begin{equation}\label{eqn 1.3}
\Delta^{Y}_a(z) \cdot \Delta^Y_b(z) ^{q^{r\deg a}} = \ \Delta^Y_{ab}(z) \ = \ \Delta^Y_b(z) \cdot \Delta^Y_a(z)^{q^{r\deg b}}.
\end{equation}
Take two elements $a_1, a_2 \in A$ such that $\text{gcd}(\ord_\infty(a_1),\ord_\infty(a_2)) = 1$, and choose $\ell_1$, $\ell_2 \in \ZZ$ such that $\ell_1\big(q_\infty^{r\ord_\infty(a_1)}-1\big)+\ell_2\big(q_\infty^{r\ord_\infty( a_2)}-1\big) = q_\infty^r-1$. Set
$$\Delta^Y(z):= \Delta^Y_{a_1}(z)^{\ell_1} \cdot \Delta^Y_{a_2}(z)^{\ell_2},$$
which is a nowhere-zero analytic function on $\Hfk^r$ satisfying that
$$\Delta^Y(\gamma z) = j(\gamma,z)^{q_\infty^r-1} \cdot \Delta^Y(z), \quad \forall \gamma  \in \GL(Y).$$

\begin{prop}\label{prop 1.4}
\text{\rm (cf.\ \cite[Chapter IV, Proposition 5.15]{Ge1})}
The function $\Delta^Y$ is, up to multiplying with $(q_\infty^r-1)$-th roots of unity, independent  of the chosen $a_1, a_2 \in A$ and $\ell_1, \ell_2 \in \ZZ$.
In particular, one has $$(\Delta^Y)^{q^{r\deg a}-1} = (\Delta^Y_a)^{q_\infty^r-1}, \quad \forall a \in A-\{0\}.$$
\end{prop}

The following \lq\lq norm compatibilities\rq\rq\ are straightforward.

\begin{lem}\label{lem 1.5}
Given two projective $A$-modules $Y$ and $Y'$ of rank $r$ in $k^r$ with $Y' \subset Y$, for $z \in \Hfk^r$ one has
$$\exp_{\Lambda^Y_z}(w) = \exp_{\Lambda^{Y'}_z}(w) \prod_{0 \neq u \in \Lambda^Y_z/\Lambda^{Y'}_z} \left(1-\frac{\exp_{\Lambda^{Y'}_z}(w)}{\exp_{\Lambda^{Y'}_z}(u)}\right),$$
and
$$\Delta^Y(z) = \Delta^{Y'}(z)^{\#(Y/Y')}\prod_{0 \neq u \in \Lambda^Y_z/\Lambda^{Y'}_z}
\exp_{\Lambda^{Y'}_z}(u)^{q_\infty^r-1}.$$
\end{lem}

\subsection{Building map}\label{sec 1.5}

Let $\Bcal^r$ be the Bruhat-Tits building associated to $\PGL_r(k_\infty)$. 
The set $V(\Bcal^r)$ of vertices of $\Bcal^r$ consists of all the homothety classes of $O_\infty$-lattices in $k_\infty^r$.
Take $L^o := O_\infty^r \subset k_\infty^r$, the standard $O_\infty$-lattice in $k_\infty^r$.
Via the left action of $\GL_r(k_\infty)$ on $\Bcal^r$, the set $V(\Bcal^r)$ can be identified with $\GL_r(k_\infty)/k_\infty^{\times} \GL_r(O_\infty)$: 
$$V(\Bcal^r) = \{[ L^o g^{-1}] : g \in \GL_r(k_\infty)/k_\infty^{\times} \GL_r(O_\infty)\}.$$
For $0\leq i < r$, a (resp.\ non-)oriented $i$-simplex is an $i$-tuples $([L_0], ...,[L_i])$ (resp.\ up to cyclic permutations), where $L_0,...,L_i$ are $O_\infty$-lattices satisfying
$$L_0 \supsetneq \cdots \supsetneq L_i \supsetneq \pi_\infty L_0.$$
We let $\vec{C}^i(\Bcal^r)$ (resp.\ $C^i(\Bcal^r)$) be the set consisting of all the (resp.\ non-)oriented $i$-simplices.

It is known that the realization $\Bcal^r(\RR)$ of $\Bcal^r$ is identified with the equivalence classes of norms on $k_\infty^r$ as follows: suppose $P \in \Bcal^r(\RR)$ belongs to the realization of an $i$-simplex, say $([L_0], ...,[L_i])$ with $L_0 \supsetneq \cdots \supsetneq L_i \supsetneq \pi_\infty L_0$. Write $P = \sum_{j=0}^i \epsilon_j [L_j]$ with $0 \leq \epsilon_j \leq  1$ and $\sum_{j=0}^i \epsilon_j = 1$.
Then
$$\nu_P:= \sup\{ q^{-\xi_j} \nu_{L_j} : 0\leq j \leq i\}$$
with
$$\xi_j := \sum_{\ell = 0}^{j-1} \epsilon_\ell \quad \text{ and } \quad \nu_L(x):= \inf\{ |a|_\infty: a \in k_\infty \text{ with } x \in a L\}.$$

\begin{defn}\label{defn 1.6}
The \textit{building map} $\lambda : \Hfk^r \rightarrow \Bcal^r(\QQ)$ is defined by
$$z = (z_1 : \cdots : z_{r-1}: 1) \in \Hfk^r \longmapsto \nu_z := \big((a_1,...,a_r) \in k_\infty^r \mapsto |a_1z_1 + \cdots + a_{r-1} z_{r-1} + a_r |_\infty\big).$$
\end{defn}

The right action of $\GL_r(k_\infty)$ on $k_\infty^r$ yields a left action on the set of norms on $k_\infty^r$ and then on $\Bcal_r(\RR)$.

\begin{prop}\label{prop 1.6}
\text{\rm (cf.\ \cite[Proposition 1.5.3]{G-R} and \cite[(4.2) Proposition]{D-H}} The building map $\lambda$ is $\GL_r(k_\infty)$-equivariant.
\end{prop}

\subsubsection{Imaginary part}\label{sec 1.5.1}

Given $z = (z_1:\cdots : z_{r-1} : 1) \in \Hfk^r$, for $1\leq i < r$, define the \lq\lq $i$-th imaginary part\rq\rq\ of $z$ by
$$\im(z)_i := \inf\left\{ \big|z_i + (\sum_{j=i+1}^{r-1} u_j z_j) + u_r\big|_\infty\ :\ u_{i+1},...,u_{r} \in k_\infty \right\}.$$
Take $\omega = (\omega_1 : \cdots :\omega_{r-1} : \omega_r) \in \Hfk^r$ where $\omega_i :=  z_i +  (\sum_{j=i+1}^{r-1} u_{ij} z_j) + u_{ir} \in \CC_\infty$ with $u_{ij} \in k_\infty$ so that $\im(z)_i = |\omega_i|_\infty$ for $1\leq i < r$ and $\omega_r :=1$.
Then
$$\nu_\omega(x) = \sup \{ |x_i \omega_i|_\infty : 1\leq i \leq r\}, \quad \forall x = (x_1,...,x_r) \in k_\infty^r.$$
Indeed, suppose $\nu_\omega(x) = |x_1\omega_1+\cdots+x_r \omega_r|_\infty < \sup\{|x_i \omega_i|_\infty: 1\leq i \leq r\}$.
Take $i_0$ minimal so that $|x_{i_0}\omega_{i_0}|_\infty =  \sup\{|x_i \omega_i|_\infty: 1\leq i \leq r\}$. Then $x_{i_0} \neq 0$ and $|x_i\omega_i|_\infty < |x_{i_0}\omega_{i_0}|_\infty$ for $i< i_0$, which implies 
$$|x_{i_0}\omega_{i_0}+ \cdots + x_r\omega_r|_\infty < |x_{i_0}\omega_{i_0}|_\infty.$$
Expressing 
$$\omega_{i_0}+ \frac{x_{i_0+1}}{x_{i_0}} \omega_{i_0+1} + \cdots + \frac{x_r}{x_{i_0}}
= z_{i_0}+ u_{i_0+1}'z_{i_0+1} + \cdots + u_r' \quad \text{ for some } u_{i_0+1}',...,u_r' \in k_\infty,$$
we get $|z_{i_0}+ u_{i_0+1}'z_{i_0+1} + \cdots + u_r'|_\infty < |\omega_{i_0}|_\infty = \im(z)_{i_0}$, a contradiction.\\

Write $|\omega_i|_\infty = q_\infty^{-\ell_i + 1 - \xi_i}$ where $\ell_i \in \ZZ$ and $\xi_i \in \QQ$ with $0\leq \xi_i < 1$ for $1\leq i < r$.
Let $\xi_0 := 0$ and $\xi_r := 1$ (so $\ell_r = 0$).
Take a permutation $\sigma$ of $\{1,...,r-1\}$ so that
$$\xi_{\sigma(1)} \leq \cdots \leq \xi_{\sigma(r-1)}.$$
We may put $\sigma(0):= 0$ and $\sigma(r) := r$.
Set
$$g_\omega = g_{\omega,0} := \begin{pmatrix} \pi_\infty^{\ell_1} & &  \\ & \ddots &  \\ & & \pi_\infty^{\ell_r} \end{pmatrix} \quad \text{ and } \quad
g_{\omega,i} := \begin{pmatrix}\pi_\infty^{\ell_1^{(i)}} & & \\ & \ddots & \\ & & \pi_\infty^{\ell_r^{(i)}}\end{pmatrix} \text{ for $1\leq i < r$,} $$
where 
$$\ell_{\sigma(j)}^{(i)} := \begin{cases} \ell_{\sigma(j)} - 1, & \text{ if $j\leq i$,} \\ \ell_{\sigma(j)}, & \text{ otherwise.}\end{cases}$$
Take $L_{\omega,i} := L^o g_{\omega,i}^{-1} \subset k_\infty^r$ for $1\leq i < r$.
Then we observe that
$$\nu_{\omega} = \sup \{q_\infty^{-\xi_{\sigma(i)}} \nu_{L_{\omega,i}} : 0\leq i <r \}.$$
In other words, we have
$$\lambda(\omega) = \sum_{0\leq i < r} \epsilon_i [L_{\omega,i}] \in \Bcal^r(\RR) \quad \text{ with } \quad \epsilon_{i} := \xi_{\sigma(i+1)} - \xi_{\sigma(i)}.$$

Note that $\omega = u \cdot z$ where 
$$u = \begin{pmatrix} 1 & & u_{ij} \\ & \ddots & \\ & & 1\end{pmatrix}.$$
For $0\leq i < r$ we take
\begin{eqnarray} \label{eqn 1.4}
g_{z,i}:= u^{-1} g_{\omega,i} \quad \text{ and } \quad L_{z,i} := L^o g_{z,i}^{-1}.
\end{eqnarray} 
Then:

\begin{lem}\label{lem 1.8}
For $z \in \Hfk^r$ with $\im(z)_i = q_\infty^{-\ell_i + 1 -\xi_i}$ where $\ell_i \in \ZZ$ and $0 \leq \xi_i < 1$, we have 
$$\lambda(z) = \sum_{i=0}^{r-1} \epsilon_i [L_{z,i}],$$
where $\epsilon_i$ and $L_{z,i}$ are taken as above for $0\leq i < r$.
\end{lem}

Define the total imaginary part of $z \in \Hfk^r$ by
\begin{eqnarray}\label{eqn 1.5}
\im(z) := \prod_{i=1}^{r-1} \im(z)_i,
\end{eqnarray}
and put
$
[\im(z)]_i := |\det(g_{z,i})|_\infty$
for $0\leq i < r$.
The above lemma implies:

\begin{cor}\label{cor 1.9}
$(1)$ For $z_1,z_2 \in \Hfk^r$ with $\lambda(z_1) = \lambda(z_2)$, we get $\im(z_1) = \im(z_2)$.\\
$(2)$ Given $s \in \CC$, the following equality holds:
$$\frac{\im(z)^s}{\nu_z(x)^{rs}} = \sum_{i=0}^{r-1} c_{z,i}(s) \cdot \frac{[\im(z)]_i^s}{\nu_{L_{z,i}}(x)^{rs}}, \quad \forall x \in k_\infty^r-\{0\},$$
where
$$c_{z,i}(s) := q_{{}_\infty}^{{}^{\sum_{j=1}^r (\xi_{\sigma(i)}-\xi_j)s}} \cdot \frac{q_{{}_\infty}^{r\epsilon_i s}-1}{q_{{}_\infty}^{is}-q_{{}_\infty}^{(i-r)s}}, \quad 0\leq i <r.$$
\end{cor}

\begin{proof}
Note that $\nu_z = \sup\{q_\infty^{-\xi_{\sigma(i)}} \nu_{L_{z,i}} : 0\leq i < r\}$.
Given $x \in k_\infty^r-\{0\}$, we have that $\nu_z(x) = q_\infty^{-\xi_{\sigma(i_0)}} \nu_{L_{z,i_0}}(x)$ with minimal $i_0 \in \{0,...,r-1\}$ if and only if
$$\nu_{L_{z,j}}(x) = \begin{cases}
q_\infty^{-1}  \nu_{L_{z,i_0}}(x), & \text{ if $j < i_0$,} \\
 \nu_{L_{z,i_0}}(x), & \text{ otherwise.}
 \end{cases}
$$
Then
$$
\frac{\im(z)^s}{\nu_z(x)^{rs}}
= q_\infty^{r\xi_{\sigma(i_0)} s} \cdot \frac{q_\infty^{-\sum_{j=1}^r (\ell_j-1+\xi_j)s}}{\nu_{L_{z,i_0}}(x)^{rs}}.
$$
On the other hand, one has
\begin{eqnarray}
&& \sum_{i=0}^{r-1}c_{z,i}(s) \frac{[\im(z)]_i^s}{\nu_{L_{z,i}}(x)^{rs}} \nonumber \\
&=&
\frac{q_\infty^{-\sum_{j=1}^r(\ell_j + \xi_j) s}}{\nu_{L_{z,i_0}}(x)^{rs}}
\cdot \left[\bigg(
\sum_{i=0}^{i_0-1} q_\infty^{r\xi_{\sigma(i)} s} \cdot \frac{q_{{}_\infty}^{r\epsilon_i s}-1}{q_{{}_\infty}^{is}-q_{{}_\infty}^{(i-r)s}} \cdot \frac{q_\infty^{is}}{q_\infty^{-rs}}\bigg)  + \bigg(\sum_{i=i_0}^{r-1} q_\infty^{r\xi_{\sigma(i)} s} \cdot \frac{q_{{}_\infty}^{r\epsilon_i s}-1}{q_{{}_\infty}^{is}-q_{{}_\infty}^{(i-r)s}} \cdot q_\infty^{is}
\bigg)\right] \nonumber \\
&=& \frac{q_\infty^{-\sum_{j=1}^r(\ell_j -1+ \xi_j) s}}{ \nu_{L_{z,i_0}}(x)^{rs}\cdot (1-q_\infty^{-rs})}\cdot \left[\bigg(\sum_{i=0}^{i_0-1} q_\infty^{r\xi_{\sigma(i+1)} s} - q_\infty^{r\xi_{\sigma(i)}s}\bigg)
+ q_\infty^{-rs}\bigg(\sum_{i=i_0}^{r-1} q_\infty^{r\xi_{\sigma(i+1)} s} - q_\infty^{r\xi_{\sigma(i)}s}\bigg)\right] \nonumber \\
&=& \frac{q_\infty^{-\sum_{j=1}^r(\ell_j -1+ \xi_j) s}}{ \nu_{L_{z,i_0}}(x)^{rs}} \cdot q_\infty^{r\xi_{\sigma(i_0)} s}  . \nonumber
\end{eqnarray}
Therefore the result holds.
\end{proof}

\begin{rem}\label{rem 1.10}
For an $A$-lattice $\Lambda$ of rank $r$ in $\CC_\infty$, the {\it lattice discriminant} of $\Lambda$, denoted by $D_A(\Lambda)$, is the \lq\lq covolume\rq\rq\ of $\Lambda$ (cf.\ \cite[Section 4]{Wei2}):
choose an \lq\lq orthogonal\rq\rq\ $k_\infty$-basis $\{\lambda_i\}_{1\leq i \leq r}$ of $k_\infty \cdot \Lambda$, i.e.\ $\lambda_1,...,\lambda_r$ satisfy that 
\begin{itemize}
\item[(i)] $\lambda_i \in \Lambda$ for $1\leq i \leq r$;
\item[(ii)] $|a_1 \lambda_1+\cdots + a_{r} \lambda_{r}|_{\infty} = \max\{ |a_i \lambda_i|_\infty; 1\leq i \leq r\}$ for all $a_1,...,a_{r} \in k_\infty$.
\item[(iii)] $k_\infty \cdot \Lambda = \Lambda + (O_\infty \lambda_1 + \cdots O_\infty \lambda_r)$.
\end{itemize}
Set
$$D_A( \Lambda):= q^{1-g_k}\cdot\left( \frac{ \prod_{1\leq i \leq r} |\lambda_i|_\infty}{\#\big(\Lambda \cap (O_\infty\lambda_1 + \cdots +O_\infty \lambda_r)\big)}\right)^{1/r} = \left( \frac{\prod_{1\leq i \leq r} |\lambda_i|_\infty}{\#\big(\Lambda /(A \lambda_1 + \cdots +A \lambda_r)\big)}\right)^{1/r} .$$
It is clear that $D_A(c \cdot \Lambda) = |c|_\infty \cdot D_A(\Lambda)$ for every $c \in \CC_\infty^\times$.
In particular, for $z \in \Hfk^r$ and a rank $r$ projective $A$-module $Y \subset k^r$, we have
\begin{eqnarray}\label{eqn 1.6}
D_A(\Lambda_z^Y)^r &=& \Vert Y \Vert \cdot \im(z).
\end{eqnarray}
Here $\Vert Y \Vert := \#(A^r/\afk Y) \cdot \Vert \afk\Vert^{-r}$ for every ideal $\afk$ of $A$ so that $\afk Y \subset A^r$.

\begin{lem}\label{lem 1.11}
Given $z \in \Hfk^r \text{ and } \gamma \in \GL_r(k_\infty)$, we have
$$\im(\gamma \cdot z) = \frac{|\det \gamma |_\infty}{|j(\gamma,z)|_\infty^r} \cdot \im(z)$$
\end{lem}

\begin{proof}
Take $\gamma' \in \GL_r(k)$ closed enough to $\gamma$ so that
$$\im(\gamma' \cdot z) = \im(\gamma \cdot z), \quad |\det \gamma'|_\infty = |\det \gamma|_\infty, \quad \text{ and } \quad |j(\gamma',z)|_\infty = |j(\gamma,z)|_\infty.$$
The result then follows from the equalities~(\ref{eqn 1.2}) and (\ref{eqn 1.6}).
\end{proof}
\end{rem}

\section{\lq\lq Non-holomorphic\rq\rq\ Eisenstein series}\label{sec 2}

We first recall the basic properties of mirabolic Eisenstein series on $\GL_r(\AA)$ to be used.

\subsection{Mirabolic Eisenstein series}\label{sec 2.1}

Let $\chi : k^\times \backslash \AA^\times \rightarrow \CC^\times$ be a unitary Hecke character.
Given a Schwartz function $\varphi \in S(\AA^r)$, i.e.\ the function $\varphi$ on $\AA^r$ is locally constant and compactly supported, put
$$\Phi(g,s;\chi,\varphi) := |\det(g)|_\AA^s \cdot  \int_{\AA^\times} \varphi\big((0,...,0,a^{-1})g\big) \chi(a) |a|_\AA^{-rs} d^\times a, \quad \forall g \in \GL_r(\AA)$$
The Haar measure $d^\times a$ on $\AA^\times$ is chosen so that $\text{vol}(O_\AA^\times, d^\times a) = 1$.
It is known that (cf.\ \cite[(4.1)]{J-S} or \cite[p.\ 119]{S-S}) the function $\Phi(g,s;\chi,\varphi)$ converges absolutely for every $g \in \GL_r(\AA)$ and $s \in \CC$ with $\re(s) > 1/r$, and has meromorphic continuation to the whole $s$-plane. Moreover, for each element $b = \begin{pmatrix} a & * \\ 0 & d \end{pmatrix} \in \GL_r(\AA)$ with $a \in \GL_{r-1}(\AA)$ and $d \in \AA^\times$, one has
$$\Phi(bg,s;\chi,\varphi) = |\det a|_\AA^s \cdot \chi(d) |d|_\AA^{(1-r)s} \cdot \Phi(g,s;\chi,\varphi), \quad \forall g \in \GL_r(\AA).$$

Let
$${\rm P} := \left\{ \begin{pmatrix} a & * \\ 0 & d \end{pmatrix} \in \GL_r : a \in \GL_{r-1}, \ d \in \GL_1 \right\}.$$

\begin{defn}\label{defn 2.1}
The \textit{mirabolic Eisenstein series on $\GL_r(\AA)$ associated to  $\chi$ and $\varphi$} is defined by:
$$\Ecal(g,s;\chi,\varphi) := \sum_{\gamma \in {\rm P}(k)\backslash \GL_r(k)} \Phi(\gamma g,s;\chi,\varphi), \quad \forall g \in \GL_r(\AA).$$
\end{defn}

The needed analytic properties of $\Ecal(g,s;\chi,\varphi)$ are stated in the following (cf.\ \cite[Section 4]{J-S} or \cite[p.\ 120]{S-S}):

\begin{thm}\label{thm 2.2}
${}$
\begin{itemize}
\item[(1)] The Eisenstein series $\Ecal(g,s;\chi,\varphi)$ converges absolutely for $\re(s) > 1$ and has meromorphic continuation to the whole $s$-plane. 
Moreover, $\Ecal(g,s;\chi,\varphi)$ is entire except for $\chi = |\cdot |_\AA^{\epsilon}$ with $\epsilon \in \sqrt{-1} \cdot \RR$.
In this case,
the only possible poles of $\Ecal(g,s;\chi,\varphi)$ are at 
$$s = \left\{-\frac{\epsilon}{r} + \frac{2n \pi \sqrt{-1}}{r\ln q} : n \in \ZZ\right\} \cup  \left\{1+\frac{\epsilon}{r} + \frac{2n \pi \sqrt{-1}}{r \ln q} : n \in \ZZ\right\},$$
with residues independent of the chosen $g \in \GL_r(\AA)$. 
\item[(2)] Let $\widehat{\varphi}$ be the Fourier transform of $\varphi$: For $x = (x_1,...,x_r) \in \AA^r$, put
$$\widehat{\varphi} (x) := \int_{\AA^r} \varphi(y) \psi( \sum_{i=1}^r x_i y_i) \ dy.$$
Here $dy$ is chosen to be self-dual with respect to the fixed additive character $\psi$ on $k\backslash \AA$. Then $\Ecal(g,s;\chi,\varphi)$ satisfies the following functional equation:
$$\Ecal(g,s;\chi,\varphi) = \Ecal({}^tg^{-1},1-s;\chi^{-1}, \widehat{\varphi}).$$
\end{itemize}
\end{thm}

\begin{rem}\label{rem 2.3}
For $\re(s)>1$, we may express $\Ecal(g,s;\chi,\varphi)$ in the following form (cf.\ \cite[p.\ 119]{S-S}):
$$
\Ecal(g,s;\chi,\varphi) = |\det g|_\AA^s \cdot 
\int_{k^\times \backslash \AA^\times} \left(\sum_{0 \neq x \in k^r} \varphi(a^{-1} x g) \right) \chi(a)|a|_\AA^{-rs} d^\times a.
$$
\end{rem}

\subsection{Finite mirabolic Eisenstein series}\label{sec 2.2}

Let $d^\times a^\infty$ be the Haar measure on $\AA^{\infty,\times}$ normalized so that $\text{vol}(O_{\AA^\infty}^\times, d^\times a^\infty) = 1$.
Then $\Ical^\infty:=k^\times \backslash \AA^{\infty,\times} (\cong 
k^\times \backslash \AA^\times/ k_\infty^\times)$ is compact and
\begin{eqnarray}\label{eqn Ivol}
\text{vol}(\Ical^\infty, d^\times a^\infty) = \frac{\#(k^\times \backslash \AA^{\infty,\times}/O_{\AA^\infty}^\times)}{\#(O_{\AA^\infty}^\times\cap k^\times)} &=& \frac{\#\Pic(A)}{q-1}.
\end{eqnarray}
Denote by $\widehat{\Ical}^\infty$ the Pontryagin dual group of $\Ical^\infty$.
Given a Schwartz function $\varphi^\infty \in S((\AA^\infty)^r)$, 
we set
$$\Ecal^\infty(g,s;\varphi^\infty) := \frac{q-1}{\#\Pic(A)} \cdot\sum_{\chi \in \widehat{\Ical}^\infty}
\Ecal(g,s;\chi,\varphi^\infty \otimes \mathbf{1}_{L^o}), \quad \forall g  \in \GL_r(\AA).$$
Here $\mathbf{1}_{L^o}$ is the characteristic function of $L^o = O_\infty^r$.
Observe that $\Ecal(g,s;\chi,\varphi^\infty \otimes \mathbf{1}_{L^o}) = 0$ for all but finitely many $\chi \in \widehat{\Ical}^\infty$. Thus $\Ecal^\infty(g,s;\varphi^\infty)$ is a well-defined meromorphic function (in the variable $s$).
\begin{lem}\label{lem 2.4}
Given $\varphi^\infty \in S((\AA^\infty)^r)$, for  $g = (g_\infty,g^\infty) \in \GL_r(k_\infty) \times \GL_r(\AA^\infty) = \GL_r(\AA)$ we have
$$\Ecal^\infty (g,s;\varphi^\infty) = \frac{|\det g|_\AA^s}{1-q_\infty^{-rs}} \cdot  \sum_{0 \neq x \in k^r} \frac{\varphi^\infty(xg^\infty)}{\nu_{L^o}(xg_\infty)^{rs}},\quad \re(s)>1.$$
Here $\nu_{L^o}$ is the norm on $k_\infty^r$ associated to $L^o$ \text{\rm (cf.\ Section~\ref{sec 1.5})}.
\end{lem}

\begin{proof}
We first claim that
\begin{eqnarray}\label{eqn:fmE}
&& \Ecal^\infty (g,s;\varphi^\infty) \nonumber \\
&=&|\det g|_\AA^s \cdot  \sum_{0 \neq x \in k^r} \varphi^\infty(xg^\infty) \cdot \left(\int_{k_\infty^\times} \mathbf{1}_{L^o}(a_\infty^{-1} x g_\infty)  |a_\infty|_\infty^{-rs} d^\times a_\infty \right), \quad \re(s)>1. 
\end{eqnarray}
Recall that $\nu_{L^o}(x) = \inf\{|a_\infty |_\infty : a_\infty \in k_\infty \text{ with } x \in a_\infty L^o\}$.
The result then follows from the identities below:
\begin{eqnarray}
\int_{k^\times_\infty} \mathbf{1}_{L^o}(a_\infty^{-1} x g_\infty) |a_\infty|_\infty^{-rs} d^\times a_\infty
&=& \int_{|a_\infty|_\infty \geq \nu_{L^o}(x g_\infty)} |a_\infty|_\infty^{-rs} d^\times a_\infty \nonumber  \\
&=& (1-q_\infty^{-rs})^{-1} \cdot \nu_{L^o}(x g_\infty)^{-rs}, \quad \re(s)>0. \nonumber
\end{eqnarray}
To show (\ref{eqn:fmE}), take an open compact subgroup $U$ of $O_{\AA^\infty}^\times$ small enough so that
\begin{itemize}
\item[(i)] $(U \cdot k_\infty^\times )\cap k^\times = \{1\}$ (in $\AA^\times$);
\item[(ii)] $\varphi^\infty(ux) = \varphi^\infty(x)$ for every $u \in U$ and $x \in (\AA^\infty)^r$.
\end{itemize}
Let $h_U := \#(k^\times \backslash \AA^{\infty,\times} / U)$, and choose a set $\{1=a_1^\infty,...,a_{h_U}^\infty\}$ of representatives of the cosets in  $k^\times \backslash \AA^{\infty,\times} / U$.
Then the integration over $k^\times \backslash \AA^\times$ is equal to the integration over 
$\coprod_{i=1}^{h_U} a_i^\infty U \cdot k_\infty^\times$.
For each $\chi \in \widehat{\Ical}^\infty$ with $\chi(U) =1$, by \textit{Remark}~\ref{rem 2.3} one has
\begin{eqnarray}
&& \Ecal(g,s;\chi,\varphi^\infty \otimes \mathbf{1}_{L^o}) \nonumber \\
&=& |\det g|_\AA^s  \cdot \text{vol}(U,d^\times a^{\infty}) \nonumber \\
&& \cdot \sum_{0 \neq x \in k^r} \sum_{i=1}^{h_U} \varphi^\infty((a_i^\infty)^{-1} x g^\infty) \chi(a_i^\infty)|a_i^\infty|_{\AA^\times}^{-rs} \cdot \int_{k_\infty^\times} \mathbf{1}_{L^o}(a_\infty^{-1} x g_\infty) |a_\infty|_\infty^{-rs} d^\times a_\infty, \quad \re(s)>1. \nonumber 
\end{eqnarray}
Since
$$\sum_{\subfrac{\chi \in \widehat{\Ical}^\infty}{\chi(U) = 1}} \chi(a_i^\infty)
= \begin{cases} 
h_U & \text{ if $i = 1$} \\
0 & \text{ otherwise,}
\end{cases}
$$
we obtain that
\begin{eqnarray}
\sum_{\chi \in \widehat{\Ical}^\infty} \Ecal(g,s;\chi,\varphi^\infty \otimes \mathbf{1}_{L^o})
&=& \sum_{\subfrac{\chi \in \widehat{\Ical}^\infty}{\chi(U) = 1}} \Ecal(g,s;\chi,\varphi^\infty \otimes \mathbf{1}_{L^o}) \nonumber \\
&=& |\det g|_\AA^s \cdot  \text{vol}(U,d^\times a^{\infty}) \cdot h_U \nonumber \\
&& \cdot \sum_{0 \neq x \in k^r}\varphi^\infty(x g^\infty)\cdot \int_{k_\infty^\times} \mathbf{1}_{L^o}(a_\infty^{-1} x g_\infty) |a_\infty|_\infty^{-rs} d^\times a_\infty. \nonumber
\end{eqnarray}
Therefore (\ref{eqn:fmE}) follows from the identity 
$$\text{vol}(U,d^\times a^\infty) \cdot h_U = \text{vol}(\Ical^\infty, d^\times a^\infty) = \frac{\#\Pic(A)}{q-1}.$$
\end{proof}

\begin{rem}\label{rem 2.5}
(1) We may view $\Ecal^\infty (\cdot,s;\varphi^\infty)$  as a function on $\GL_r(\AA)/k_\infty^\times \GL_r(O_\infty) \cong \Bcal^r_\AA(\ZZ)$.\\
(2) Given $\chi \in \widehat{\Ical}^\infty$ and $\varphi^\infty \in S((\AA^\infty)^r)$, we have
$$\Ecal(g,s;\chi,\varphi^\infty \otimes \mathbf{1}_{L^o}) = \int_{k^\times \backslash \AA^{\infty,\times}} \overline{\chi(a^\infty)} \cdot \Ecal^\infty (ga^\infty,s;\varphi^\infty) d^\times a^\infty.$$
\end{rem}

\subsection{Adelic Drinfeld half space and \lq\lq non-holomorphic\rq\rq\ Eisenstein series}\label{sec 2.3}

We let $\Hfk_\AA^r := \Hfk^r \times \GL_r(\AA^\infty)$.
Then $\Hfk_\AA^r$ admits a left action of $\GL_r(k)$ and a right action of $\GL_r(\AA^\infty)$ as follows:
$$\gamma \cdot z_\AA \cdot h^\infty := (\gamma z, \gamma g^\infty h^\infty),\quad  \forall z_\AA = (z,g^\infty) \in \Hfk_\AA^r,\ \gamma \in \GL_r(k),\   h^\infty \in \GL_r(\AA^\infty).$$
We may extend the building map to $\lambda_\AA: \Hfk_\AA^r \rightarrow \Bcal^r_\AA(\RR):= \Bcal^r(\RR)\times \GL_r(\AA^\infty)$ by setting
$\lambda_\AA(z,g^\infty):= (\lambda(z), g^\infty)$. Note that $\Bcal^r_\AA(\ZZ):= \Bcal^r(\ZZ)\times \GL_r(\AA^\infty) \subset \Bcal^r_\AA(\RR)$ can be identified with $\GL_r(\AA)/k_\infty^\times \GL_r(O_\infty)$.\\

For $z_\AA = (z,g^\infty) \in \Hfk_\AA^{r}$, the \lq\lq total imaginary part\rq\rq\ of $z_\AA$ is defined by: 
\begin{eqnarray}\label{eqn 2.1}
\im(z_\AA) &:=& \im(z) \cdot |\det g^\infty|_\AA.
\end{eqnarray}

\begin{defn}\label{defn 2.6}
Given a Schwartz function $\varphi^\infty \in S((\AA^\infty)^r)$, the \lq\lq non-holomorphic\rq\rq\ Eisenstein series associated to $\varphi^\infty$ on $\Hfk_\AA^r$ is:
$$\EE(z_\AA,s;\varphi^\infty) := \sum_{0 \neq x \in k^r} \big(\varrho(g^\infty)\varphi^\infty\big)(x) \cdot \frac{\im(z_\AA)^s}{\nu_z(x)^{rs}}, \quad \forall z_\AA = (z,g^\infty) \in \Hfk_\AA^r.$$
Here $\nu_z$ is the norm on $k_\infty^r$ associated to $z \in \Hfk^r$ in Definition~\ref{defn 1.6},
and 
$\varrho$ is the left action of $\GL_r(\AA^\infty)$ on the space $S((\AA^\infty)^r)$ of Schwartz functions on $(\AA^\infty)^r$ induced by right translations:
$$\big(\varrho(g^\infty)\varphi^\infty\big) (x) := \varphi^\infty(x g^\infty), \quad \forall x \in (\AA^\infty)^r.$$
\end{defn}
It is observed that for $h^\infty \in \GL_r(\AA^\infty)$, one has
$$\EE(z_\AA \cdot h^\infty,s;\varphi^\infty) = |h^\infty|_\AA^s \cdot \EE\big(z_\AA,s;\varrho(h^\infty)\varphi^\infty\big).$$

\begin{prop}\label{prop 2.7}
Given $z_\AA \in \Hfk_\AA^r$ and $\varphi^\infty \in S((\AA^\infty)^r)$,
the Eisenstein series $\EE(z_\AA,s;\varphi^\infty)$ converges absolutely for $\re(s)>1$ and has meromorphic continuation to the whole $s$-plane.
Moreover, $\EE(\cdot,s;\varphi^\infty)$ factors through $\lambda_\AA$.
\end{prop}

\begin{proof}
By Corollary~\ref{cor 1.9} (1), it follows that $\EE(z_\AA,s;\varphi^\infty)$, as a function in $z_\AA$, only depends on $\lambda_\AA (z_\AA)$.
Moreover, for $z_\AA = (z,g^\infty) \in \Hfk_\AA^r$, Corollary~\ref{cor 1.9} (2) says that
$$\frac{\im(z)^s}{\nu_z(x)^{rs}}
= \sum_{i=0}^{r-1} c_{z,i}(s) \frac{|\det g_{z,i}|_\infty^s}{\nu_{L^o}(xg_{z,i})^{rs}}, \quad \forall x \in k_\infty^r-\{0\},$$
where $g_{z,i} \in \GL_r(k_\infty)$ for $0\leq i <r$ are introduced in (\ref{eqn 1.4}) and $\nu_{L^o}$ is the norm on $k_\infty^r$ associated to $L^o = O_\infty^r$.
Let $g_{z_\AA,i} := (g_{z,i},g^\infty) \in \GL_r(k_\infty) \times \GL_r(\AA^\infty) = \GL_r(\AA)$ for $0\leq i < r$.
By Lemma~\ref{lem 2.4}, we obtain that
\begin{eqnarray}\label{eqn 2.2}
\EE(z_\AA,s;\varphi^\infty)  &=& (1-q_\infty^{-rs}) \cdot \sum_{i=0}^{r-1}c_{z,i}(s) \cdot \Ecal^\infty (g_{z_\AA,i},s;\varphi^\infty).
\end{eqnarray}
Therefore the result holds.
\end{proof}

\begin{rem}\label{rem 2.8}
(1) It is straightforward that $\EE(\gamma \cdot z_\AA,s;\varphi^\infty) = \EE(z_\AA,s;\varphi^\infty)$ for every $\gamma \in \GL_r(k)$.\\
(2) The equality~(\ref{eqn 2.2}) says that the following modified Eisenstein series
$$\widetilde{\EE}(z_\AA,s;\varphi^\infty) := (1-q_\infty^{-rs})^{-1} \cdot \EE(z_\AA,s;\varphi^\infty)$$
can be viewed as an extension function of the finite mirabolic Eisenstein series $\Ecal^\infty (g,s;\varphi^\infty)$ to $\Hfk_\AA^r$ via the building map $\lambda_\AA$.
In particular, for $\chi \in \widehat{\Ical}^\infty$ we let
$$\EE(z_\AA,s;\chi,\varphi^\infty) := \int_{k^\times \backslash \AA^{\infty,\times}} \overline{\chi(a^\infty)} \cdot \EE (z_\AA \cdot a^\infty,s;\varphi^\infty) d^\times a$$
and
$$\widetilde{\EE}(z_\AA,s;\chi,\varphi^\infty):= (1-q_\infty^{-rs})^{-1} \cdot \EE(z_\AA,s;\chi,\varphi^\infty).
$$ 
Then given $[g] \in \GL_r(\AA)/k_\infty^\times \GL_r(O_\infty) \cong \Bcal^r_\AA(\ZZ)$, from the equality~(\ref{eqn 2.2}) and \textit{Remark}~\ref{rem 2.5} (2) one has
$$\Ecal(g,s;\chi,\varphi^\infty \otimes \mathbf{1}_{L^o}) = \widetilde{\EE}(z_\AA,s;\chi,\varphi^\infty)$$
for every $z_\AA \in \Hfk_\AA^r$ with $\lambda_\AA(z_\AA) = [g]$.\\
(3) Theorem~\ref{thm 2.2} implies the following \lq\lq weak\rq\rq\ functional equations:
given $z_\AA, z_\AA' \in \Hfk_\AA^r$ with $\lambda_\AA(z_\AA) = {}^t \lambda_\AA(z_\AA')^{-1} \in \Bcal^r_\AA(\ZZ)$ and $\chi \in \widehat{\Ical}^\infty$, we have
$$\widetilde{\EE}(z_\AA,s;\varphi^\infty) = q_\infty^{-r\delta_\infty (s-1/2)} \cdot \widetilde{\EE}(z_\AA',1-s;\widehat{\varphi}^\infty)$$
and
$$
\widetilde{\EE}(z_\AA,s;\chi,\varphi^\infty) = q_\infty^{-r\delta_\infty (s-1/2)} \cdot \widetilde{\EE}(z_\AA',1-s;\chi^{-1},\widehat{\varphi}^\infty),$$
where $\widehat{\varphi}^\infty$ is the Fourier transform of $\varphi^\infty$, and $\delta_\infty$ is the conductor of $\psi$ at $\infty$ (cf. Section~\ref{sec 1.1}).
\end{rem}

\begin{ex}\label{ex 2.9}
Note that $ (\AA^\infty)^r = {\varprojlim}_{Y'} (k^r/Y')$.
For every $\varphi^\infty \in S((\AA^\infty)^r)$, there exists a projective $A$-module $Y$ of rank $r$ in $k^r$ so that 
$\varphi^\infty$ corresponds to a function $D_{\varphi^\infty}$ on $k^r/Y$ with finite support.
\begin{itemize}
\item[(1)] Let $Y$ be a projective $A$-module of rank $r$ in $k^r$. Take $\varphi_Y^\infty \in S((\AA^\infty)^r)$ corresponding to the characteristic function of $0+Y \in k^r/Y$. For $z = (z_1:\cdots : z_{r-1} : 1) \in \Hfk^r$, we have
\begin{eqnarray}
 \EE\big((z,1),s;\varphi_Y^\infty\big) &=& \sum_{0 \neq (c_1,...,c_{r-1},d) \in Y} \frac{ \im(z)^s}{|c_1z_1+\cdots + c_{r-1} z_{r-1} + d|_\infty^{rs}} \nonumber \\
&=& \sum_{0 \neq \lambda \in \Lambda^Y_z}\frac{\im(z)^s}{|\lambda|_\infty^{rs}} = \EE^{Y}(z,s). \nonumber
\end{eqnarray}
\item[(2)] Let $Y$ be a projective $A$-module of rank $r$ in $k^r$. Given $z = (z_1:\cdots : z_{r-1}: 1) \in \Hfk^r$ and $w \in \CC_\infty - \Lambda^Y_z$, consider the the following Eisenstein series of \lq\lq Jacobi-type\rq\rq:
$$\EE^Y(z,w,s):= \sum_{\lambda \in \Lambda^Y_z} \frac{\im(z)^s}{|\lambda - w|_\infty^{rs}},$$
which satisfies
$$\EE^Y\left(\gamma \cdot z, \frac{w}{j(\gamma,z)},s\right) = \EE^Y(z,w,s), \quad \forall \gamma =\begin{pmatrix}*&*\\ c_1 \ \cdots \ c_{r-1} &d\end{pmatrix} \in \GL(Y).$$
The following equality
\begin{eqnarray}\label{eqn 2.3}
\ \ \ \EE^Y(z,w,s) - \EE^Y(z,s) &=&\im(z)^s \left( \frac{1}{|w|_{\infty}^{rs}} + \sum_{0 \neq \lambda \in \Lambda^Y_z \atop |\lambda|_\infty \leq |w|_\infty}\left(\frac{1}{|\lambda - w|_\infty^{rs}} - \frac{1}{|\lambda|_\infty^{rs}}\right)\right)
\end{eqnarray}
gives the meromorphic continuation of $\EE^Y(z,w,s)$.
For $\alpha = (\alpha_1,...,\alpha_r) \in k^r - Y$, take $\varphi_{Y,\alpha}^\infty \in S((\AA^\infty)^r)$ corresponding to the characteristic function of $\alpha + Y \in k^r/Y$. 
Then 
$$\EE\big((z,1),s;\varphi_{Y,\alpha}^\infty\big) = \EE^Y(z,\alpha_z,s),$$
where $$\alpha_z := \alpha_1 z_1+ \cdots \alpha_{r-1} z_{r-1} + \alpha_r \in k\cdot \Lambda_z^Y \subset \CC_\infty.$$
\item[(3)] Conversely, given $\varphi^\infty \in S((\AA^\infty)^r)$, we take a projective $A$-module $Y$ of rank $r$ in $k^r$ so that 
$\varphi^\infty$ corresponds to a function $D_{\varphi^\infty}$ on $k^r/Y$ with finite support. Then
for $z \in \Hfk^r$, the following equality holds:
\begin{eqnarray}\label{eqn 2.4}
\EE\big((z,1),s;\varphi^\infty\big)
= D_{\varphi^\infty}(0) \cdot \EE^Y(z,s) + \sum_{0 \neq \alpha \in k^r/Y} D_{\varphi^\infty}(\alpha) \cdot \EE^Y(z,\alpha_z,s).\end{eqnarray}
\end{itemize}
\end{ex}

\subsection{Derivatives of non-holomorphic Eisenstein series}\label{sec 2.4}

The meromorphic continuation of $\EE^Y(z,s)$ can be described explicitly from the following identity:

\begin{lem}\label{lem 2.10}
Given a rank $r$ projective $A$-module $Y \subset k^r$ and $z \in \Hfk^r$, we have
$$E^Y(z,s) = \frac{ \im(z)^s}{|a|_\infty^{rs} - |a|_\infty^r} \cdot \sum_{0 \neq w \in \frac{1}{a} \Lambda_z^Y/\Lambda_z^Y} \left[\frac{1}{|w|_\infty^{rs}} + \sum_{0 \neq \lambda \in \Lambda_z^Y \atop |\lambda|_\infty \leq |w|_\infty} \left(\frac{1}{|\lambda-w|_\infty^{rs}} - \frac{1}{|\lambda|_\infty^{rs}}\right)\right], \quad \forall a \in A-\FF_q.$$
\end{lem}

\begin{proof}
Taking $a \in A-\FF_q$, one has
\begin{eqnarray}
(1-|a|_\infty^{-rs}) \cdot \EE^Y(z,s)
&=& \sum_{\lambda \in \Lambda_z^Y - a\Lambda_z^Y} \frac{\im(z)^s}{|\lambda|_\infty^{rs}} \nonumber \\
&=& \sum_{0 \neq w' \in \Lambda_z^Y/ a\Lambda_z^Y} \left( \sum_{\lambda \in \Lambda_z^Y} \frac{\im(z)^s}{|a\lambda - w'|_\infty^{rs}} \right) \nonumber \\
&=& |a|_\infty^{-rs} \cdot \sum_{0 \neq w \in  \frac{1}{a} \Lambda_z^Y/\Lambda_z^Y} \EE^Y(z,w,s). \nonumber 
\end{eqnarray}
Thus from the equality~(\ref{eqn 2.3}) we get
\begin{eqnarray}
(|a|_\infty^{rs}-1) \cdot \EE^Y(z,s)
&=& \sum_{0 \neq w \in  \frac{1}{a} \Lambda_z^Y/\Lambda_z^Y} \EE^Y(z,w,s) \nonumber \\
&=& (|a|_\infty^r-1) \cdot \EE^Y(z,s) \nonumber \\
&& +\im(z)^s  \sum_{0 \neq w \in  \frac{1}{a} \Lambda_z^Y/\Lambda_z^Y}
 \left( \frac{1}{|w|_{\infty}^{rs}} + \sum_{0 \neq \lambda \in \Lambda^Y_z \atop |\lambda|_\infty \leq |w|_\infty}\left(\frac{1}{|\lambda - w|_\infty^{rs}} - \frac{1}{|\lambda|_\infty^{rs}}\right)\right). \nonumber
\end{eqnarray}
Therefore
$$
(|a|_\infty^{rs} - |a|_\infty^r) \cdot \EE^Y(z,s) = 
\im(z)^s  \sum_{0 \neq w \in  \frac{1}{a} \Lambda_z^Y/\Lambda_z^Y}
 \left( \frac{1}{|w|_{\infty}^{rs}} + \sum_{0 \neq \lambda \in \Lambda^Y_z \atop |\lambda|_\infty \leq |w|_\infty}\left(\frac{1}{|\lambda - w|_\infty^{rs}} - \frac{1}{|\lambda|_\infty^{rs}}\right)\right)
 $$
and the result follows.
\end{proof}

For $a \in A-\FF_q$, we rewrite $\EE^Y(z,s)$ as
$$
\EE^Y(z,s)= \frac{\im(z)^s}{1- |a|_\infty^{r(1-s)}} \cdot \sum_{0 \neq w \in \Lambda_z^Y/ a\Lambda_z^Y} \left[\frac{1}{|w|_\infty^{rs}} + \sum_{0 \neq \lambda \in a \Lambda_z^Y \atop |\lambda|_\infty \leq |w|_\infty} \left(\frac{1}{|\lambda-w|_\infty^{rs}} - \frac{1}{|\lambda|_\infty^{rs}}\right)\right].$$
One immediately gets:

\begin{cor}\label{cor 2.11}
Given a projective $A$-module $Y$ of rank $r$ in $k^r$ and $z \in \Hfk^r$, for every $n\in\ZZ_{\geq 0}$ we have
\begin{eqnarray}
&& 
\frac{\partial^n}{\partial s^n}\left(\frac{1- |a|_\infty^{r(1-s)}}{\im(z)^s} \cdot \EE^Y(z,s)\right)\bigg|_{s=0} \nonumber \\
&=&(-r)^n  \cdot
\sum_{0 \neq w \in  \Lambda_z^Y/a \Lambda_z^Y} \left[\ln^n |w|_\infty + \sum_{0 \neq \lambda \in a \Lambda_z^Y \atop |\lambda|_\infty \leq |w|_\infty} \Big(\ln^n |\lambda-w|_\infty - \ln^n |\lambda|_\infty\Big)\right], \quad \forall a \in A-\FF_q.\nonumber
\end{eqnarray}
Here $\ln^n (x) := (\ln (x) )^n$ for $x \in \RR_{>0}$.
\end{cor}

\begin{rem}\label{rem 2.12}
The above result can be viewed as an analogue of the formula for the Stieltjes constants occurring in the Laurent expansion of the Riemann zeta function $\zeta_\QQ(s)$ at $s=1$:
$$\zeta_\QQ(s) = \frac{1}{s-1} + \sum_{n=0}^\infty \frac{(-1)^n}{n!} \gamma_n (s-1)^n$$
with 
$$\gamma_n = \lim_{m \rightarrow \infty} \left[\left(\sum_{\ell=1}^m \frac{\ln^n(\ell)}{\ell}\right) - \frac{\ln^{n+1}(m)}{m+1}\right].$$
\end{rem}

From the fact that
\begin{eqnarray}
&& \frac{\partial^n}{\partial s^n}\left(\frac{1 - |a|_\infty^{r(1-s)}}{\im(z)^s} \cdot \EE^Y(z,s)\right)\bigg|_{s=0} \nonumber \\
&=& \sum_{m=0}^n \binom{n}{m} (-1)^m \Big(\ln^m \big(\im(z)\big)-|a|_\infty^r \ln^m \big(\im(z)|a|_\infty^r\big)\Big)
\cdot 
\frac{\partial^{n-m}}{\partial s^{n-m}} \EE^Y(z,s)\bigg|_{s=0}, \nonumber 
\end{eqnarray}
all the derivatives of $\EE^Y(z,s)$ at $s=0$ are determined recursively from Corollary~\ref{cor 2.11}. 
Consequently, we may describe all the derivatives of $\EE(z_\AA,s;\varphi^\infty)$ at $s=0$  from the equaity~(\ref{eqn 2.3}) and (\ref{eqn 2.4}).
It is natural to ask for a \lq\lq modular\rq\rq\ interpretation for all the derivatives in question, which will be explored in future work.

\section{Kronecker limit formula}\label{sec 3}

In this section, we shall present a Kronecker-type limit formula for $\EE(z_\AA,s;\varphi^\infty)$, which connects its first derivative at $s=0$ with Drinfeld-Siegel units.

\subsection{Drinfeld-Siegel units}\label{sec 3.1}

Let $Y\subset k^r$ be a projective $A$-module of rank $r$. Given a $\ZZ$-valued function
$D$ on $k^r/Y$ with finite support, put
$$\mu(D) := \sum_{\alpha \in k^r/Y} D(\alpha).$$ 
The \textit{Drinfeld-Siegel unit} associated to $D$ on $\Hfk^r$ is:
$$u^Y(z;D) := \Delta^Y(z)^{\mu(D)} \cdot \prod_{0 \neq \alpha \in k^r/Y} \exp_{\Lambda_z^Y}(\alpha_z)^{(q_\infty^r -1)D(\alpha)},
 \quad \forall z \in \Hfk^r.$$
Here for $z = (z_1:\cdots: z_{r-1}:1) \in \Hfk^r$, the point $\alpha_z$ is the image of $\alpha$ via the natural isomorphism 
$$
\begin{tabular}{ccl}
$k^r$ & $\cong$ & $k \cdot \Lambda^Y_z \ \subset \CC_\infty$ \\
$(\alpha_1,...,\alpha_r)$ & $\longmapsto$ & $\alpha_z := \alpha_1 z_1 + \cdots + \alpha_{r-1} z_{r-1} + \alpha_r$.
\end{tabular}
$$ 
Let $\mfk \lhd A$ be an ideal so that the support of $D$ is contained in $\mfk^{-1}Y/Y$.
Put $$\Gamma^Y(\mfk) := \ker \big(\GL(Y) \rightarrow \GL(Y/\mfk Y)\big).$$
Then $u^Y(z;D)$ is an invertible rigid analytic function on $\Hfk^r$ satisfying
$$u^Y(\gamma z;D) = j(\gamma,z)^{(q_\infty^r-1)D(0)} \cdot u(z;D), 
\quad \forall \gamma  \in \Gamma^Y(\mfk).$$
To see the above transformation law, notice that for $0 \neq \alpha \in \mfk^{-1}Y/Y$ and $\gamma \in \Gamma^Y(\mfk)$, we can derive from the equality~(\ref{eqn 1.2}) that
$$\exp_{\Lambda_{\gamma z}^Y}(\alpha_{\gamma z}) = j(\gamma,z)^{-1} \cdot \exp_{\Lambda_z^Y}(\alpha_z).$$
Together with the transformation law of $\Delta^Y$ the result follows.

\begin{lem}\label{lem 3.1}
Let $Y\subset k^r$ be a projective $A$-module of rank $r$.
Let $D$ be a $\ZZ$-valued function on $k^r/Y$ with finite support.
Given a projective $A$-modules $Y'$ of rank $r$ in $k^r$ with $Y' \subset Y$,
let $p$ be the canonical map from $k^r/Y'$ to $k^r/Y$. Then we have
$$u^{Y'}(z;{D \circ p}) = u^Y(z;D), \quad \forall z \in \Hfk^r.$$
\end{lem}

\begin{proof}
This follows directly from the \lq\lq norm compatibilities\rq\rq\ in Lemma~\ref{lem 1.5}.
\end{proof}

From the identification ${\varprojlim}_Y (k^r/Y) = (\AA^\infty)^r$, every $\ZZ$-valued Schwartz function $\varphi^\infty$ on $(\AA^\infty)^r$ can be viewed as a $\ZZ$-valued function $D_{\varphi^\infty}$ of $k^r/Y$ with finite support for a suitable projective $A$-module $Y$ of rank $r$ in $k^r$.
Set 
$$u(z;\varphi^\infty):= u^Y(z,D_{\varphi^\infty}), \quad \forall z \in \Hfk^r.$$
Then Lemma~\ref{lem 3.1} assures that $u(z;\varphi^\infty)$ is independent of the chosen $Y$.
Recall that $\varrho$ is the left action of $\GL_r(\AA^\infty)$ on the Schwartz space $S((\AA^\infty)^r)$ defined by
$$\big(\varrho(g^\infty)\varphi^\infty\big) (x) := \varphi^\infty(x g^\infty), \quad \forall x \in (\AA^\infty)^r.$$
Let $S((\AA^\infty)^r)_\ZZ$ be the space of $\ZZ$-valued Schwartz function on $(\AA^\infty)^r$.

\begin{defn}\label{defn 3.2}
Given $\varphi^\infty \in S((\AA^\infty)^r)_\ZZ$, the \textit{Drinfeld-Siegel unit associated to $\varphi^\infty$ on $\Hfk_\AA^{r}$} is 
$$\ufk(z_\AA; \varphi^\infty) := u(z, \varrho(g^\infty)\varphi^\infty), \quad \forall z_\AA = (z,g^\infty) \in \Hfk_\AA^{r}.$$
\end{defn}

It is straightforward that
\begin{lem}\label{lem 3.3}
For $z_\AA = (z,g^\infty) \in \Hfk_\AA^r$ and $\gamma \in \GL_r(k)$, one has
$$\ufk(\gamma \cdot z_\AA; \varphi^\infty) = j(\gamma,z)^{(q_\infty^r-1)\varphi^\infty (0)} \cdot \ufk(z_\AA; \varphi^\infty).$$
In particular, the unit $\ufk(\cdot;\varphi^\infty)$ on $\Hfk_\AA^r$ is invariant under the action of $\GL_r(k)$ if $\varphi^\infty(0) = 0$.
\end{lem}

\subsection{Kronecker limit formula}\label{sec 3.2}

We now arrive at:

\begin{thm}\label{thm 3.4}
Take $\varphi^\infty \in S((\AA^\infty)^r)_\ZZ$. For $z_\AA \in \Hfk_\AA^r$ we have
\begin{eqnarray}\label{eqn 3.1}
\EE(z_\AA,0;\varphi^\infty) &=& - \varphi^\infty(0) \quad \text{ and}  \nonumber \\
\frac{\partial}{\partial s} \EE(z_\AA,s;\varphi^\infty)\bigg|_{s=0} &=& - \varphi^\infty(0) \cdot \ln \im(z_\AA) - \frac{r}{q_\infty^r-1} \ln |\ufk(z_\AA;\varphi^\infty)|_\infty. 
\end{eqnarray}
\end{thm}

\begin{proof}
Note that for $z_\AA = (z,g^\infty) \in \Hfk_\AA^r$, one observes that
$$\EE(z_\AA,s;\varphi^\infty) = |g^\infty|_\AA^s \cdot \EE\big((z,1),s;\varrho(g^\infty)\varphi^\infty\big).$$
From the equality~(\ref{eqn 2.4}), it then suffices to show the equalities~(\ref{eqn 3.2}) and (\ref{eqn 3.3}) below:
\begin{eqnarray}\label{eqn 3.2}
\EE^Y(z,0) = -1, \quad  \frac{\partial}{\partial s} \EE^Y(z,s)\Big|_{s=0} 
= 
- \ln \im(z) - \frac{r}{q_\infty^r-1} \cdot \ln |\Delta^Y(z)|_\infty
\end{eqnarray}
for every projective $A$-module $Y$ of rank $r$ in $k^r$ and $z \in \Hfk^r$, and
\begin{eqnarray}\label{eqn 3.3}
&& \EE^Y(z,w,0) = 0, \quad 
\frac{\partial}{\partial s} \EE^Y(z,w,s)\Big|_{s=0}
= - \frac{r}{q_\infty^r-1} \cdot \ln |\Delta^Y(z)|_\infty - r \ln |\exp_{\Lambda_z^Y}(w)|_\infty.
\end{eqnarray}
for every $w \in \CC_\infty - \Lambda^Y_z$.
Note that the equality~(\ref{eqn 3.2}) follows from Lemma~\ref{lem 1.3} and Corollary~\ref{cor 2.11} (for the case of $n=0$ and $n=1$). To show the equality~(\ref{eqn 3.3}), from the equaity~(\ref{eqn 2.3}) we get
$\EE^Y(z,w;0) = 0$ and
\begin{eqnarray}
&&\frac{\partial}{\partial s} \EE^Y(z,w,s)\Big|_{s=0} \nonumber \\
&=& \frac{\partial}{\partial s} \EE^Y(z,s)\Big|_{s=0}
+ \ln \im(z) - r \cdot \left(\ln |w|_\infty + \sum_{\subfrac{0 \neq \lambda \in \Lambda_z^Y}{|\lambda|_\infty \leq |w|_\infty}} \ln |\lambda-w|_\infty - \ln |\lambda|_\infty\right) \nonumber \\
&=& - \frac{r}{q_\infty^r-1} \cdot \ln |\Delta^Y(z)|_\infty - r \cdot \ln\Big|w \cdot\prod_{0\neq \lambda \in \Lambda_z^Y}(1-\frac{w}{\lambda})\Big|_\infty \nonumber \\
&=& - \frac{r}{q_\infty^r-1} \cdot \ln |\Delta^Y(z)|_\infty - r \ln |\exp_{\Lambda_z^Y}(w)|_\infty. \nonumber
\end{eqnarray}
The second equality in the above comes from the fact that $|1-w/\lambda|_\infty = 1$ if $|\lambda|_\infty > |w|_\infty$.
\end{proof}

\begin{rem} \label{rem 3.5}
${}$
\begin{itemize}
\item[(1)] We may view the equality~(\ref{eqn 3.2}) and (\ref{eqn 3.3}) as an analogue of the classical Kronecker first and second limit formulas, respectively.
\item[(2)] Given a  rank $r$ $A$-lattice $\Lambda \subset \CC_\infty$, take $\alpha \in \CC_\infty^\times$, $z \in \Hfk^r$, and a projective $A$-module $Y$ of rank $r$ in $k^r$ so that 
$\Lambda = \alpha \cdot \Lambda^Y_z$.
Define $$\EE(\Lambda,s):= \sum_{0 \neq \lambda \in \Lambda}\left(\frac{D_A(\Lambda)}{|\lambda|_\infty}\right)^{rs} = \Vert Y\Vert^s \cdot \EE^Y(z,s) \quad \text{(from the equation (\ref{eqn 1.6}))}$$
and for $w \in \CC_\infty - \Lambda$ we put
$$\EE(\Lambda,w,s) := \sum_{\lambda \in \Lambda} \left(\frac{D_A(\Lambda)}{|\lambda - w|_\infty}\right)^{rs} = \Vert Y \Vert^s \cdot \EE^Y(z,w/\alpha,s)  \quad \text{(from the equation (\ref{eqn 1.6}))},$$
which both have meromorphic continuation with $\EE(\Lambda,0) = -1$ and $\EE(\Lambda,w,0) = 0$.
On the other hand, let $\rho^{\Lambda}$ be the rank $r$ Drinfeld $A$-module over $\CC_\infty$ associated to $\Lambda$.
Writing $\rho^{\Lambda}_a(x) = ax+\cdots + \Delta_a(\Lambda) x^{q^{r\deg a}}$ for each $a\in A$,
one has
$$\Delta_a(\Lambda) = \alpha^{1-q^{r\deg a}} \Delta_a^Y(z).$$
Let $\Delta(\Lambda):= \alpha^{1-q_\infty^r} \cdot \Delta^Y(z)$. The equality~(\ref{eqn 3.2}) and (\ref{eqn 3.3}) can be reformulated to:
\begin{eqnarray}\label{eqn 3.4}
\frac{\partial}{\partial s} \EE(\Lambda,s)\bigg|_{s=0} &=& -r\left(\ln\big(D_A(\Lambda) \big) + \frac{1}{q_\infty^r-1} \ln |\Delta(\Lambda)|_\infty\right), \\ \label{eqn 3.5}
\frac{\partial}{\partial s} \EE(\Lambda,w,s)\bigg|_{s=0} &=& 
- r \left(\frac{1}{q_\infty^r-1} \cdot \ln |\Delta(\Lambda)|_\infty + \ln |\exp_{\Lambda}(w)|_\infty\right). 
\end{eqnarray}
\item[(3)] The equality~(\ref{eqn 3.5}) agrees with Kondo's formula in \cite[Theorem 1 and Proposition 4]{Kon} (without the factor $D_A(\Lambda)$ in the definition of $\EE(\Lambda,w,s)$).
\end{itemize}
\end{rem}

\subsection{Lerch-type formula}\label{sec 3.3}

Let $\varphi^\infty \in S((\AA^\infty)^r)_\ZZ$ and $\chi \in \widehat{\Ical}^\infty$.
Recall in \textit{Remark}~\ref{rem 2.8} (2) that
$$\EE(z_\AA,s;\chi,\varphi^\infty) = \int_{k^\times \backslash \AA^{\infty,\times}} \overline{\chi(a^\infty)} \cdot \EE(z_\AA a^\infty,s;\varphi^\infty) d^\times a^\infty, \quad \forall z_\AA \in \Hfk_\AA^r.$$
Thus Theorem~\ref{thm 3.4} leads directly to the following Lerch-type formula of $\EE(z_\AA,s;\chi,\varphi^\infty)$:

\begin{cor}\label{cor 3.6}
Given $\varphi^\infty \in S((\AA^\infty)^r)_\ZZ$ and $\chi \in \widehat{\Ical}^\infty$, we have
\begin{eqnarray}
&& \EE(z_\AA,s;\chi,\varphi^\infty) \nonumber \\
&=& -\varphi^\infty(0) \cdot \int_{k^\times \backslash \AA^{\infty,\times}} \overline{\chi(a^\infty)} d^\times a^\infty \nonumber \\
&& -
\left[\frac{r}{q_\infty^r-1} \cdot \int_{k^\times \backslash \AA^{\infty,\times}} \overline{\chi(a^\infty)} \cdot \left(
\frac{(q_\infty^r-1)\varphi^\infty(0)}{r} \ln \im(z_\AA a^\infty) + \ln \big| \ufk(z_\AA a^\infty ,\varphi^\infty)\big|_\infty\right) d^\times a^\infty\right] \cdot s  \nonumber \\
&& + \text{\rm O}(s^2), \quad \text{ as $s \rightarrow 0$.}\nonumber
\end{eqnarray}
In particular, when $\varphi^\infty(0) = 0$ or $\chi \neq 1$, we get $\EE(z_\AA,0;\chi,\varphi^\infty) = 0$ and
\begin{eqnarray}
&& \frac{\partial}{\partial s} \EE (z_\AA,s;\chi,\varphi^\infty)\bigg|_{s=0} \nonumber \\
&=&
-\frac{r}{q_\infty^r-1} \cdot \int_{k^\times \backslash \AA^{\infty,\times}} \overline{\chi(a^\infty)} \cdot \left(
\frac{(q_\infty^r-1)\varphi^\infty(0)}{r} \ln \im(z_\AA a^\infty) + \ln \big| \ufk(z_\AA a^\infty ,\varphi^\infty)\big|_\infty\right) d^\times a^\infty. \nonumber
\end{eqnarray}
\end{cor}

We then immediately get a Lerch-type formula for mirabolic Eisenstein series. 
More precisely, for $\varphi^\infty \in S((\AA^\infty)^r)_\ZZ$ and $\chi \in \widehat{\Ical}^\infty$, recall in the equality~(\ref{eqn 2.2}) and \text{Remark}~\ref{rem 2.8} (2) that
\begin{eqnarray}\label{eqn 3.6}
&&\Ecal^\infty(g,s;\varphi^\infty) = \frac{\EE(z_\AA,s;\varphi^\infty)}{1-q_\infty^{-rs}} \quad \text{and} \quad \Ecal(g,s;\chi,\varphi^\infty\otimes \mathbf{1}_{L^o}) = \frac{\EE(z_\AA,s;\chi,\varphi^\infty)}{1-q_\infty^{-rs}},
\end{eqnarray}
where $z_{\AA} \in \Hfk_\AA^r$ with $\lambda_\AA(z_\AA) = [g] \in \Bcal^r_\AA(\ZZ)$.
Observe that
\begin{eqnarray}\label{eqn 3.7}
\frac{1}{1-q_\infty^{-rs}} = \frac{1}{r\ln q_\infty} \cdot s^{-1} + \frac{1}{2} + \text{O}(s), \quad \text{ as $s \rightarrow 0$.}
\end{eqnarray}
Define 
$$\eta(g;\varphi^\infty) := \frac{(q_\infty^r-1)\varphi^\infty(0)}{r} \log_{q_\infty} \im(z_\AA) + \log_{q_\infty} | \ufk(z_\AA;\varphi^\infty)|_\infty \quad \text{ with $\lambda_\AA(z_\AA) = [g] \in \Bcal^r_\AA(\ZZ)$},$$
and set
$$\eta_{\chi}(g;\varphi^\infty):= \int_{k^\times \backslash \AA^{\infty,\times}} \overline{\chi(a^\infty)} \cdot 
\eta(ga^\infty;\varphi^\infty) d^\times a^\infty.
$$
From the equalities~(\ref{eqn 3.6}) and (\ref{eqn 3.7}), Theorem~\ref{thm 3.4} and Corollary~\ref{cor 3.6} say:

\begin{cor}\label{cor 3.7}
Let $\varphi^\infty \in S((\AA^\infty)^r)_\ZZ$ and $\chi \in \widehat{\Ical}^\infty$.
For $g \in \GL_r(\AA)$ we have
$$\Ecal^\infty(g,s;\varphi^\infty) = -\frac{\varphi^\infty(0)}{r\ln q_\infty} s^{-1} 
- \left(\frac{\varphi^\infty(0)}{2} + \frac{1}{q_\infty^r-1} \eta(g;\varphi^\infty)\right) + \text{\rm O}(s) \quad \text{ as $s \rightarrow 0$}$$
and
\begin{eqnarray}
&&\Ecal(g,s,\chi;\varphi^\infty\otimes \mathbf{1}_{L^o}) \nonumber \\
&=&
-\left(\frac{\varphi^\infty(0)}{r \ln q_\infty} \int_{k^\times \backslash \AA^{\infty,\times}} \overline{\chi(a^\infty)} d^\times a^\infty\right) \cdot s^{-1} \nonumber \\
&& -\left(\frac{\varphi^\infty(0)}{2} \cdot
\int_{k^\times \backslash \AA^{\infty,\times}} \overline{\chi(a^\infty)} d^\times a^\infty + \frac{1}{q_\infty^r-1} \cdot \eta_\chi(g;\varphi^\infty)\right) \nonumber \\
&& + \text{\rm O}(s) \quad \text{ as $s \rightarrow 0$.} \nonumber
\end{eqnarray}
In particular, when $\varphi^\infty(0) = 0$ or $\chi \neq 1$, we get
$$\Ecal(g,0,\chi;\varphi^\infty\otimes \mathbf{1}_{L^o}) = 
\frac{1}{1-q_\infty^r} \cdot \eta_{\chi}(g;\varphi^\infty).
$$
\end{cor}

\begin{rem}\label{rem 3.8}
(1) Recall that our normalization of the Haar measure $d^\times a$ on $\AA^{\infty,\times}$ implies (cf.\ equation~(\ref{eqn Ivol}))
$$\int_{k^\times \backslash \AA^{\infty,\times}} \overline{\chi(a^\infty)} d^\times a^\infty
= \begin{cases}
\#\Pic(A)/(q-1), & \text{if $\chi = 1 \in \widehat{\Ical}^\infty$,} \\
0, & \text{otherwise.}
\end{cases}
$$
(2) When $r=1$, the Drinfeld period domain $\Hfk^1$ only consists of a single point, say $z_o$. 
Take $\varphi^\infty = \mathbf{1}_{1+\widehat{\mfk}}$ where $\mfk$ is a non-zero proper ideal of $A$ and $\widehat{\mfk}$ is the closure of $\mfk$ in $\AA^{\infty}$.
It is known that
$$\ufk_\mfk := \ufk((z_o,1);\mathbf{1}_{1+\widehat{\mfk}}) \in H_\mfk^\times,$$
where $H_\mfk^\times$ is the ray class field of $k$ of conductor $\mfk$.
In fact, $\ufk_\mfk$ is an \lq\lq elliptic unit\rq\rq\ introduced by
Hayes in \cite{Ha3}, and the explicit class field theory says
$$
\ufk((z_0,a);\mathbf{1}_{1+\widehat{\mfk}}) = \ufk_\mfk^{\sigma_a}, \quad \forall a \in \AA^{\infty,\times},
$$
where $\sigma_a \in \gal(H_\mfk/k)$ is the corresponding automorphism on $H_\mfk$ via the Artin map.
Then Corollary~\ref{cor 0.2} coincides with Hayes' formula of the Dirichlet $L$-function $L(s,\chi)$ for every character $\chi$ of conductor $\mfk$ (cf.\ \cite[p.\ 238]{Ha2}):
$$L(0,\chi) = \frac{1}{1-q^{\deg \infty}} \cdot \sum_{\sigma \in \gal(H_\mfk/k)} \overline{\chi(\sigma)} \log_{q^{\deg \infty}}|\ufk_\mfk^\sigma|_\infty,
$$
and the Hayes-Stickelberger element is
$$\omega_\mfk = -\sum_{\sigma \in \gal(H_\mfk/k)} \log_{q^{\deg \infty}}|\ufk_\mfk^\sigma|_\infty \cdot \sigma^{-1} \in \ZZ[\gal(H_\mfk/k)].$$
(3) Corollary~\ref{cor 3.7} can be stated in the following variant form:
$$(1-q_\infty^{-rs}) \cdot \Ecal^\infty(g,s;\varphi^\infty) = -\varphi^\infty(0)
- \left(\frac{r\ln q_\infty }{q_\infty^r-1} \cdot \eta(g;\varphi^\infty)\right) \cdot s + \text{\rm O}(s^2) \quad \text{ as $s \rightarrow 0$}$$
and
\begin{eqnarray}
&& (1-q_\infty^{-rs}) \cdot \Ecal(g,s,\chi;\varphi^\infty\otimes \mathbf{1}_{L^o}) \nonumber \\
&=&
-\left(\varphi^\infty(0) \cdot \int_{k^\times \backslash \AA^{\infty,\times}} \overline{\chi(a^\infty)} d^\times a^\infty\right)  - \left(\frac{r \ln q_\infty}{q_\infty^r-1} \cdot \eta_\chi(g;\varphi^\infty) \right) \cdot s + \text{\rm O}(s^2) \quad \text{ as $s \rightarrow 0$.} \nonumber
\end{eqnarray}
This will be used in Section~\ref{sec Z-E}.
\end{rem}

\section{Colmez-type formula for CM Drinfeld modules}\label{sec 4}


\subsection{Taguchi height of Drinfeld modules}\label{sec 4.1}

Let $F$ be a finite extension of $k$ (viewing as an $A$-field via $A \hookrightarrow F$) 
and $\rho$ be a Drinfeld $A$-module of rank $r$ over $F$.
Recall that for each $a \in A$, 
we write 
$$\rho_a = a+ \sum_{i=1}^{r \deg a} l_i(\rho_a) \tau^i \in F\{\tau\}.$$

For each place $w$ of $F$ with $w \nmid \infty$, put
$$\ord_w(\rho):= \min\left\{ \frac{\ord_w(l_i(\rho_a))}{q^i-1}: 0 \neq a \in A, \ 1\leq i \leq r\deg a \right\}.$$ 
Let $\FF_w$ be the residue field of $w$. 
The local height of $\rho$ at $w$ is given by
$$h_{\text{Tag},w}(\rho/F):= -[\FF_w:\FF_q] \cdot \lfloor\ord_{w}(\rho)\rfloor.$$

For each place $\tilde{\infty}$ of $F$ with $\tilde{\infty}\mid\infty$, we embed $F$ into $\CC_\infty$ via $\tilde{\infty}$ and let $\Lambda_{\rho,\tilde{\infty}}\subset \CC_\infty$ be the rank $r$ $A$-lattice in $\CC_\infty$ associated to $\rho$.
We set 
$$h_{\text{Tag},\tilde{\infty}}(\rho/F):= -[F_{\tilde{\infty}}: k_\infty] \cdot \log_q D_A( \Lambda_{\rho,\tilde{\infty}}).$$

\begin{defn}\label{defn 4.1}
(cf.\ \cite[Section 5]{Tag}) The \textit{Taguchi height of $\rho/F$} is defined by
$$h_{\text{Tag}}(\rho/F):= \frac{1}{[F:k]} \cdot \left(\sum_{w \nmid \infty} h_{\text{Tag},w}(\rho/F) + \sum_{\tilde{\infty}\mid \infty} h_{\text{Tag},\tilde{\infty}}(\rho/F)\right).$$
\end{defn}

\begin{rem}\label{rem 4.2}
$(1)$ Let $F'$ be a finite extension over $F$. Given places $\tilde{\infty}$ of $F$ and $\tilde{\infty}'$ of $F'$ with $\tilde{\infty}' \mid \tilde{\infty} \mid \infty$, it is clear that $\Lambda_{\rho,\tilde{\infty}} = \Lambda_{\rho,\tilde{\infty}'} \subset \CC_\infty$, and
$$h_{\text{Tag},\tilde{\infty}'}(\rho/F') = [F'_{\tilde{\infty}'}: F_{\tilde{\infty}}] \cdot h_{\text{Tag},\tilde{\infty}}(\rho/F).$$
For places $w$ of $F$ and $w'$ of $F'$ with $w'\mid w \nmid \infty$, one has $\ord_{w'}(\rho) = e_{w'/w} \cdot \ord_{w}(\rho)$, where $e_{w'/w}$ is the ramification index of $w' / w$. Thus we get
$$h_{\text{Tag},w'}(\rho/F') \leq [F'_{w'}:F_w] \cdot h_{\text{Tag},w}(\rho/F).$$
In particular, if $\rho$ has stable reduction at $w$, then $\ord_w(\rho)$ is an integer, which implies that
$h_{\text{Tag},w'}(\rho/F') = [F'_{w'}:F_w] \cdot h_{\text{Tag},w}(\rho/F)$. In conclusion, we have  $h_{\text{Tag}}(\rho/F') \leq h_{\text{Tag}}(\rho/F)$, and the equality holds when $\rho$ has stable reduction everywhere.

$(2)$
Note that every Drinfeld $A$-module $\rho$ over $F$ has potentially stable reduction everywhere (cf.\ \cite[Proposition 7.2]{Ha1}).
Define
$$h^{\text{st}}_{\text{Tag}}(\rho) := \ln q \cdot \lim_{F': \atop [F':F]<\infty} h_{\text{Tag}}(\rho/F'),$$
which is always convergent by $(1)$, called the \textit{stable Taguchi height of $\rho$}.
In particular, we may express the stable Taguchi height of $\rho$ by:
\begin{eqnarray}\label{eqn: hst}
h^{\text{st}}_{\text{Tag}}(\rho) &=& \frac{-\ln q}{[F:k]} \cdot \left(
\sum_{w \nmid \infty} [\FF_w:\FF_q] \ord_w(\rho)
+ \sum_{\tilde{\infty} \mid \infty} [F_{\tilde{\infty}} : k_\infty] \ln D_A(\Lambda_{\rho,\tilde{\infty}})\right).
\end{eqnarray}

$(3)$ Let $\rho$ be a Drinfeld $A$-module of rank $r$ over a finite extension $F$ of $k$.
Similar to Proposition~\ref{prop 1.4}, there exists a unique element $\Delta(\rho) \in F^\times$, up to multiplying with $(q_\infty^r-1)$-th roots of unity, so that
$$\Delta(\rho)^{q^{r\deg a} -1} = l_{r\deg a}(\rho_a)^{q_\infty^r-1}, \quad \forall a \in A-\{0\}.$$
In particular, for each place $\tilde{\infty}$ of $F$ lying above $\infty$, we have the following identity (up to multiplying with $(q_\infty^r-1)$-th roots of unity):
$$\Delta(\Lambda_{\rho,\tilde{\infty}}) = \iota_{\tilde{\infty}}(\Delta(\rho))$$
where $\iota_{\tilde{\infty}} : F \hookrightarrow \CC_\infty$ is the embedding corresponding to $\tilde{\infty}$.
Set
$$c(\rho) := -\frac{\ln q}{[F:k]} \cdot \sum_{w \nmid \infty} [\FF_w : \FF_q] \Big(\ord_w(\rho) - \frac{\ord_w(\Delta(\rho))}{q_\infty^r-1}\Big),$$
which is actually a finite sum. We refer $c(\rho)$ to a {\it conductor quantity of $\rho$}.
On the other hand, put 
\begin{eqnarray}
h^{\text{disc}}_{\text{Tag}}(\rho) &:=& \frac{-1}{[F:k]} \sum_{\tilde{\infty} \mid \infty} [F_{\tilde{\infty}} : k_\infty] \cdot \ln \Big(|\Delta(\Lambda_{\rho,\tilde{\infty}})|_\infty^{\frac{1}{q_\infty^r-1}} \cdot D_A(\Lambda_{\rho,\tilde{\infty}})\Big). \nonumber 
\end{eqnarray}
Note that $c(\rho)$ and $h^{\text{disc}}_{\text{Tag}}(\rho)$ are both independent of the chosen defining field $F$ of $\rho$.
Moreover, from the expression of $h^{\text{st}}_{\text{Tag}}(\rho)$ in (\ref{eqn: hst}) we have
\begin{eqnarray}
&& c(\rho) + h^{\text{disc}}_{\text{Tag}}(\rho) \nonumber \\
&=& \frac{-\ln q}{[F:k]} \sum_{w \nmid \infty} \left([\FF_w:\FF_q] \ord_w(\rho) + \frac{1}{q_\infty^r-1} \log_q \Vert\Delta(\rho)\Vert_{F,w}\right) \nonumber \\
&& + \frac{-\ln q}{[F:k]} \sum_{\tilde{\infty}\mid \infty} \left([F_{\tilde{\infty}}:k_\infty] D_A(\Lambda_{\rho,\tilde{\infty}}) + \frac{1}{q_\infty^r-1} \cdot \log_q\Vert\Delta(\rho)\Vert_{F,\tilde{\infty}}\right) \nonumber \\
&=& h^{\text{st}}_{\text{Tag}}(\rho). \nonumber 
\end{eqnarray}
Here for a place $w$ of $F$ lying above a place $v$ of $k$, we put $\Vert\alpha\Vert_{F,w} := |N_{F_w/k_v}(\alpha)|_v$ for every $\alpha \in F_w$.
The last equality comes from the product formula:
$$\prod_{\text{place $w$ of $F$}} \Vert \Delta(\rho) \Vert_{F,w} = 1.$$

$(4)$ Let $\rho$ and $\rho'$ be two Drinfeld $A$-modules over $\overline{k}$ where $\overline{k}$ is the algebraic closure of $k$ in $\CC_\infty$. 
Suppose $\rho$ and $\rho'$ are isomorphic over $\overline{k}$. 
Then
$h^{\text{st}}_{\text{Tag}}(\rho) = h^{\text{st}}_{\text{Tag}}(\rho')$.
In fact, we have 
$$c(\rho) = c(\rho') \quad \text{ and } \quad h^{\text{disc}}_{\text{Tag}}(\rho) = h^{\text{disc}}_{\text{Tag}}(\rho').$$
\end{rem}

\subsection{CM Drinfeld $A$-modules}\label{sec 4.2}

Given a Drinfeld $A$-module $\rho$ of rank $r$ over $\overline{k}$, it is known that the endomorphism ring $\End_A(\rho/\overline{k})$ can be identified with an $A$-order $\Ocal$ of an \lq\lq imaginary\rq\rq\ field $K$ with $[K:k] \mid r$. Here \lq\lq imaginary\rq\rq\ means that the place $\infty$ of $k$ does not split in $K$.  We say that $\rho$ is CM if $[K:k] = r$. 

To calculate the stable Taguchi height of a Drinfeld module $\rho$ of rank $r$ with CM by $\Ocal$, an $A$-order of an imaginary field $K$ with $[K:k] = r$, we may assume that $\rho$ is defined over the ring class field $H_\Ocal$ of $\Ocal$ (cf.\ Theorem~\ref{thm A.1} (1)). 
We point out that the unique place of $K$ lying above $\infty$ is split completely in $H_\Ocal$.
Fix an embedding of $H_\Ocal \hookrightarrow \CC_\infty$. Let $\Lambda_\rho \subset \CC_\infty$ be the $A$-lattice associated to $\rho$. Then there exists an ideal $\Ifk$ of $\Ocal$ and $\alpha \in \CC_\infty^{\times}$ such that $\Lambda_\rho = \alpha \cdot \Ifk$.
From the fact that $\rho$ has good reduction at every finite places of $H_{\Ocal}$,
we have $c(\rho) = 0$ and
\begin{eqnarray}\label{eqn 4.1}
h^{\text{st}}_{\text{Tag}}(\rho) &=&  h^{\text{disc}}_{\text{Tag}}(\rho) \nonumber \\
&=&
\frac{-1}{[H_\Ocal: k]} \cdot \sum_{[\Acal] \in \Pic(\Ocal)}[K_\infty:k_\infty] \ln \big(D_A(\Lambda_{\Acal*\rho})|\Delta(\Lambda_{\Acal*\rho})|_\infty^{\frac{1}{q_\infty^r-1}}\big) \nonumber \\
&=&   \frac{- 1}{\#\Pic(\Ocal)}\sum_{[\Acal] \in \Pic(\Ocal)} \ln \big(D_A(\Acal^{-1} \Ifk ) \cdot |\Delta(\Acal^{-1} \Ifk )|_\infty^{\frac{1}{q_\infty^r-1}}\big)\nonumber \\
&=& - \ln\big(D_A(\Ocal)\big) - \frac{1}{r\cdot\# \Pic(\Ocal)} \sum_{[\Acal] \in \Pic(\Ocal)} \ln \big(N_{\Ocal}(\Acal \Ifk ) \cdot |\Delta( \Acal \Ifk)|_\infty^{\frac{r}{q_\infty^r-1}}\big). 
\end{eqnarray}
Here $\Pic(\Ocal)$ is the class group of the invertible ideals of $\Ocal$, which is isomorphic to the Galois group $\gal(H_\Ocal/K)$ via the Artin map; and $$N_{\Ocal}(\Acal):= \frac{\#(\Ocal/a \Acal)}{|a|_\infty^r}$$ for $0 \neq a \in A$ such that $a \Acal \subset \Ocal$, which is independent of the chosen $a$.
The second equality follows from Theorem~\ref{thm A.1} (2), and the third equality comes from the fact that the lattice $\Lambda_{\Acal*\rho}$ is equal to $\alpha_\Acal \cdot \Acal^{-1} \Ifk$ for some $\alpha_\Acal \in \CC_\infty^\times$ (cf.\ \textit{Remark}~\ref{rem A.2}).
\\

On the other hand, let $\zeta_\Ifk(s)$ be the zeta function associated to the ideal $\Ifk$:

\begin{eqnarray}
\zeta_{\Ifk}(s) &:=& N_\Ocal(\Ifk)^s \cdot \sum_{\text{invertible ideal $I$ of $\Ocal$} \atop I \subset \Ifk} \frac{1}{N_{\Ocal}(I)^s} \nonumber \\
&=& \frac{N_\Ocal(\Ifk)^s}{\#(\Ocal^\times)} \cdot \sum_{[\Acal] \in \Pic(\Ocal)} \left( \sum_{0 \neq \lambda \in \Acal \Ifk} \frac{N_{\Ocal}(\Acal)^s}{|\lambda|_\infty^{rs}} \right)\nonumber \\
&=& \frac{1}{\#(\Ocal^\times)} \cdot \sum_{[\Acal] \in \Pic(\Ocal)}
\frac{N_{\Ocal}(\Acal\Ifk)^s}{D_A(\Acal\Ifk)^{rs}} \cdot 
\EE(\Acal \Ifk,s), \nonumber 
\end{eqnarray}
where $\EE(\Lambda,s)$ is introduced in \textit{Remark}~\ref{rem 3.5} (2) for every rank $r$ $A$-lattice $\Lambda \subset \CC_\infty$.
By our Kronecker limit formula (the version in the equation~(\ref{eqn 3.4})), we have 
\begin{eqnarray}\label{eqn 4.2}
\zeta_{\Ifk}(0) &=& -\frac{\#\Pic(\Ocal)}{\#(\Ocal^{\times})}, \quad \text{ and }\nonumber \\
\zeta_{\Ifk}'(0)
&=& -\frac{1}{\#(\Ocal^{\times})} \cdot  \sum_{[\Acal] \in \Pic(\Ocal)} \ln \big(N_{\Ocal}(\Ifk\Acal) \cdot |\Delta(\Ifk \Acal)|_\infty^{\frac{r}{q_\infty^r-1}}\big).
\end{eqnarray}

Combining the equality~(\ref{eqn 4.1}) and (\ref{eqn 4.2}), we then arrive at:

\begin{thm}\label{thm 4.3}
Given a CM Drinfeld $A$-module $\rho$ of rank $r$ over $\overline{k}$, let $\Ocal = \End_A(\rho/\overline{k})$, which is identified with an $A$-order of an imaginary field $K$ satisfying $[K:k] = r$. 
Then
$$ h^{\text{\rm st}}_{\text{\rm Tag}}(\rho) = -\ln \big(D_A(\Ocal)\big) - \frac{1}{r}\cdot \frac{\zeta_{\Ifk}'(0)}{\zeta_{\Ifk}(0)}.$$
Here $\Ifk$ is an ideal of $\Ocal$ which is isomorphic to the $A$-lattice $\Lambda_\rho$ corresponding to $\rho$ $($as $\Ocal$-modules$)$.
\end{thm}

\begin{rem}\label{rem 4.4}
Suppose $r = 1$ or $2$. Observe that an ideal $\Ifk$ of a quadratic $A$-order $\Ocal$ is invertible if and only if 
$$\{b \in K : b \Ifk \subset \Ifk\} = \Ocal.$$
Thus Theorem~\ref{thm 4.3} coincides with the formula in \cite{Wei2}.
\end{rem}
\begin{lem}\label{lem 4.5}
Let $q_K$ be the cardinality of the constant field of $K$ and $g_K$ be the genus of $K$.
Then
$$D_A(O_K) = q^{1-g_k} \cdot {q_K^{\frac{g_K-1}{r}}} \cdot q_\infty^{\frac{-r+f_K(\infty)}{2r}}.$$
Here $f_K(\infty)$ is the residue degree of $\infty$ in $K$.
\end{lem}

\begin{proof}
Let $\tilde{\infty}$ be the unique place of $K$ lying above $\infty$, and let $O_{K_{\tilde{\infty}}}$ be the valuation ring in $K_{\tilde{\infty}}$, the completion of $K$ at $\tilde{\infty}$. Choose $\Pi_{\tilde{\infty}} \in K$ to be a uniformizer at $\tilde{\infty}$, and denote by $\FF_{\tilde{\infty}} := O_{K_{\tilde{\infty}}}/ \Pi_{\tilde{\infty}} O_{K_{\tilde{\infty}}}$, the residue field at $\tilde{\infty}$.
Take $\xi \in K \cap O_{K_{\tilde{\infty}}}$ so that $\FF_{\tilde{\infty}} = \FF_\infty(\bar{\xi})$, where $\bar{\xi} := \xi \bmod \Pi_{\tilde{\infty}} \in \FF_{\tilde{\infty}}$
Then
$$ O_{K_{\tilde{\infty}}} = O_\infty[\xi,\Pi_{\tilde{\infty}}] = \oplus_{i=0}^{f_K(\infty)-1} \oplus_{j=0}^{e_K(\infty)-1} O_\infty \xi^i \Pi_{\tilde{\infty}}^j,$$
where $e_K(\infty):= [K:k]/f_K(\infty)$, the ramification index of $\infty$ in $K$.
In particular, $|\xi|_\infty = 1$ and $|\Pi_{\tilde{\infty}}|_\infty = q_\infty^{-1/e_K(\infty)}$.
Moreover, it can be checked that given $a_{ij} \in K_\infty$ for $0\leq i < f_K(\infty)$ and $0\leq j < e_K(\infty)$, one has
$$\Bigg|\sum_{i=0}^{f_K(\infty)-1} \sum_{j=0}^{e_K(\infty)-1} a_{ij} \xi^i \Pi_{\tilde{\infty}}^j\Bigg|_\infty = \max\Big( \big|a_{ij} \xi^i \Pi_{\tilde{\infty}}^j\big|_\infty \ \Big|\ 0\leq i < f_K(\infty),\ 0\leq j < e_K(\infty)\Big).$$
Choose $a \in k$ \lq\lq sufficiently large\rq\rq\ so that $a \xi^i \Pi_{\tilde{\infty}}^j \in O_K$ for $0\leq i < f_K(\infty)$, $0\leq j < e_K(\infty)$, and  $K_\infty = O_K + a O_{K_{\tilde{\infty}}}$.
Then from \textit{Remark}~\ref{rem 1.10} we obtain that
$$ D_A(O_K) = q^{1-g_k} \cdot \left(\frac{\prod_{0\leq i < f_K(\infty)} \prod_{0\leq j < e_K(\infty)} \big|a\xi^i \Pi_{\tilde{\infty}}^j\big|_\infty}{\#\big(O_K \cap aO_{K_{\tilde{\infty}}}\big)}\right)^{1/r}.$$
Applying the Riemann-Roch Theorem, one can get
$$
\#\big(O_K \cap aO_{K_{\tilde{\infty}}}\big)
= q^{[K:k]\cdot \deg a} \cdot q_K^{1-g_K} \cdot \#\left(\frac{K_\infty}{O_K + aO_{K_{\tilde{\infty}}}}\right) = q^{r\deg a} \cdot q_K^{1-g_K}.
$$
Therefore
\begin{eqnarray}
D_A(O_K) &=& q^{1-g_k} \cdot q_K^{\frac{g_K-1}{r}} \cdot q_\infty^{\frac{f_K(\infty)}{r} \cdot \big(-\frac{1}{e_K(\infty)} \cdot \frac{e_K(\infty) (e_K(\infty)-1)}{2}\big)} \nonumber \\
&=& q^{1-g_k} \cdot q_K^{\frac{g_K-1}{r}} \cdot q_\infty^{\frac{-r+f_K(\infty)}{2r}}. \nonumber
\end{eqnarray}
\end{proof}

\begin{rem}\label{rem 4.6}
For every $A$-order $\Ocal$ in $K$,
let $\dfk(\Ocal/A) \lhd A$ be the (relative) discriminant ideal of $\Ocal/A$.
When $K/k$ is separable and tamely ramified at $\infty$, 
from the Riemann-Hurwitz formula we get
\begin{eqnarray}
\Vert \dfk(O_K/A) \Vert &=& q_K^{(2g_K-2)-[K:\FF_K k](2g_k-2)} \cdot q_\infty^{-f_K(\infty)\cdot (e_K(\infty)-1)} \nonumber \\
&=&q^{2r(1-g_k)} \cdot q_K^{2(g_K-1)} \cdot q_\infty^{-r + f_K(\infty)}.
\end{eqnarray}
Here $\FF_K$ is the constant field of $K$ and $e_K(\infty)$ is the ramification index of $\infty$ in $K$.
Thus for every $A$-order $\Ocal$ in $K$, by Lemma~\ref{lem 4.5} we obtain
\begin{eqnarray}
D_A(\Ocal) &=& D_A(O_K) \cdot \#(O_K/\Ocal)^{1/r} \nonumber \\
&=& \Vert \dfk(O_K/A)\Vert^{1/2r} \cdot \#(O_K/\Ocal)^{1/r} \nonumber \\
&=& \Vert \dfk(\Ocal/A)\Vert^{1/2r}. \nonumber
\end{eqnarray}
\end{rem}

\subsubsection{Asymptotic behavior}\label{sec 4.2.1}
Let $\zeta_K(s):= (1-q_\infty^{-f_K(\infty)s})^{-1} \cdot  \zeta_{O_K}(s)$. 
Write
$$\zeta_K(s) = \frac{c_{-1}}{s-1} + c_0 + {\rm O}(s-1)$$
and put $\gamma_K := c_0/c_{-1}$, the \textit{Euler-Kronecker constant of $K$} (following the definition in \cite[equation (0.1) and (0.2)]{Iha}).
From the functional equation of $\zeta_K(s)$ (cf.\ \cite[Theorem 5.9]{Ros}):
$$\zeta_K(s) = q_K^{(g_K-1)(1-2s)} \cdot \zeta_K(1-s),$$ 
one gets
$$\zeta_{O_K}(s) =  q_K^{(g_K-1)(1-2s)} \cdot (1-q_\infty^{-f_K(\infty)s}) \cdot \zeta_K(1-s),$$
whereas the logarithmic derivative of $\zeta_{O_K}(s)$ at $s=0$ equals to:
\begin{eqnarray}\label{eqn: log-der zeta}
\frac{\zeta_{O_K}'(0)}{\zeta_{O_K}(0)} &=& -\left[2(g_K-1)\ln q_K + \frac{ f_K(\infty)}{2}\ln q_\infty  + \gamma_K\right].
\end{eqnarray}

Let $\rho^o$ be a rank $r$ Drinfeld $A$-module over $\overline{k}$ with CM by $O_K$.
Then Theorem~\ref{thm 4.3}, Lemma~\ref{lem 4.5}, and the equality~(\ref{eqn: log-der zeta}) give us:
\begin{eqnarray}
h^{\text{\rm st}}_{\text{\rm Tag}}(\rho^o) &=& 
\left[(g_k-1)\ln q + \frac{1-g_K}{r} \ln q_K + \frac{r-f_K(\infty)}{2r} \ln q_\infty\right] \nonumber \\
&& + \frac{1}{r}\left[2(g_K-1)\ln q_K + \frac{ f_K(\infty) }{2}\ln q_\infty  + \gamma_K\right] \nonumber \\
&=&(g_k-1) \ln q + \frac{g_K-1}{r} \ln q_K +  \frac{1}{2} \ln q_\infty + \frac{\gamma_K}{r}. \nonumber
\end{eqnarray}
From Ihara's estimate of $\gamma_K$ (cf.\ \cite[equation (0.6) and (0.12)]{Iha}), we get
$$\gamma_K = {\rm O}_{k,r}\big(\ln (g_K-1)\big), \quad \quad \text{ $g_K \gg 0$}.$$
Here ${\rm O}_{k,r}$ is the big-O notation (depending on the base field $k$ and the rank $r$).
Therefore we obtain the following asymptotic formula:

\begin{cor}\label{cor 4.5}
Let $\rho^o$ be a Drinfeld $A$-module of rank $r$ over $\overline{k}$ with CM by $O_K$. Then
$$h^{\text{\rm st}}_{\text{\rm Tag}}(\rho^o) = \frac{g_K-1}{r} \ln q_K + 
{\rm O}_{k,r}\big(\ln (g_K-1)\big), \quad \quad \text{ $g_K \gg 0$}.$$
\end{cor}


\section{On zeta functions over \lq\lq totally real\rq\rq\ function fields}\label{sec Z-E}

Let $F$ be a finite extension over $k$ with $[F:k] = r$. 
Suppose $F/k$ is \lq\lq totally real\rq\rq\ with respect to $\infty$, i.e.\ $\infty$ is totally split in $F$. 
Note that $F/k$ must be separable. 
Let $\Ocal$ be an $A$-order in $F$. 
Let $\Ifk$ be an invertible fractional ideal of $\Ocal$, the partial zeta function associated to $\Ifk$ is
$$\zeta_{\Ocal}(s;\Ifk) := \sum_{\text{invertible } I \lhd \Ocal \atop [I] = [\Ifk]} \frac{1}{N_{\Ocal}(I)^s}.$$
Here $[\Ifk] \in \Pic(\Ocal)$ denotes the ideal class represented by $\Ifk$. Note that $\zeta_{\Ocal}(s;\Ifk)$ only depends on the ideal class of $\Ifk$. We shall first prove a \lq\lq limit\rq\rq\ formula for $\zeta_{\Ocal}(s;\Ifk)$ at $s = 0$.

Let 
$$(F_\infty^\times)_0 :=\frac{(F\otimes_k k_\infty)^\times}{(k \otimes_k k_\infty)^\times} \quad \text{ and } \quad (\Ocal^\times)_0 := \frac{\Ocal^\times}{\FF_q^\times}.$$
Then the canonical embedding $O_F \hookrightarrow F \hookrightarrow F_\infty := F \otimes_k k_\infty$ induces an injective map $(O_F^\times)_0 \hookrightarrow (F_\infty^\times)_0$, which is cocompact by the Dirichlet unit Theorem.
Identifying $F_\infty$ with $k_\infty^r$, each $\alpha \in F$ corresponds to a vector denoted by $(\alpha^{(1)},...,\alpha^{(r)}) \in k_\infty^r$.

Given an invertible fractional ideal $\Ifk$ of $\Ocal$,
take a projective $A$-module $Y_\Ifk \subset k^r$ of rank $r$ and $\lambda_1,...,\lambda_r \in \Ifk$ so that
$$\Ifk = \{ c_1 \lambda_1 + \cdots c_r \lambda_r : (c_1,...,c_r) \in Y_\Ifk \}.$$

We introduce the following {\it \lq\lq Heegner cycle\rq\rq} on the $\GL(Y_\Ifk) \backslash \Bcal^r(\ZZ)$:
$$\Zfk_\Ifk := \{z_\Ifk(\mathbf{t}) : \mathbf{t} \in (F_\infty^\times)_0\},$$
where for $\mathbf{t} = (t_1: \cdots : t_r) \in (F_\infty^\times)_0$, the vertex $z_\Ifk(\mathbf{t})$ is represented by
$$
\begin{pmatrix}
t_1 & & \\
& \ddots & \\
& & t_r 
\end{pmatrix}
\cdot \ g_\infty(\lambda_1,...,\lambda_r), \quad \text{ where } \quad 
g_\infty(\lambda_1,...,\lambda_r) := 
\begin{pmatrix}
\lambda_1^{(1)} & \cdots & \lambda_1^{(r)} \\
\vdots & & \vdots \\
\lambda_r^{(1)} & \cdots & \lambda_r^{(r)}
\end{pmatrix} \quad \in \GL_r(k_\infty).
$$
Note that the cycle $\Zfk_\Ifk$ only depends on the ideal $\Ifk$.
Recall the following bijection:
$$
\begin{tabular}{ccl}
$\left(\coprod\limits_{[Y] \in \Pscr_A^r} \GL(Y)\backslash \Bcal^r(\ZZ)\right)$ &$\longleftrightarrow$ & $\GL_r(k) \backslash \GL_r(\AA) / k_\infty^\times \GL_r(O_\AA) = : \YY(1)$ \\
$[g_\infty] \in \GL(Y)\backslash \Bcal^r(\ZZ)$ & $\longleftrightarrow$ & $[g_Y^\infty, g_\infty]$,
\end{tabular} 
$$
where the element $g_Y^\infty \in \GL(\AA^\infty)$ is chosen so that $O_{\AA^\infty}^r \cdot (g_Y^\infty)^{-1} = \widehat{Y}$, the closure of $Y$ in $(\AA^\infty)^r$.
Then $\Zfk_\Ifk$ can be identified as a \lq\lq cycle\rq\rq\ in $\YY(1)$ associated to $\Ifk$.
We remark that
\begin{eqnarray}\label{eqn:disc}
&& \quad  |\det(g_{Y_\Ifk}^\infty)|_\AA = \Vert Y_\Ifk \Vert \quad \text{and} \quad
\Vert Y_\Ifk \Vert \cdot |\det(g_\infty(\lambda_1,...,\lambda_r))|_\infty = N_{\Ocal}(\Ifk^{-1} )\cdot \Vert \dfk(\Ocal/A) \Vert^{1/2}.
\end{eqnarray}
${}$

Let $\mathbf{1}_{O_{\AA^\infty}^r} \in S((\AA^\infty)^r)$ be the characteristic function of $O_{\AA^\infty}^r$ and $\Omega(\Ocal) := (F_\infty^\times)_0/(\Ocal^\times)_0$.
Then: 

\begin{prop}\label{prop part-E}
Let $\dfk(\Ocal/A)$ be the discriminant ideal of $\Ocal$ over $A$. We have
$$
\Vert \dfk(\Ocal/A) \Vert^{s/2} \cdot (1-q_\infty^{-s})^{-r} \cdot \zeta_{\Ocal}(s;\Ifk)
=
\frac{1}{q-1} \int_{\Omega(\Ocal)} \Ecal^\infty(z_{\Ifk^{-1}}(\mathbf{t}),s; \mathbf{1}_{{O_{\AA^\infty}^r}}) d^\times \mathbf{t}.
$$
Here $d^\times \mathbf{t}$ is induced from the Haar measure on $F_\infty^\times \cong (k_\infty^\times)^r$ so that the volume of $(O_\infty^{\times})^r$ is one.
\end{prop}

\begin{proof}
For $\re(s)>1$, we first express $\Vert \dfk(\Ocal/A) \Vert^{s/2} \cdot (1-q_\infty^{-s})^{-r} \cdot \zeta_{\Ocal}(s;\Ifk)
$ as
\begin{eqnarray}
&& \frac{\Vert \dfk(\Ocal/A) \Vert^{s/2} }{N(\Ifk)^s (1-q_\infty^{-s})^r } \cdot \sum_{\lambda \in (\Ifk^{-1}-\{0\})/ \Ocal^\times} \frac{1}{|\text{N}_{F/k}(\lambda)|_\infty^s} \nonumber \\
&=& \frac{\Vert \dfk(\Ocal/A) \Vert^{s/2} }{N(\Ifk)^s} 
\cdot \int_{(k_\infty^\times)^r} \left(\sum_{\lambda \in (\Ifk^{-1}-\{0\})/ \Ocal^\times} \mathbf{1}_{(O_\infty)^r} (t_1 \lambda^{(1)},...,t_r \lambda^{(r)}) |t_1\cdots t_s|_\infty^s \right) d^\times t_1 \cdots d^\times t_r \nonumber \\
&=& \frac{\Vert \dfk(\Ocal/A) \Vert^{s/2} }{(q-1) N(\Ifk)^s} \cdot
\int_{\Omega(\Ocal)} \left(\int_{k_\infty^\times} 
\sum_{(c_1,...,c_r) \in Y_\Ifk}' \mathbf{1}_{O_\infty^r}( a(c_1,...,c_r) z_{\Ifk^{-1}}(\mathbf{t})) |a|_\infty^{rs} d^\times a\right) d^\times \mathbf{t} \nonumber 
\end{eqnarray}
The result then follows from the expression of the finite mirabolic Eisenstein series in \eqref{eqn:fmE} and the equality \eqref{eqn:disc}.
\end{proof}

Set
$$ \widetilde{\zeta}_{\Ocal}(s;\Ifk) := \Vert \dfk(\Ocal/A) \Vert^{s/2} \cdot \frac{1-q_\infty^{-rs}}{(1-q_\infty^{-s})^{r}} \cdot \zeta_\Ocal(s; \Ifk) \quad \text{ and } \quad  \eta(g) := \eta(g,\mathbf{1}_{O_{\AA^\infty}^r}),$$
where $\eta(\cdot,\varphi^\infty)$ is introduced in Section~\ref{sec 3.3}.
Let $\Rcal_\Ocal$ be the regulator of $\Ocal$ (with respect to $q_\infty$), i.e.\
$$
\Rcal_\Ocal := \Big|\det \begin{pmatrix} \log_{q_\infty} |\varepsilon_1^{(1)}|_\infty & \cdots &\log_{q_\infty} |\varepsilon_1^{(r-1)}|_\infty \\
\vdots & & \vdots \\
\log_{q_\infty} |\varepsilon_{r-1}^{(1)}|_\infty & \cdots & \log_{q_\infty} |\varepsilon_{r-1}^{(r-1)}|_\infty
\end{pmatrix}
\Big| \quad \in \ZZ_{>0},
$$
where $\{\varepsilon_1,...,\varepsilon_{r-1}\}$ is a set of {\it fundamental units} of $\Ocal$.
One can check that
\begin{eqnarray}\label{eqn vol}
\text{\rm vol}(\Omega(\Ocal), d^\times \mathbf{t})
= \frac{r \cdot (q-1)}{q_\Ocal-1} \cdot \Rcal_\Ocal.
\end{eqnarray}
Here $q_\Ocal$ is the cardinality of the constants of $F$ contained in $\Ocal$.
The Lerch-type formula of finite mirabolic Eisenstein series then shows immediately the following limit formula of $ \widetilde{\zeta}_{\Ocal}(s;\Ifk)$:

\begin{cor}\label{cor Z-E}
\begin{eqnarray}
\widetilde{\zeta}_{\Ocal}(s;\Ifk) 
&=& -\frac{r}{q_\Ocal-1}\cdot \Rcal_{\Ocal} - \left( \frac{r \ln q_\infty}{(q-1) (q_\infty^r-1)}\cdot \int_{\Omega(\Ocal)} \eta(z_{\Ifk^{-1}}(\mathbf{t})) d^\times \mathbf{t}\right) \cdot s \nonumber \\
&& + \text{\rm O}(s^2) \quad \text{ as $s \rightarrow 0$.} \nonumber 
\end{eqnarray}
\end{cor}

\begin{proof}
By \textit{Remark}~\ref{rem 3.8} (3), we have
\begin{eqnarray}
&&\int_{\Omega(\Ocal)} (1-q_\infty^{-rs}) \cdot \Ecal^\infty(z_{\Ifk^{-1}}(\mathbf{t}),s; \mathbf{1}_{{O_{\AA^\infty}^r}}) d^\times \mathbf{t} \nonumber \\
&=&-\text{vol}(\Omega(\Ocal),d^\times \mathbf{t}) - \left(\frac{r \ln q_\infty}{q_\infty^r-1}\cdot\int_{\Omega(\Ocal)}
 \eta(z_{\Ifk^{-1}}(\mathbf{t}))d^\times \mathbf{t}\right)\cdot s  + \text{\rm O}(s^2) \quad \text{ as $s \rightarrow 0$.} \nonumber \\
&=& - \frac{r \cdot (q-1)}{q_\Ocal-1} \cdot  \Rcal_{\Ocal}  - \left( \frac{r \ln q_\infty}{q_\infty^r-1}\cdot \int_{\Omega(\Ocal)} \eta(z_{\Ifk^{-1}}(\mathbf{t})) d^\times \mathbf{t}\right) \cdot s  + \text{\rm O}(s^2) \quad \text{ as $s \rightarrow 0$.} \nonumber 
\end{eqnarray}
The second equality comes from (\ref{eqn vol}).
Hence the result follows from Proposition~\ref{prop part-E}.
\end{proof}

\begin{rem}
We may view the above formula as an analogue of 
\cite[Theorem 1]{L-M} for totally real number fields (generalzation of the work of Hecke \cite{Hec} and Bump-Goldfield \cite{B-G}).
\end{rem}

We then arrive at:

\begin{thm}\label{thm E-K}
${}$
\begin{itemize}
\item[(1)] \text{\rm (Kronecker term)}
Let $ \widetilde{\zeta}_{\Ocal}(s) := \sum_{[\Ifk] \in \Pic(\Ocal)}  \widetilde{\zeta}_{\Ocal}(s;\Ifk)$.
Writing $$\widetilde{\zeta}_{\Ocal}(s) = c_{0}(\Ocal)  + c_1(\Ocal)s + \cdots,$$
the natural logarithmic derivative of $\widetilde{\zeta}_{\Ocal}(s)$ at $s=0$ equals to
\begin{eqnarray}
\frac{c_1(\Ocal)}{c_0(\Ocal)} = \frac{r \ln q_\infty}{(q_\infty^r-1)} \cdot \left( \frac{1}{h_\Ocal \cdot \text{\rm vol}\big(\Omega(\Ocal), d^\times \mathbf{t}\big)} \sum_{[\Ifk] \in \Pic(\Ocal)}
 \int_{\Omega(\Ocal)} \eta(z_{\Ifk}(\mathbf{t})) d^\times \mathbf{t}\right). \nonumber 
\end{eqnarray}
Here $h_{\Ocal} := \#(\Pic(\Ocal))$.

\item[(2)] ({\rm Lerch-type formula})
Let $\chi$ be a non-trivial character of $\Pic(\Ocal)$, and set 
$$\widetilde{L}_{\Ocal}(s,\chi):= \sum_{[\Ifk] \in \Pic(\Ocal)} \chi([\Ifk]) \cdot \widetilde{\zeta}_{\Ocal}(s;\Ifk).$$
Then $\widetilde{L}_{\Ocal}(s,\chi)$ vanishes at $s=0$, and
$$
\frac{d}{d s}\widetilde{L}_{\Ocal}(s,\chi)\Big|_{s=0} = 
- \frac{r \ln q_\infty}{(q-1)(q_\infty^r-1)}\cdot \sum_{[\Ifk] \in \Pic(\Ocal)}\overline{\chi(\Ifk)} \left(\int_{\Omega(\Ocal)} \eta(z_{\Ifk}(\mathbf{t})) d^\times \mathbf{t}\right).
$$
\end{itemize}
\end{thm}

\begin{proof}
Let $\chi$ be a character of $\Pic(\Ocal)$. From Corollary~\ref{cor Z-E}, we obtain
\begin{eqnarray}
&& \sum_{[\Ifk] \in \Pic(\Ocal)} \chi([\Ifk]) \cdot \widetilde{\zeta}_{\Ocal}(s;\Ifk) \nonumber \\
&=& -\frac{r \Rcal_{\Ocal}}{q_\Ocal-1}\cdot \left(\sum_{[\Ifk] \in \Pic(\Ocal)} \chi([\Ifk])\right) \nonumber \\
&&  -\frac{r \ln q_\infty}{(q-1) (q_\infty^r-1)}\cdot \left( \sum_{[\Ifk] \in \Pic(\Ocal)} \chi([\Ifk]) \cdot \int_{\Omega(\Ocal)} \eta(z_{\Ifk^{-1}}(\mathbf{t})) d^\times \mathbf{t}\right) \cdot s \nonumber \\
&& + \text{\rm O}(s^2) \quad \text{ as $s \rightarrow 0$.} \nonumber 
\end{eqnarray}
Take $\chi$ to be trivial and get
$$
c_0(\Ocal) = \frac{-r h_\Ocal \Rcal_{\Ocal}}{q_\Ocal -1} \quad \text{ and } \quad
c_1(\Ocal) = -\frac{r \ln q_\infty}{(q-1) (q_\infty^r-1)}\cdot \sum_{[\Ifk] \in \Pic(\Ocal)}  \int_{\Omega(\Ocal)} \eta(z_{\Ifk}(\mathbf{t})) d^\times \mathbf{t}.
$$
Thus (1) follows from the equality~(\ref{eqn vol}).
When $\chi$ is non-trivial, one has $\sum_{[\Ifk] \in \Pic(\Ocal)} \chi([\Ifk]) = 0$.
Hence
$$
\widetilde{L}_\Ocal(s,\chi) = 
-\left(\frac{r \ln q_\infty}{(q-1) (q_\infty^r-1)}\cdot  \sum_{[\Ifk] \in \Pic(\Ocal)} \overline{\chi([\Ifk])} \cdot \int_{\Omega(\Ocal)} \eta(z_{\Ifk}(\mathbf{t})) d^\times \mathbf{t}\right) \cdot s + \text{\rm O}(s^2) \quad \text{ as $s \rightarrow 0$,}
$$
which shows (2).
\end{proof}

\section{Special values of Automorphic $L$-functions}\label{sec 5}


\subsection{Rankin-Selberg $L$-functions}\label{sec 5.1}

Let $\Pi_1$ and $\Pi_2$ be two automorphic cuspidal representations of $\GL_r(\AA)$ with unitary central character denoted by $\omega_1$ and $\omega_2$, respectively. 
Given a Schwartz function $\varphi \in S(\AA^r)$ and cusp forms $f_i \in \Pi_i$ for $i = 1,2$, the \textit{Rankin-Selberg $L$-function associated to $f_1,\ f_2,$ and $\varphi$} is defined by
$$L(s;f_1,f_2,\varphi) :=
\int_{\AA^\times \GL_r(k)  \backslash \GL_r(\AA)} f_1(g) f_2(g) \Ecal(g,s;\chi,\varphi) dg.$$
Here $\chi := \omega_1^{-1}  \cdot \omega_2^{-1}$, and $dg$ is the Tamagawa measure (i.e.\ $\text{vol}(\GL_r(k) \AA^\times \backslash \GL_r(\AA), dg) = 2$).
The meromorphic continuation and the functional equation of $\Ecal(g,s,\chi,\varphi)$ enable us to extend
$L(s;f_1,f_2,\varphi)$ to a meromorphic function on the $s$-plane satisfying
$$L(s;f_1,f_2,\varphi) = L(1-s;\tilde{f}_1,\tilde{f}_2,\widehat{\varphi}),$$
where $\tilde{f}_i(g) := f_i(^tg^{-1})$ for $i = 1,2$. 
In particular, Theorem~\ref{thm 2.2} (1) tells us that $L(s;f_1,f_2,\varphi)$ is holomorphic at $s=0$ if $f_1$ and $f_2$ are orthogonal with respect to the Petersson inner product.\\

The Whittaker function of $f_1$ (resp.\ $f_2$) with respect to $\psi$ (resp.\ $\overline{\psi}$) is defined by:
$$W_{f_1}(g) := \int_{N_r(k)\backslash N_r(\AA)} f_1(ng) \psi(-n) dn, \ \
W_{f_2}'(g) := \int_{N_r(k)\backslash N_r(\AA)} f_2(ng) \psi(n) dn, \ \  \forall g \in \GL_r(\AA).$$
If $f_1$ and $f_2$ are factorizable, then for $g = (g_v)_v \in \GL_r(\AA)$ one has
$$W_{f_1}(g) = \prod_v W_{f_1,v}(g_v) \quad \text{ and } \quad W_{f_2}'(g) = \prod_v W_{f_2,v}'(g_v),$$
where $W_{f_1,v} := W_{f_1}\big|_{k_v}$ and $W_{f_2,v}' := W_{f_2} \big|_{k_v}$. Moreover, for a factorizable $\varphi \in S(\AA^r)$ we get
$$L(s;f_1,f_2, \varphi) = \prod_vL_v(s;f_1,f_2,\varphi_v), \quad \re(s)>1.$$
Here
\begin{eqnarray}
&& L_v(s;f_1 ,f_2,\varphi_v) \nonumber \\
&:=& \int_{N_r(k_v)\backslash \GL_r(k_v)} W_{f_1,v}(g_v) W_{f_2,v}' (g_v) \varphi_v\big((0,...,0,1) g_v\big) |\det g_v|_v^s dg_v, \quad \forall \re(s)>0, \nonumber
\end{eqnarray}
which has meromorphic continuation to the whole complex $s$-plane (cf.\ \cite[p.\ 121]{S-S}).\\

Let $L(s,\Pi_1\times \Pi_2) = \prod_v L_v(s,\Pi_1\times \Pi_2)$ be the Rankin-Selberg $L$-function associated to $\Pi_1$ and $\Pi_2$. For factorizable $f_1 \in \Pi_1,\ f_2 \in \Pi_2$ and $\varphi \in S(\AA^r)$, the quotient
$$L_v^o(s;f_1 ,f_2,\varphi_v) := \frac{L_v(s;f_1 , f_2,\varphi_v)}{L_v(s,\Pi_1\times \Pi_2)}$$
is a polynomial in $\CC[q_v^{-s}, q_v^s]$ for every place $v$ of $k$, which is equal to $1$ when $v$ is \lq\lq good\rq\rq\ (i.e.\ the representations $\Pi_1$ and $\Pi_2$ are both unramified at $v$, the additive character $\psi_v$ has trivial conductor at $v$, the automorphic forms $f_1$ and $f_2$ are both spherical at $v$ with 
$W_{f_1,v}(1) =W_{f_2,v}'(1) = 1$, and $\varphi_v = \mathbf{1}_{O_v^r}$).
Therefore the following identity holds for factorizable $f_1 \in \Pi_1$, $f_2 \in \Pi_2$, and $\varphi \in S(\AA^r)$:
\begin{eqnarray}\label{eqn RS-local-global}
L(s;f_1,f_2,\varphi) &=& L(s,\Pi_1\times \Pi_2) \cdot \prod_v L_v^o(s;f_1,f_2,\varphi_v).
\end{eqnarray}

Suppose $\chi\big|_{k_\infty^\times} = 1$. 
Let $\Pscr^{\text{\rm RS}} := : \Pi_1 \times \Pi_2 \times S((\AA^\infty)^r)_\ZZ \longrightarrow \CC$
be the multilinear functional satisfying
\begin{eqnarray}\label{eqn 5.1}
\Pscr^{\text{\rm RS}}(f_1,f_2,\varphi^\infty) &:=& L_\infty^o(0;f_1 ,f_2, \mathbf{1}_{L^o}) 
\cdot \prod_{v \neq \infty} L_v^o(0;f_1,f_2,\varphi_v)
\end{eqnarray}
for factorizable $f_1 \in \Pi_1$, $f_2 \in \Pi_2$, and $\varphi^\infty \in S((\AA^\infty)^r)_\ZZ$. 
On the other hand, we have another multilinear functional $\Pcal^{\text{\rm RS}}: \Pi_1 \times \Pi_2 \times S((\AA^\infty)^r)_\ZZ \longrightarrow \CC$ defined by:
$$\Pcal^{\text{\rm RS}}(f_1,f_2,\varphi^\infty) := \frac{1}{1-q_\infty^r} \int_{\AA^\times \GL_r(k) \backslash \GL_r(\AA)} f_1(g)f_2(g) \eta_{\chi}(g;\varphi^\infty) dg,$$
where $\eta_{\chi}(g;\varphi^\infty)$ is introduced in Section~\ref{sec 3.3}.
Then:

\begin{cor}
Let $\Pi_1$ and $\Pi_2$ be two automorphic cuspidal representations of $\GL_r(\AA)$ with unitary central characters $\omega_1$ and $\omega_2$, respectively.
Suppose $\Pi_1$ is not the contragredient representation of $\Pi_2$ and $(\omega_1\omega_2)\big|_{k_\infty^{\times}} = 1$. 
We then have
$$\Pcal^{\text{\rm RS}} = L(0,\Pi_1 \times \Pi_2) \cdot \Pscr^{\text{\rm RS}}.$$
\end{cor}

\begin{proof}
Given factorizable $f_1 \in \Pi_1$, $f_2 \in \Pi_2$, and $\varphi^\infty \in S((\AA^\infty)^r)_\ZZ$, by the definition of $\Pscr^{\text{\rm RS}}$ in (\ref{eqn 5.1}) one has
$$ \Pscr^{\text{\rm RS}}(f_1,f_2,\varphi^\infty) = L_\infty^o(0;f_1 ,f_2, \mathbf{1}_{L^o}) 
\cdot \prod_{v \neq \infty} L_v^o(0;f_1,f_2,\varphi_v).$$
On the other hand, Corollary~\ref{cor 3.7} shows that
\begin{eqnarray}
&&L(s;f_1,f_2,\varphi^\infty \otimes \mathbf{1}_{L^o}) \nonumber \\
&=& \int_{\AA^\times \GL_r(k) \backslash \GL_r(\AA)} f_1(g) f_2(g) \Ecal(g,s;\chi,\varphi^\infty \otimes \mathbf{1}_{L^o}) dg \nonumber \\
&=& -\left(\int_{\AA^\times \GL_r(k) \backslash \GL_r(\AA)} f_1(g) f_2(g) d g\right) \cdot \left(\varphi^\infty(0) \cdot \int_{k^\times \backslash \AA^{\infty,\times}} \overline{\chi(a^\infty)} d^\times a^\infty\right) \cdot  \left(\frac{1}{r\ln q_\infty} \cdot s^{-1}+ \frac{1}{2}\right) \nonumber \\
&&+\frac{1}{1-q_\infty^r}  \cdot \int_{\AA^\times \GL_r(k) \backslash \GL_r(\AA)} f_1(g) f_2(g) \eta_\chi(g;\varphi^\infty) dg \nonumber \\
&& + O(s) \quad \text{ as $s\rightarrow 0$.} \nonumber 
\end{eqnarray}
Since $\Pi_1$ is not the contragredient representation of $\Pi_2$, one has
$$\int_{\AA^\times \GL_2(k)\backslash \GL_2(\AA)} f_1(g) f_2(g) dg = 0.$$
Thus $L(s;f_1,f_2,\varphi^\infty \otimes \mathbf{1}_{L^o})$ is holomorphic at $s=0$, and
\begin{eqnarray}
L(0;f_1,f_2,\varphi^\infty \otimes \mathbf{1}_{L^o})
&=& \frac{1}{1-q_\infty^r}  \cdot \int_{\AA^\times \GL_r(k) \backslash \GL_r(\AA)} f_1(g) f_2(g) \eta_\chi(g;\varphi^\infty) dg \nonumber\\
&=&\Pcal^{\text{\rm RS}}(f_1,f_2,\varphi^\infty). \nonumber
\end{eqnarray}
Therefore the result follows from the identity~(\ref{eqn RS-local-global}).
\end{proof}




\subsection{Godement-Jacquet $L$-functions}\label{sec 5.2}

Let $\Pi$ be an automorphic cuspidal representation of $\GL_r(\AA)$ with a unitary central character denoted by $\omega$.
Given  $f_1, f_2 \in \Pi$ and a Schwartz function $\varPhi \in S(\Mat_r(\AA))$, the Godement-Jacquet $L$-function associated to $f_1, f_2$ and $\varPhi$ is defined by
$$L^{\text{GJ}}(s;f_1,f_2,\varPhi):= \int_{\GL_r(\AA)} |\det (g)|_\AA^s \varPhi(g) \cdot \langle \Pi(g) f_1, f_2\rangle_{\text{Pet}} dg, \quad \forall \re(s)> r,$$
where $\langle \cdot, \cdot \rangle_{\text{Pet}}$ is the Petersson inner product on $\Pi$:
$$\langle f_1, f_2\rangle_{\text{Pet}} := \int_{\AA^\times \GL_r(k)\backslash \GL_r(\AA)} f_1(g) \overline{f_2(g)} dg, \quad \forall f_1,f_2 \in \Pi.$$
As before, we choose $dg$ to be the Tamagawa measure. 
For $f \in \Pi$, recall that we put
$\tilde{f}(g) = f({}^t g ^{-1})$ for $g \in \GL_r(\AA)$.
The contragredient representation of $\Pi$ can be realized by
$$\widetilde{\Pi} := \{\tilde{f} : f \in \Pi\},$$
where $\GL_r(\AA)$ acts via right translations.
Let $\widehat{\varPhi}$ be the following Fourier transform of $\varPhi$:
$$\widehat{\varPhi}(X) := \int_{\Mat_r(\AA)} \varPhi(Y) \psi(\tr(X \cdot  {}^tY)) dY, \quad \forall X \in \Mat_r(\AA).$$
It is known that $L^{\text{GJ}}(s;f_1,f_2,\varPhi)$ has analytic continuation to the whole complex $s$-plane satisfying the following functional equation (cf.\ \cite[Theorem 13.8]{G-J}):
\begin{eqnarray}\label{eqn 5.2}
L^{\text{GJ}}(s;f_1,f_2,\varPhi) = L^{\text{GJ}}(r-s;\tilde{f}_1,\tilde{f}_2,\widehat{\varPhi}).
\end{eqnarray}
${}$

Identify $\Mat_r(k)$ with the vector space $k^{r^2}$ as follows:
\begin{eqnarray}\label{eqn 5.3}
(a_{ij})_{1\leq i,j \leq r} \in \Mat_r(k) &\longleftrightarrow& (a_{11},...,a_{1r}, a_{21},...,a_{2r},...,a_{r1},...a_{rr}) \in k^{r^2}.
\end{eqnarray}
For each Schwartz function $\varPhi \in S(\Mat_r(\AA))$, recall the mirabolic Eisenstein series associated to $\omega^{-1}$ and $\varPhi$:
$$\Ecal(\mathbf{g},s; \omega^{-1},\varPhi) 
= |\det(\mathbf{g})|_\AA^s \cdot \int_{k^\times \backslash \AA^\times} 
\sum_{0 \neq X \in k^{r^2}} \varPhi(a^{-1} \cdot X \cdot \mathbf{g}) \omega^{-1}(a) |a|_\AA^{-r^2s} d^\times a, \quad \forall \mathbf{g} \in \GL_{r^2}(\AA).$$
The right action of $\GL_r \times \GL_r$ on $\Mat_r$ defined by
$$ (g_1,g_2) \cdot X :=  {}^t g_2 X g_1, \quad \forall (g_1,g_2) \in \GL_r \times \GL_r \text{ and } X \in \Mat_r,$$
induces a group homomorphism $\iota: \GL_r \times \GL_r \rightarrow \GL_{r^2}$.
the mirabolic Eisenstein series restricting to $\iota(\GL_r^2)(\AA)$ can be  written as follows:
for $\varPhi \in S(\Mat_r(\AA))$ and $(g_1,g_2) \in \GL_r^2(\AA)$ we have
\begin{eqnarray}
&& \Ecal(\iota(g_1,g_2),s; \omega^{-1},\varPhi) \nonumber \\
&=& |\det(g_1) \det(g_2)|_\AA^{rs} \cdot
\int_{a \in k^\times \backslash \AA^\times} 
\sum_{0 \neq X \in \Mat_r(k)} \varPhi( {}^tg_2 a^{-1}  X g_1) \omega^{-1}(a) |a|_\AA^{-r^2s} d^\times a.\nonumber
\end{eqnarray}

It is known that the Godement-Jacquet $L$-function can be expressed as the following \lq\lq doubling method\rq\rq\ integral:

\begin{prop}\label{prop 5.2}
\text{\rm (cf.\ \cite[Proposition 3.2]{GPSR})}
Let $\Pi$ be an automorphic cuspidal representation of $\GL_r(\AA)$ with a unitary central character denoted by $\omega$.
For $f_1,f_2 \in \Pi$ and $\varPhi \in S(\Mat_r(\AA))$ we have
\begin{eqnarray}
&& L^{\text{\rm GJ}}(rs;f_1,f_2,\varPhi) \nonumber\\
&=& \iint_{\big(\AA^\times \GL_r(k)\backslash \GL_r(\AA)\big)^2}
f_1(g_1)  \Ecal(\iota(g_1,g_2),s; \omega^{-1},\varPhi) \overline{\tilde{f}_2(g_2)} dg_1 dg_2. \nonumber
\end{eqnarray}
\end{prop}

\subsubsection{Inner product formula of $L(s,\Pi)$}\label{sec 5.2.1}

Writing $\Pi = \otimes_v \Pi_v$ and $\widetilde{\Pi} = \otimes_v \widetilde{\Pi}_v$,
take a pairing $\langle \cdot, \cdot \rangle_v : \Pi_v \times \widetilde{\Pi}_v \rightarrow \CC$ for each $v$ so that
$$\langle f_1, f_2 \rangle_{\text{\rm Pet}} = \prod_v \langle f_{1,v}, \overline{f_{2,v}}\rangle_v$$
for factorizable $f_1, f_2 \in \Pi$.
Then for factorizable $f_1,f_2 \in \Pi$ and $\varPhi \in S(\Mat_r(\AA))$, one has 
\begin{eqnarray}\label{eqn GJ-local-global}
L^{\text{\rm GJ}}(s;f_1,f_2,\varPhi) = L(s-\frac{r-1}{2},\Pi) \cdot \prod_v L_v^o(s; f_{1,v},f_{2,v},\varPhi_v),
\end{eqnarray}
where
$$L_v^o(s;f_{1,v},f_{2,v},\varPhi_v) := \frac{1}{L_v(s-(r-1)/2,\Pi)} \cdot 
\int_{\GL_r(k_v)} \varPhi_v(g_v) \langle \Pi_v(g_v) f_{1,v}, \overline{f_{2,v}}\rangle_v |\det g_v|_v^s dg_v,$$
which converges absolutely when $s > r$ and can be extended to an entire function.
In fact, we have $L_v^o(s;f_1,f_2,\varPhi) \in \CC[q_v^{-s},q_v^s]$, and $L_v^o(s,f_1,f_2,\varPhi) = 1$ when $v$ is \lq\lq good\rq\rq\ (cf.\ \cite[Theorem 3.3]{G-J}).\\

Suppose $\omega\big|_{k_\infty^\times} = 1$.
Let $\Pscr^{\text{\rm GJ}} :  \Pi \times \Pi \times S(\Mat_r(\AA^\infty))_\ZZ \longrightarrow \CC$
be the multilinear functional satisfying
\begin{eqnarray}\label{eqn 5.4}
\Pscr^{\text{\rm GJ}}(f_1,f_2,\varphi^\infty) &:=& L_\infty^o(0;f_{1,v} ,f_{2,v}, \mathbf{1}_{\Mat_r(O_\infty)}) 
\cdot \prod_{v \neq \infty} L_v^o(0;{f_{1,v}} ,{f_{2,v}},\varPhi_v)
\end{eqnarray}
for factorizable $f_1,f_2 \in \Pi$ and $\varPhi^\infty \in S(\Mat_r(\AA^\infty))_\ZZ$.

On the other hand, the group homomorphism $\iota: \GL_r^2 = \GL_r \times \GL_r \rightarrow \GL_{r^2}$ (coming from left and right translations) induces a natural map
$$\Bcal^r_\AA(\ZZ) \times \Bcal^r_\AA(\ZZ) \rightarrow \Bcal^{r^2}_\AA(\ZZ).$$
For $\varPhi^\infty \in S(\Mat(\AA^\infty))_\ZZ$,
let $\eta_{\omega^{-1}}(\cdot; \varPhi^\infty)$ be
the function on $\Bcal^{r^2}_\AA(\ZZ)$ induced from the Drinfeld-Siegel units on $\Hfk^{r^2}_\AA$ associated to $\varPhi^\infty$ (cf.\ Section~\ref{sec 3.3}).
We may restrict the function $\eta_{\omega^{-1} }(\cdot; \varPhi^\infty)$ to $\Bcal^r_\AA(\ZZ) \times \Bcal^r_\AA(\ZZ)$, and  view
$\eta_{\omega^{-1}}(\cdot; \varPhi^\infty)$ as a function on $\GL_r^2 (\AA)$.
Define the following multi-linear functional $\Pcal^{\text{\rm GJ}}$ on $\Pi \times \Pi \times S(\Mat_r(\AA^\infty))_\ZZ \rightarrow \CC$ via the Petersson inner product:
$$\Pcal^{\text{\rm GJ}}(f_1,f_2,\varPhi^\infty) :=  \frac{1}{1-q_\infty^{r^2}} \cdot \iint_{\big(\AA^\times \GL_r(k)\backslash \GL_r(\AA)\big)^2}
f_1(g_1) \eta_{\omega^{-1}}(\iota(g_1,g_2);\varPhi^\infty)) \cdot  \overline{\tilde{f}_2(g_2)} dg_1 dg_2.$$
We arrive at:

\begin{cor}\label{cor 5.4}
Let $\Pi$ be an automorphic cuspidal representation of $\GL_r(\AA)$ with a unitary central character denoted by $\omega$.
Suppose $\omega\big|_{k_\infty^\times} = 1$.
Then
$$\Pcal^{\text{\rm GJ}} = L(\frac{1-r}{2},\Pi) \cdot \Pscr^{\text{\rm GJ}}.$$
\end{cor}

\begin{proof}
Given factorizable $f_1,f_2 \in \Pi$ and $\varPhi^\infty \in S(\Mat_r(\AA^\infty))$, by Proposition~\ref{prop 5.2} and Corollary~\ref{cor 3.7} one gets
\begin{eqnarray}
&&L^{\text{GJ}}(rs;f_1,f_2,\varPhi^\infty \otimes \mathbf{1}_{\Mat_r(O_\infty)}) \nonumber \\
&=& \iint_{\big(\AA^\times \GL_r(k)\backslash \GL_r(\AA)\big)^2}
f_1(g_1)  \cdot \Ecal(\iota(g_1,g_2),s; \omega^{-1},\varPhi^\infty \otimes \mathbf{1}_{\Mat_r(O_\infty)})\cdot \overline{\tilde{f}_2(g_2)} dg_1 dg_2 \nonumber \\
&=&
-\left(\iint_{\big(\AA^\times \GL_r(k)\backslash \GL_r(\AA)\big)^2}
f_1(g_1) \overline{\tilde{f}_2(g_2)} dg_1 dg_2\right) \cdot \left(\varPhi^\infty(0) \cdot \int_{k^\times \backslash \AA^{\infty,\times}} \omega(a^\infty) d^\times a^\infty\right) \nonumber \\
&& \cdot \left(\frac{1}{r^2\ln q_\infty } \cdot s^{-1} + \frac{1}{2}\right)
\nonumber \\
&& + \frac{1}{1-q_\infty^{r^2}} \cdot \iint_{\big(\AA^\times \GL_r(k)\backslash \GL_r(\AA)\big)^2}
f_1(g_1)\eta_{\omega^{-1}}(\iota(g_1,g_2);\varPhi^\infty)) \cdot  \overline{\tilde{f}_2(g_2)} dg_1 dg_2 \nonumber \\
&& + \text{O}(s) \quad \text{as $s \rightarrow 0$.} \nonumber 
\end{eqnarray}
Since $\int_{\AA^\times \GL_r(k)\backslash \GL_r(\AA)} f(g) dg = 0$ for every $f \in \Pi$ (as $\Pi$ is cuspidal), we obtain that
\begin{eqnarray}
&& L^{\text{GJ}}(0;f_1,f_2,\varPhi^\infty \otimes \mathbf{1}_{\Mat_r(O_\infty)}) \nonumber \\
&=&  \frac{1}{1-q_\infty^{r^2}} \cdot \iint_{\big(\AA^\times \GL_r(k)\backslash \GL_r(\AA)\big)^2}
f_1(g_1)\eta_{\omega^{-1}}(\iota(g_1,g_2);\varPhi^\infty)) \cdot  \overline{\tilde{f}_2(g_2)} dg_1 dg_2 \nonumber \\
&=& \Pcal^{\text{GJ}}(f_1,f_2,\varPhi^\infty). \nonumber
\end{eqnarray}
On the other hand, the equality~(\ref{eqn GJ-local-global}) says that
\begin{eqnarray}
&& L^{\text{GJ}}(0;f_1,f_2,\varPhi^\infty \otimes \mathbf{1}_{\Mat_r(O_\infty)}) \nonumber \\
&=& L(\frac{1-r}{2}, \Pi) \cdot \left(L_\infty^o(0,f_{1,\infty},f_{2,\infty},\mathbf{1}_{\Mat_r(O_\infty)}) \cdot \prod_v L_v^o(0,f_{1,v},f_{2,v},\varPhi^\infty_v)\right) \nonumber \\
&=& L(\frac{1-r}{2},\Pi) \cdot \Pscr^{\text{GJ}}(f_1,f_2,\varPhi^\infty). \nonumber
\end{eqnarray}
Therefore the result holds.
\end{proof}

\appendix

\section{CM theory of Drinfeld modules}\label{sec A}

Let $\rho$ be a Drinfeld $A$-module of rank $r$ over $\CC_\infty$. 
For $f \in \End_A(\rho/\CC_\infty) \subset \CC_\infty\{\tau\}$, let $d_f$ be the constant term of $f$.
Put $$\Ocal := \{d_f : f \in \End_A(\rho/\CC_\infty) \} \subset \CC_\infty.$$ Then $(f \mapsto d_f)$ gives a ring isomorphism from $\End_A(\rho/\CC_\infty)$ to $\Ocal$ (cf.\ \cite[Theorem 13.25]{Ros}). In particular, it is known that:
\begin{itemize}
\item Let $\Lambda_\rho\subset \CC_\infty$ be the $A$-lattice associated to $\rho$. Then $\Ocal = \{ c \in \CC_\infty: c \Lambda_\rho \subset \Lambda_\rho\}$;
\item The field $K$ of fractions of $\Ocal$ is imaginary (i.e.\ $\infty$ is not split in $K$) with $[K:k] \mid r$.
\end{itemize}
For $c \in \Ocal$, we let $\rho_c$ be the endomorphism of $\rho$ with $d_{\rho_c} = c$.

The $A$-lattice $\Lambda_\rho$ can be viewed as an $\Ocal$-module. We may say that two Drinfeld $A$-modules $\rho_1$ and $\rho_2$ of rank $r$ over $\CC_\infty$ with $\End_A(\rho_1/\CC_\infty) = \End_A(\rho_2/\CC_\infty) \cong \Ocal$ \textit{have the same genus} if 
$\Lambda_{\rho_1}\otimes_A O_v \cong \Lambda_{\rho_2}\otimes_A O_v$
 as $\Ocal\otimes_A O_v$-modules for each finite place $v$ of $k$.\\

Suppose $\rho$ is CM (i.e.\ $[K:k] = r$). Then as an $\Ocal$-module, the lattice $\Lambda_{\rho}$ is isomorphic to an ideal $\Ifk$ of $\Ocal$ with 
$$\End_\Ocal(\Ifk):= \{ \alpha \in K : \alpha \Ifk \subset \Ifk \} = \Ocal.$$
We say that $\rho$ has \textit{principal genus} if $\Lambda_\rho \otimes_A O_v \cong \Ocal \otimes_A O_v$ for each finite place $v$ of $k$, or equivalently, $\Ifk$ is an invertible ideal of $\Ocal$.

In this section, we extend the work of Hayes in \cite[Theorem 8.5]{Ha1} on the CM theory of Drinfeld modules having principal genus to the case of arbitrary genus.
More precisely, we shall prove the following theorem:

\begin{thm}\label{thm A.1}
Let $\rho$ be a CM Drinfeld $A$-module of rank $r$ over $\CC_\infty$. Identify the endomorphism ring $\End_A(\rho/\CC_\infty)$ with an $A$-order $\Ocal$ of an imaginary field $K$ with $[K:k] = r$. \\
$(1)$ The Drinfeld $A$-module $\rho$ is isomorphic (over $\CC_\infty$) to a CM Drinfeld $A$-module of rank $r$ defined over $H_\Ocal$, where $H_\Ocal$ is the ring class field of $\Ocal$.\\
$(2)$ Suppose the Drinfeld module $\rho$ is defined over $H_\Ocal$. Then $\End_A(\rho/\CC_\infty) = \End_A(\rho/H_\Ocal)$. Moreover, for an integral ideal $\Acal$ of $\Ocal$ which is invertible,
let $\Frob_\Acal \in \gal(H_\Ocal/K)$ be the Frobenius automorphism associated to $\Acal$ via the Artin map. Then
$$\Frob_\Acal(\rho) \cong \Acal * \rho.$$
Here $\Acal * \rho$ is the unique Drinfeld $A$-module satisfying
$$\rho_\Acal \cdot \rho_a = (\Acal * \rho)_a \cdot \rho_\Acal, \quad \forall a \in A,$$
where $\rho_\Acal \in H_\Ocal\{ \tau \}$ is the monic generator of the left ideal of $H_\Ocal\{ \tau \}$ generated by endomorphisms $\rho_c$ for all $c \in \Acal \subset \Ocal$.
\end{thm}

\begin{rem}\label{rem A.2}
Let $\rho$ be a Drinfeld module of rank $r$ over $\CC_\infty$ with CM by $\Ocal$.
Let $\Lambda_\rho \subset \CC_\infty$ be the $A$-lattice associated to $\rho$, which is equipped with an $\Ocal$-module structure.
Then for an invertible ideal $\Acal$ of $\Ocal$,
the $A$-lattice associated to $\Acal * \rho$ is homothetic to $\Acal^{-1} \cdot \Lambda_\rho$ (cf.\ \cite[Proposition 5.10 and the equation (5.18)]{Ha1}).
Suppose $\rho$ is defined over $H_\Ocal$ (via a fixed embedding $H_\Ocal \hookrightarrow \CC_\infty$). Then Theorem~\ref{thm A.1} (2) tells us that
the $A$-lattice associated to $\text{Frob}_\Acal(\rho)$ also lies in the homothety class of $\Acal^{-1} \cdot \Lambda_\rho$.
\end{rem}

We first recall the needed properties in the explicit class field theory over global function fields. Further details are referred to  \cite[Chapter 7]{Gos} and \cite{Ha1}.

\subsection{Explicit class field theory}\label{sec A.1}

Let $\rho^o$ be a CM Drinfeld $A$-module of rank $r$ over $\CC_\infty$ so that $\End(\rho^o)$ is identified with $O_K$, the integral closure of $A$ in the imaginary field $K$. Viewing $\rho$ as a Drinfeld $O_K$-module of rank $1$, suppose $\rho^o$ is \textit{sign-normalized} (cf.\ \cite[Theorem 7.2.15]{Gos}). Then $\rho^o$ is actually defined over $O_{H^+}$, the integral closure of $A$ in $H^+$, where $H^+$ is the \lq\lq narrow\rq\rq\ Hilbert class field of $O_K$ (cf.\ \cite[Section 7.4]{Gos}).  Given an ideal $\Afk$ of $O_K$, let $\Frob_\Afk \in \gal(H^+/K)$ be the Frobenius automorphism associated to $\Afk$. Then (cf.\ \cite[Theorem 7.4.8]{Gos}):
$$\Frob_\Afk(\rho^o) = \Afk * \rho^o.$$

For an integral ideal $\Cfk$ of $O_K$, let $$\rho^o[\Cfk] := \{ \lambda \in \overline{K} : \rho^o_a (\lambda) = 0,\ \forall a \in \Cfk\},$$
and put $H^+_{\Cfk} := H^+(\rho^o[\Cfk])$. Then (cf.\ \cite[Proposition 7.5.4 and Corollary 7.5.5]{Gos}):

\begin{thm}\label{thm A.2}
The field extension $H^+_\Cfk/ K$ is abelian with
$$\gal(H^+_\Cfk /K ) \cong \Ical_{O_K}(\Cfk)/\Pcal^+_\Cfk.$$
Here $\Ical_{O_K}(\Cfk)$ is the group generated by ideals of $O_K$ coprime to $\Cfk$, and $\Pcal^+_\Cfk$ is the subgroup generated by principal ideals $\alpha O_K$, where $\alpha \in O_K$ is positive and $\alpha \equiv 1 \bmod \Cfk$.
Moreover, for an integral ideal $\Afk$ of $O_K$ coprime to $\Cfk$, one has
$$\Frob_\Afk(\lambda) = \rho^o_\Afk(\lambda), \quad \forall \lambda \in \rho^o[\Cfk].$$
\end{thm}

\begin{rem}\label{rem A.3}
(1) Let $\infty_K$ be the unique place of $K$ lying above $\infty$ and $\FF_{\infty_K}$ the residue field at $\infty_K$. Then $H^+_\Cfk/K$ is tamely ramified at $\infty_K$ with ramification index $\#(\FF_{\infty_K})-1$ (cf.\ \cite[Proposition 7.5.8 and Corollary 7.5.9]{Gos}).\\
(2) Let $\Lambda^o = \Lambda_{\rho^o}\subset \CC_\infty$ be the $A$-lattice corresponding to $\rho^o$.
Then we may write $\rho^o[\Cfk]$ as
$$\rho^o[\Cfk] = \left\{ \exp_{\Lambda^o}(\alpha) : \alpha \in \frac{\Cfk^{-1} \Lambda^o}{\Lambda^o} \right\}.$$
Given an integral ideal $\Afk$ of $O_K$ coprime to $\Cfk$, recall that $\rho^o_\Afk \in H^+\{\tau\}$ is the monic generator of the left ideal generated by $\rho^o_c$ for all $c \in \Afk$.
let $d_\Afk$ be the constant term of $\rho^o_\Afk$. Then we have
$$d_\Afk = \prod_{0 \neq \lambda \in \rho^o[\Afk]} \lambda \quad \quad  \text{ and } \quad \quad\Lambda_{\Frob_\Afk(\rho^o)} = \Lambda_{\Afk * \rho^o} = d_\Afk \cdot \Afk^{-1} \Lambda^o.$$
Theorem~\ref{thm A.2} implies that
\begin{eqnarray}\label{eqn A.1}
\quad \quad\Frob_\Afk\big(\exp_{\Lambda^o}(\alpha)\big) = \rho_\Afk^o\big(\exp_{\Lambda^o}(\alpha)\big) = \exp_{d_\Afk \Afk^{-1} \Lambda^o} ( d_\Afk \alpha), 
\quad \quad  \forall \alpha \in \frac{\Cfk^{-1} \Lambda^o}{\Lambda^o}.
\end{eqnarray}
\end{rem}

\subsubsection{Frobenius action on Drinfeld $\Ocal$-modules}\label{sec A.2}

Let $\Lambda \subset \CC_\infty$ be an $A$-lattice of rank $r$ so that
$\Ocal :=  \{ c \in \CC_\infty : c\Lambda \subset \Lambda\}$
is an $A$-order of an imaginary field $K$ with $[K:k] = r$.
Let $\Cfk \lhd O_K$ be the conductor of $\Ocal$.
Take $\Lambda^o := \Cfk \cdot \Lambda$. Then
$$O_K \cdot \Lambda = \Cfk^{-1} \cdot\Lambda^o \supset \Lambda \supset \Lambda^o,$$
and the Drinfeld $A$-module $\rho^o$ corresponding to $\Lambda^o$ is CM by $O_K$. \\

Assume that $\rho^o$ is sign-normalized (thus defined over $H^+$).
Take 
$$u(x) = u_{\Lambda/\Lambda^o}(x) := \prod_{\alpha \in \frac{\Lambda}{\Lambda^o}} \big(x - \exp_{\Lambda^o}(\alpha)\big) \quad \text{and} \quad
d_u:= \prod_{0 \neq \alpha \in \frac{\Lambda}{\Lambda^o}}\exp_{\Lambda^o}(\alpha).
$$
Then $u$ corresponds to a twisted polynomial $u(\tau) \in H^+_\Cfk\{ \tau\}$ with constant term $d_u$.
Moreover, let $\rho^{d_u \Lambda}$ be the Drinfeld $A$-module corresponding to the $A$-lattice $d_u \Lambda$.
Then
$$u \cdot  \rho^o = \rho^{d_u\Lambda} \cdot u,$$
which tells us that $\rho^{d_u\Lambda}$ is defined over $H^+_\Cfk$.
In fact, we have:

\begin{prop}\label{prop A.4}
The Drinfeld $A$-module $\rho = \rho^{d_u \Lambda}$ is defined over $H^+_{\Ocal}$, the \lq\lq narrow\rq\rq\ ring class field of $\Ocal$. Moreover, given an invertible ideal $\Acal$ of $\Ocal$, let $\Frob_\Acal \in \gal(H^+_{\Ocal}/K)$ be the Frobenius automorphism associated to $\Acal$. Then
$$\Frob_\Acal(\rho) = \Acal * \rho.$$
\end{prop}

\begin{proof}
Let $\Ical_\Ocal$ (resp.\ $\Ical_\Ocal(\Cfk)$) be the group generated by invertible ideals of $\Ocal$ (resp.\ coprime to $\Cfk$), and $\Pcal_\Ocal^+$ (resp.\ $\Pcal_\Ocal^+(\Cfk)$) is the subgroup generated by principal ideals $\alpha \Ocal$, where $\alpha$ is positive (resp.\ and coprime to $\Cfk$). 
Then 
$$
\gal(H^+_\Ocal/K) \cong \Pic^+(\Ocal) := \Ical_\Ocal/ \Pcal_\Ocal^+\cong  \Ical_\Ocal(\Cfk) / \Pcal_\Ocal^+(\Cfk),
$$
and we have the following commutative diagram:
$$
\begin{tikzcd}
 \gal(H_\Cfk^+/K)  \arrow[swap,d, "\text{restriction}"] \arrow[r, leftrightarrow,"\sim"] &  \Ical_{O_K}(\Cfk)/\Pcal_\Cfk^+ \arrow[d]  \\
 \gal(H_\Ocal^+/K) \arrow[r, leftrightarrow, "\sim"] & \Ical_\Ocal(\Cfk)/\Pcal_\Ocal^+(\Cfk).
\end{tikzcd}
$$
Here the vertical map on the right hand side is induced from $\Afk \mapsto \Afk \cap \Ocal$ for every integral ideal $\Afk$ of $O_K$ coprime to $\Cfk$.
Thus to prove that $\rho$ is defined in $H_\Ocal^+$, it suffices to show that $\Frob_\Afk(\rho) = \rho$ for every integral ideal $\Afk$ of $O_K$ coprime to $\Cfk$ and $\Afk \cap \Ocal \in \Pcal_\Ocal^+(\Cfk)$.\\

Let $\Afk$ be an integral ideal of $O_K$ coprime to $\Cfk$
and $\Frob_\Afk \in \gal(H_{\Cfk}^+/K)$ be the Frobenius element corresponding to $\Afk$.
We have
$$\Frob_\Afk(u) \cdot \Frob_\Afk(\rho^o) = \Frob_\Afk(\rho) \cdot \Frob_\Afk(u).$$
Note that $O_K\cdot\Lambda$ is a projective $O_K$-module of rank $1$.
Since $\Afk$ and $\Cfk$ are relatively prime, one gets $O_K \cap \Afk^{-1} \Cfk = \Cfk$ and
$$\Lambda^o \subseteq \Lambda \cap \Afk^{-1}\Lambda^o \subseteq (O_K\Lambda) \cap \Afk^{-1}\Cfk \cdot (O_K\Lambda) = (O_K \cap \Afk^{-1}\Cfk) \cdot (O_K\Lambda) = \Cfk \cdot (O_K\Lambda) = \Lambda^o.$$
Thus $\Lambda \cap \Afk^{-1}\Lambda^o = \Lambda^o$ and we have the following isomorphism
$$\frac{\Lambda}{\Lambda^o} \cong \frac{\Lambda + \Afk^{-1}\Lambda^o}{\Afk^{-1}\Lambda^o} \quad \text{ by sending $\alpha + \Lambda^o$ to $\alpha + \Afk^{-1}\Lambda^o$ for all $\alpha \in \Lambda$.}$$
The equality~(\ref{eqn A.1}) then implies
$$ \Frob_\Afk(u) (x)
= \prod_{\alpha' \in \frac{d_\Afk (\Lambda + \Afk^{-1} \Lambda^o)}{d_\Afk \Afk^{-1} \Lambda^o}}
\Big(x-\exp_{d_\Afk \Afk^{-1} \Lambda^o}(\alpha')\Big).$$
Since $\Frob_\Afk(\rho^o) = \Afk * \rho^o$, which corresponds to the lattice $d_\Afk \Afk^{-1} \Lambda^o$,
we obtain that $\Frob_\Afk(\rho)$ actually corresponds to the lattice 
$$(d_{\Frob_\Afk(u)} d_\Afk)\cdot  (\Lambda + \Afk^{-1} \Lambda^o) = (d_{\Frob_\Afk(u)} d_\Afk) \cdot \Acal^{-1} \Lambda,$$
where $d_{\Frob_\Afk(u)}$ is the constant term of ${\Frob_\Afk(u)}$ and $\Acal := \Afk \cap \Ocal$ is an invertible ideal of $\Ocal$ coprime to $\Cfk$. 
On the other hand, the lattice corresponding to $\Acal * \rho$ is $(d_\Acal \cdot d_u) \cdot\Acal^{-1} \Lambda$,
where $d_\Acal$ is the constant term of $\rho_\Acal$.
It is straightforward that 
$$d_\Acal \cdot d_u = d_{\Frob_\Afk(u)} d_\Afk.$$
Therefore 
\begin{eqnarray}\label{eqn Reciprocity}
\Frob_\Afk(\rho) &=& \Acal * \rho.
\end{eqnarray}

Suppose $\Afk \cap \Ocal = \Acal = \alpha \cdot \Ocal$, where $\alpha$ is positive.
Then $\rho_{\alpha \Ocal} = \rho_\alpha$ and 
$$\Frob_\Afk(\rho) = (\alpha \Ocal) * \rho = \rho.$$
Therefore $\rho$ is defined over $H_\Ocal^+$.\\

Moreover, for an invertible ideal $\Acal$ of $\Ocal$ coprime to $\Cfk$, we put $\Afk := \Acal \cdot O_K$,
and let $\Frob_\Acal \in \gal(H_\Ocal^+/K)$ and $\Frob_\Afk \in \gal(H_\Cfk^+/K)$ be the Frobenius elements corresponding to $\Acal$ and $\Afk$, respectively. 
Then $\Frob_\Acal = \Frob_\Afk \big|_{H_\Ocal^+}$, and the equation~(\ref{eqn Reciprocity}) says
$$\Frob_\Acal(\rho) = \Frob_\Afk(\rho) = \Acal * \rho,$$ 
which completes the proof.
\end{proof}

\subsection{Proof of Theorem~\ref{thm A.1}}

Let $\rho$ be a CM Drinfeld $A$-module of rank $r$.
Identifying $\End(\rho)$ with an $A$-order $\Ocal$ of an imaginary field $K$, suppose $\Ocal$ has conductor $\Cfk$.
Let $H_\rho$ be the field of invariants of $\rho$ (cf.\ \cite[Theorem 6.6]{Ha1}). Then it is known that $H_\rho$ is contained in $K_\infty$ (cf.\ \cite[Proposition 6.2 and 6.4]{Ha1}).
Therefore Proposition~\ref{prop A.4} implies that 
$$H_\rho \subset H^+_\Ocal \cap K_\infty = H_\Ocal.$$ 
Let $\Lambda_\rho \subset \CC_\infty$ be the $A$-lattice corresponding to $\rho$.
Viewing $\Lambda_\rho$ as an $\Ocal$-module, it is the fact that for two invertible ideals $\Acal, \Bcal$ of $\Ocal$, we have
$$\Acal \cdot \Lambda_\rho \cong \Bcal \cdot \Lambda_\rho \quad \text{ if and only if } \quad \Acal^{-1} \Bcal \text{ is a principal ideal of $\Ocal$.} $$
The explicit description of Frobenius actions in Proposition~\ref{prop A.4} then assures that 
$$[H_\rho: K] = \#\Pic(\Ocal) = [H_\Ocal:K].$$ Therefore $H_\rho = H_\Ocal$ and the Theorem holds.
\hfill  $\Box$

\end{document}